\documentclass[11pt,reqno]{amsart}

\usepackage{geometry}
\usepackage[utf8]{inputenc}
\usepackage{amstext,latexsym,amsbsy,amsmath,amssymb,amsthm,mathtools,relsize,geometry,enumerate}
\usepackage{hyperref}
\hypersetup{
	colorlinks   = true, %Colours links instead of ugly boxes
	urlcolor     = blue, %Colour for external hyperlinks
	linkcolor    = red, %Colour of internal links
	citecolor   = red %Colour of citations
}
\usepackage{mathtools, enumerate,enumitem}

\usepackage{mathrsfs}

\usepackage{comment}
\usepackage{graphicx}
\usepackage{epstopdf}
\setkeys{Gin}{width=\linewidth,totalheight=\textheight,keepaspectratio}
\graphicspath{{./images/}}

\usepackage[normalem]{ulem}

\usepackage{multicol}
\usepackage{multirow}
\usepackage{booktabs}
\usepackage{mathrsfs}
\usepackage{color}

\numberwithin{equation}{section}

\newtheorem{remark}{Remark}[section]
\newtheorem{theorem}{Theorem}[section]
\newtheorem{corollary}{Corollary}[section]
\newtheorem{lemma}{{Lemma}}[section]

\allowdisplaybreaks[4]

\newcommand{\cP}{\mathcal{P}}

\newcommand{\mP}{\mathbb{P}}

\newcommand{\mE}{\mathbb{E}}

\newcommand{\Ome}{\Omega}
\newcommand{\p}{\partial}
\newcommand{\nab}{\nabla}

\newcommand{\vH}{{\bf H}}

\newcommand{\vv}{{\bf v}}

\newcommand{\e}{\pmb{\varepsilon}}

\newcommand{\vvarphi}{\pmb{\varphi}}
\newcommand{\vsigma}{\pmb{\sigma}}
\newcommand{\vtheta}{\pmb{\theta}}

\begin{document}
	
	\title{A splitting mixed finite element method for a stochastic Keller-Segel system with multiplicative noise}

	\author{
		Thoa Thieu$^1$ and 
		Liet Vo$^2$ 
	}

	\thanks{$^1$ School of Mathematical and Statistical Sciences, The University of Texas Rio Grande Valley, Edinburg, TX 78539, U.S.A.  ({\tt thoa.thieu@utrgv.edu}).}
	
	\thanks{$^2$ School of Mathematical and Statistical Sciences, The University of Texas Rio Grande Valley, Edinburg, TX 78539, U.S.A.  ({\tt liet.vo@utrgv.edu}). This author was partially supported by the NSF grant DMS-2530211.}

\begin{abstract}
	In this paper, we propose and analyze a splitting mixed finite element method for a stochastic Keller--Segel system with logistic growth driven by multiplicative noise. By introducing an auxiliary variable representing the chemical gradient together with a time-lagged splitting strategy, the proposed method decouples the original coupled system into a sequence of simpler subproblems. Consequently, it eliminates the Ladyzhenskaya--Babu\v{s}ka--Brezzi stability constraint, permits the use of continuous piecewise linear finite element spaces for all unknowns, and avoids solving a fully coupled nonlinear system at each time step, thereby significantly reducing the computational cost. Combined with an implicit Euler time discretization, the proposed approach yields a fully discrete numerical scheme for the stochastic Keller--Segel system. Using a localization technique together with suitable stochastic stability arguments, we establish optimal strong error estimates for the fully discrete approximations and prove convergence in probability with explicit convergence rates. Numerical experiments verify the theoretical convergence rates and demonstrate that the proposed method successfully captures the global boundedness induced by the logistic growth term as well as the influence of multiplicative noise on chemotactic aggregation.
\end{abstract}

	\maketitle

	
	\medskip
	

	\section{Introduction}\label{sec-1}
	In this paper, we study the following stochastic Keller-Segel system with logistic growth and multiplicative noise
	\begin{subequations}\label{eq1.1}
		\begin{align} \label{eq1.1a}
			du  &=  \bigl[\Delta u -\chi\nab\cdot (u \nab v)  \bigr]\, dt + g(u)\, dt+  B(u) dW(t) \quad&&\mbox{a.s. in}\, (0,T)\times D,\\ 
			\label{eq1.1b}	\Delta v&=  v-u\qquad&&\mbox{a.s. in}\, (0,T)\times D,    \\
			u(0) &= u_0, \quad v(0) = v_0 \qquad&&\mbox{a.s. in}\, D,
		%	\frac{\partial u}{\partial n} &=0, \quad v=0\qquad&&\mbox{a.s. on}\, (0,T)\times\p D,
		\end{align}
	\end{subequations}
	where $D$ is a bounded set in $ \mathbb{R}^2$, and $T > 0$. $u$ and $v$ denote, respectively, the cell density and the concentration of an attractive chemical signal, which satisfy the following Neumann boundary conditions
	\begin{align*}
		\frac{\partial u}{\partial \mathbf{n}} &=0, \quad \frac{\partial v}{\partial \mathbf{n}}=0\qquad&&\mbox{a.s. on}\, (0,T)\times\p D,
	\end{align*}
	where the symbol $\frac{\partial}{\partial \mathbf{n}}$ denotes the derivative with respect to the outer normal of $\partial D$.
	
	The constant $\chi>0$ denotes the chemotactic sensitivity, and
	$\{W(t):t\ge0\}$ is a real-valued Wiener process. We consider the
	prototype logistic reaction term $g(u)=u-u^3,$ which is commonly used in mathematical models of population dynamics and chemotaxis; see \cite{chen2025well}.
	
Chemotaxis refers to the directed movement of cells or microorganisms in response to chemical gradients in their surrounding environment. It plays a fundamental role in numerous biological and biomedical processes, including bacterial aggregation, embryonic development, wound healing, angiogenesis, tumor invasion, and immune cell migration. One of the most influential mathematical descriptions of chemotactic movement is the Keller--Segel model, originally proposed by Keller and Segel \cite{keller1970initiation,keller1971model}, in which the evolution of the cell density is coupled with the concentration of a chemoattractant. Owing to its broad range of biological applications and rich mathematical structure, the Keller--Segel system has become a prototype model for studying pattern formation, aggregation phenomena, and finite-time blow-up in chemotaxis; see, for example, \cite{hillen2009user,tindall2008overview,osaki2001finite} and the references therein.

The deterministic Keller--Segel systems with logistic sources have received considerable attention due to their important biological relevance and remarkable mathematical properties. In many biological populations, the logistic source models the combined effects of cell proliferation and density-dependent inhibition arising from competition for nutrients, space, or other environmental resources. From a mathematical viewpoint, the logistic reaction provides a natural damping mechanism that counteracts the excessive aggregation induced by chemotaxis, thereby preventing finite-time blow-up in the classical Keller--Segel model. Consequently, the incorporation of logistic growth not only yields a more realistic description of population dynamics but also fundamentally changes the qualitative behavior of solutions, leading to global existence, boundedness, long-time stabilization, and the emergence of nontrivial spatial patterns under suitable conditions. Furthermore, the interaction among diffusion, chemotaxis, and logistic growth gives rise to a variety of aggregation patterns observed in biological systems. We refer the reader to \cite{mimura1996aggregating,tello2007chemotaxis,lankeit2017generalized,cao2014boundedness} and the references therein.

In many biological systems, however, chemotactic movement is inevitably affected by random environmental fluctuations, uncertain external stimuli, and intrinsic cellular variability. Such random effects cannot be adequately described by deterministic models and naturally lead to stochastic Keller--Segel equations driven by transport or multiplicative noise. These stochastic models provide a more realistic description of chemotactic phenomena and have attracted increasing attention in recent years. Existing analytical studies have focused primarily on the well-posedness and qualitative properties of stochastic Keller--Segel equations. For stochastic Keller--Segel systems with transport noise and multiplicative noise, important solution properties such as positivity, mass preservation, global well-posedness, and finite-time blow-up have been investigated in \cite{misiats2022global,huang2021microscopic,mukherjee2026uniqueness,mayorcas2023blow}. More recently, Chen et al.~\cite{chen2025well} established the existence and uniqueness of global mild solutions for a stochastic Keller--Segel system with logistic growth driven by multiplicative noise. Similar to the deterministic setting, the logistic reaction term prevents excessive chemotactic aggregation and guarantees the global existence of solutions.

Despite the significant progress in the analytical theory, the numerical analysis of stochastic Keller--Segel equations remains in its infancy. To the best of our knowledge, the only rigorous numerical analysis available is the recent work \cite{vo2025analysis}, in which a Crank--Nicolson splitting mixed finite element method was developed for the stochastic Keller--Segel equation driven by transport noise in the Stratonovich sense. In contrast, for stochastic Keller--Segel equations with multiplicative noise, the development and analysis of numerical methods remain completely open.

The numerical approximation of stochastic Keller--Segel equations presents several substantial challenges. For general stochastic Keller--Segel models, an effective numerical method should faithfully capture the fundamental qualitative properties of the underlying system, including positivity, mass preservation, and, whenever applicable, the finite-time blow-up behavior of solutions. In the presence of a logistic source, although the continuous problem admits global solutions and the logistic damping prevents finite-time blow-up, it remains essential for the numerical method to accurately reproduce this global boundedness and the long-time dynamics of the solution.

From the analytical viewpoint, the main difficulty stems from the nonlinear chemotactic flux $\nabla\!\cdot(u\nabla v)$, which contains a nonlinear second-order derivative that is strongly coupled with the stochastic forcing. This interaction considerably complicates the derivation of discrete stability estimates, since the stochastic perturbation acts directly on the highest-order nonlinear term. Consequently, standard finite element discretizations typically require higher-order finite element spaces to approximate the chemical gradient accurately, leading to increased computational complexity. Moreover, the simultaneous treatment of the nonlinear chemotactic coupling and the stochastic perturbation presents significant analytical challenges in establishing rigorous stability, convergence, and error estimates.

The above challenges motivate the present work, whose objective is to develop an efficient and rigorously analyzed finite element method for stochastic Keller--Segel equations with multiplicative noise. To the best of our knowledge, this is the first work establishing rigorous error estimates for finite element approximations of stochastic Keller--Segel equations driven by multiplicative noise. Since the stochastic Keller--Segel system with logistic growth admits global solutions, it is also of particular interest to investigate whether the proposed numerical method is capable of reproducing this important qualitative behavior. Through a series of numerical experiments, we demonstrate that the proposed method successfully captures the long-time boundedness of the solution, the regularizing effect of the logistic growth term, and the influence of multiplicative noise on chemotactic aggregation.

Inspired by the mixed formulations in \cite{zhang2016characteristic,vo2025analysis}, we first reformulate the original stochastic Keller--Segel system by introducing the auxiliary variable $\vsigma=\nabla v$. This reformulation converts the nonlinear second-order chemotactic flux into a first-order coupled system and provides a direct approximation of the chemical gradient. A standard mixed finite element discretization of this formulation, however, requires the satisfaction of the Ladyzhenskaya--Babu\v{s}ka--Brezzi (inf--sup) condition to guarantee the stability and accuracy of the numerical approximation. Consequently, the choice of finite element spaces is considerably restricted and, in practice, often requires higher-order polynomial finite elements, leading to increased computational complexity.

To overcome these difficulties, we adopt the time-lagged strategy proposed in \cite{vo2025analysis,zhang2016characteristic} and develop a splitting mixed finite element method. The proposed splitting procedure decouples the original coupled system into a sequence of subproblems that are solved successively at each time step, thereby eliminating the need to solve a fully coupled mixed system and allowing the use of continuous piecewise linear finite element spaces for all unknown variables. This substantially reduces the computational complexity while retaining a direct approximation of the chemical gradient. Such a decoupled strategy is particularly attractive for large-scale stochastic simulations, where a large number of independent sample paths must be computed in Monte Carlo approximations of statistical quantities. Since each realization requires repeatedly solving the underlying deterministic system, reducing the computational cost of a single time step leads to a significant improvement in the overall efficiency of the stochastic simulation.

To analyze the proposed method, we employ a localization technique \cite{carelli2012rates} together with suitable stochastic stability arguments to control the strong interaction between the nonlinear chemotactic flux and the multiplicative stochastic forcing. Based on these ingredients, we establish rigorous strong error estimates for the proposed splitting mixed finite element method and derive convergence in probability with explicit convergence rates under appropriate regularity assumptions.

The remainder of the paper is organized as follows. In Section~\ref{sec2}, we introduce the notation, assumptions, and preliminary results used throughout the paper. We then derive the mixed variational formulation of the stochastic Keller--Segel system and establish several stability and regularity estimates for the exact solution. In Section~\ref{sec4}, we present the fully discrete splitting mixed finite element method, prove its stability, and establish rigorous strong error estimates together with convergence in probability. Section~\ref{sec5} is devoted to numerical experiments, where we verify the theoretical convergence rates and investigate the qualitative behavior of the stochastic Keller--Segel system with logistic growth, including the long-time boundedness of the numerical solutions and the influence of multiplicative noise on chemotactic aggregation. Finally, Section~\ref{sec6} concludes the paper with several remarks and possible directions for future research.

	\section{Preliminaries}\label{sec2}
	\subsection{Notations}
Standard notation for function spaces will be adopted throughout this paper.
Let $(\cdot,\cdot)$ denote the standard $L^2(D)$ inner product with induced norm $\|\cdot\|_{L^2}$.
Throughout the paper, $C$ denotes a generic positive constant independent of the mesh parameters $h$ and $k$, whose value may change from line to line.

Let $(\Omega,\mathcal F,\{\mathcal F_t\}_{t\ge0},\mathbb P)$ be a filtered probability space satisfying the usual conditions, where $\mathbb P$ is the probability measure, $\mathcal F$ is the $\sigma$-algebra, and $\{\mathcal F_t\}_{t\ge0}\subset\mathcal F$ is a filtration. For a random variable $v$ defined on $(\Omega,\mathcal F,\{\mathcal F_t\},\mathbb P)$, we denote its expectation by $\mathbb E[v]$.

For a Banach space $X$ with norm $\|\cdot\|_X$ and $1\le p<\infty$, we define the Bochner space
\[
L^p(\Omega;X)
:=
\left\{
v:\Omega\to X:
\mathbb E\bigl[\|v\|_X^p\bigr]<\infty
\right\},
\]
equipped with the norm
\[
\|v\|_{L^p(\Omega;X)}
:=
\bigl(
\mathbb E[\|v\|_X^p]
\bigr)^{1/p}.
\]

We also make use of the Ladyzhenskaya inequality in two space dimensions,
\[
\|u\|_{L^4(D)}
\le
C_L
\|u\|_{L^2(D)}^{1/2}
\|u\|_{H^1(D)}^{1/2},
\]
where $C_L>0$ depends only on the domain $D$.

Next, we denote
\[ \vH_{\vsigma}^1 := \left\{\vv \in \vH^1:\, \vv\cdot\mathbf{n} = 0\,\mbox{ on \,}\partial D\right\},
\]
and we shall use the following equivalent norm (see \cite[Corollary 3.5]{amrouche2013lp} and \cite[Section 2.1]{duarte2021numerical}):
\begin{align}
	\|\vsigma\|_{H^1}^2
	=
	\|\vsigma\|_{L^2}^2
	+
	\|\nabla\cdot\vsigma\|_{L^2}^2
	+
	\|\operatorname{rot}\vsigma\|_{L^2}^2,
	\qquad
	\forall\,\vsigma\in\mathbf H_{\vsigma}^1(D).
\end{align}

For $1<p<\infty$, let $A:=A_p$ denote the realization of the Neumann Laplacian $-\Delta$ on $L^p(D)$ with domain
\[
D(A)
=
\left\{
\phi\in W^{2,p}(D):
\frac{\partial\phi}{\partial \mathbf{n}}=0
\quad\text{on }\partial D
\right\}.
\]
It is well known that the shifted operator $A+I$ is sectorial and therefore admits fractional powers $(A+I)^\beta$ for every $\beta\ge0$. Furthermore, if $\beta>\frac1p$, then
\[
D\bigl((A+I)^\beta\bigr)
\hookrightarrow
C(\overline D).
\]

Let $\{e^{-tA}\}_{t\ge0}$ denote the analytic semigroup generated by $-A$.

Next, consider the elliptic Neumann problem
\[
-\Delta w+w=f
\quad\text{in }D,
\qquad
\frac{\partial w}{\partial \mathbf{n}}=0
\quad\text{on }\partial D,
\]
where $f\in L^p(D)$. The unique weak solution admits the Green's function representation
\[
w(x)
=
\int_D
G(x,y)f(y)\,dy,
\qquad
x\in D,
\]
where $G(\cdot,\cdot)$ denotes the Green's function associated with the operator $-\Delta+I$ under homogeneous Neumann boundary conditions.

By the elliptic regularity estimate \cite{evans2022partial},
\begin{align}\label{elliptic_estimate}
	\|w\|_{W^{2,r}(D)}
	\le
	C
	\|f\|_{L^r(D)},
	\qquad
	2\leq r<\infty,
\end{align}
and the Sobolev embedding
\[
W^{1,r}(D)
\hookrightarrow
C(\overline D),
\]
we further obtain
\begin{align}\label{sobolev_embedding}
	\|w\|_{L^\infty(D)}
	+
	\|\nabla w\|_{L^\infty(D)}
	\le
	C
	\|f\|_{L^r(D)}.
\end{align}
	
		\subsection{Assumptions}
	
	Concerning the stochastic forcing, in this paper, we will make the following assumptions on $B$.
	\begin{enumerate}[label=(\Alph*)]
		
		\item For $m = 0 ,1$, suppose that $B: H^{m}(D) \rightarrow H^m(D)$. Moreover, there exists a constant $C_{B}>0$ such that
		\begin{align}\label{Assump_Lipschitz}
			\|B(u) - B(v)\|_{H^m} \leq C_{B}\|u-v\|_{H^m}.
		\end{align}
		
		\item There exists a constant $c_B>0$ such that
	\begin{align}\label{Assump_LinearGrowth}
			\|B(u)\|_{H^m} \leq c_B (1 + \|u\|_{H^m}),
	\end{align}
		
	\end{enumerate}
	
Under the above assumptions of $B$, we recall the following mild solution for \eqref{eq1.1}, which has been established in \cite[Theorem 2.1]{chen2025well}. Suppose $u_0 \in C(\overline{D})$, then there exists a unique mild solution $(u,v) \in C(\bar{D}) \times W^{1,\infty}(D)$ such that $\mP$-a.s.
	\begin{align}\label{mild_sol}
		u(t)
		&=
		e^{-tA}u_0
		-\chi
		\int_0^t
		e^{-(t-s)A}
		\nabla\!\cdot\bigl(u(s)\nabla v(s)\bigr)\,ds
		\nonumber\\
		&\quad
		+
		\int_0^t
		e^{-(t-s)A}
		g(u(s))\,ds
		+
		\int_0^t
		e^{-(t-s)A}
		B(u(s))\,dW(s),
		\qquad t\in[0,T],
	\end{align}
	where
	\begin{equation}
		v(t)=G*u(t), \qquad t\in[0,T],
		\label{eq:mild_v}
	\end{equation}
	and $G$ denotes the Green's function associated with the elliptic operator
	$-\Delta+I$ subject to homogeneous Neumann boundary conditions.
	
	In addition, through the paper, we assume that if $u_0 \in L^2(\Ome; L^2(D))$, then there exists a unique weak solution $(u,v) \in L^2(\Ome; C(0,T; L^2(D))\cap L^2(0,T;H^1(D)))$ satisfying $\mP$-a.s.
		\begin{subequations}\label{eq_2.1}
			\begin{align}\label{eq_2.1a}
				(u(t),\phi) &= (u_0,\phi) - \int_0^t (\nab u(s), \nab \phi)\, ds \\\nonumber
				&\qquad+ \chi \int_0^t (u\nab v, \nab \phi)\, ds + \int_{0}^t(g(u),\phi)\, ds+  \int_0^t (B(u),\phi) dW(s),\\
				\label{eq_2.1b}		(\nab v(t),& \nab\psi) + (v(t), \psi) = (u(t),\psi),
			\end{align}
		\end{subequations}
		for all $\phi \in H^1(D)$, and $\psi \in H^1(D)$.
\begin{remark}
	In this paper, we focus on the development and analysis of numerical methods for \eqref{eq1.1}. Although we do not establish the existence and uniqueness of variational (weak) solutions, they can be obtained from the corresponding well-posedness theory for mild solutions. Indeed, \cite[Theorem~2.1]{chen2025well} establishes the existence and uniqueness of a mild solution to \eqref{eq1.1}. By combining the regularity properties of the mild solution with the assumption that $B(v)\in L^2(D)$ for every $v\in L^2(D)$ and adapting the argument of \cite[Theorem~6.5]{da2014stochastic}, one can show that the mild solution is also a variational (weak) solution satisfying \eqref{eq_2.1}. Throughout the error analysis, we further assume that the weak solution satisfies the appropriate spatial and temporal regularity assumptions required to establish the optimal convergence rates of the proposed numerical methods.
\end{remark}

\begin{remark}
	For simplicity of presentation, we restrict our attention to a real-valued Wiener process. The proposed splitting mixed finite element method and the corresponding error analysis can be extended to stochastic Keller--Segel systems driven by more general $Q$-Wiener or cylindrical Wiener processes through the Karhunen--Lo\`eve expansion; see, for example, \cite[Theorem~10.7]{lord2014introduction}. The extension requires only routine modifications based on the Hilbert-space formulation of stochastic integration, and the overall proof strategy remains unchanged.
\end{remark}

	Next, we establish the following high-moment stability estimates for the variational solution $u$.

	{	\begin{lemma}\label{Stability_PDE} Let $(u,v)$ be the variational solution of \eqref{eq_2.1}. Then, the following estimates are satisfied:
			\begin{enumerate}
				\item[(a)] Let $p\ge 2$. If 
				$u_0\in L^p(\Omega;L^2(D))$ and suppose that $B$ satisfies the assumptions \eqref{Assump_LinearGrowth} with $m =0$, then there exists a constant
				\[
				C_{2,p}=C\bigl(p,T,\chi,D,c_B,\mathbb E[\|u_0\|_{L^2}^{p}]\bigr)>0
				\]
				such that the weak solution $u$ satisfies the following high-moment estimate:
				\begin{align}\label{eq2.5}
					\mathbb E\Bigg[
					\sup_{0\le s\le T}\|u(s)\|_{L^2}^{p}
					&+
					\int_0^T
					\|u(s)\|_{L^2}^{p-2}\|\nabla u(s)\|_{L^2}^{2}\,ds
					+
					\int_0^T
					\|u(s)\|_{L^2}^{p-2}\|u(s)\|_{L^4}^{4}\,ds
					\Bigg]
					\le C_{2,p} .
				\end{align}

\item[(b)] Let $u_0\in L^4(\Omega;L^4(D))$ and suppose that $B$ 
satisfies \eqref{Assump_LinearGrowth} with $m = 0$ and $m =1$.
Then there exists a constant
\[
C_{3}
=
C\left(
T,\chi,D,c_B,
\mathbb E\left[\|u_0\|_{L^4}^{4}\right]
\right)>0
\]
such that the weak solution $u$ satisfies the estimate
\begin{align}\label{eq2.7}
	\mathbb E\Bigg[
	\sup_{0\le s\le T}\|u(s)\|_{L^4}^{4}
	&+
	\int_0^T
	\|u(s)\nabla u(s)\|_{L^2}^{2}\,ds
	+
	\int_0^T
	\|u(s)\|_{L^6}^{6}\,ds
	\Bigg]
	\le C_{3}.
\end{align}
			\end{enumerate}
			
		\end{lemma}
	}
	
	\begin{proof} 
		(a) To prove \eqref{eq2.5}, we apply the It\^o formula to $\Phi(u(t)) = \|u(t)\|^p_{L^2}$ to yield
		\begin{align}
			\|u(t)\|_{L^2}^{p}
			&=
			\|u_0\|_{L^2}^{p}
			+p\int_0^t
			\|u(s)\|_{L^2}^{p-2}
			\bigl(u(s),\Delta u(s)
			-\chi\nabla\cdot(u(s)\nabla v(s))
			+g(u(s))\bigr)\,ds
			\nonumber\\
			&\quad
			+p\int_0^t
			\|u(s)\|_{L^2}^{p-2}
			\bigl(u(s),B(u(s))\bigr)\,dW(s)
			+\frac{p}{2}\int_0^t
			\|u(s)\|_{L^2}^{p-2}
			\|B(u(s))\|_{L^2}^{2}\,ds
			\nonumber\\
			&\quad
			+\frac{p(p-2)}{2}\int_0^t
			\|u(s)\|_{L^2}^{p-4}
			\bigl(u(s),B(u(s))\bigr)^{2}\,ds,
		\end{align}
		which, together with integration by parts, implies
		\begin{align}
			\|u(t)\|_{L^2}^{p}
			&+
			p\int_0^t
			\|u(s)\|_{L^2}^{p-2}\|\nabla u(s)\|_{L^2}^{2}\,ds
			+
			p\int_0^t
			\|u(s)\|_{L^2}^{p-2}\|u(s)\|_{L^4}^{4}\,ds
			\\\nonumber
			&=
			\|u_0\|_{L^2}^{p}
			+
			p\int_0^t\|u(s)\|_{L^2}^{p}\,ds
			+p\chi\int_0^t
			\|u(s)\|_{L^2}^{p-2}
			\bigl(u(s)\nabla v(s),\nab u(s)\bigr)\,ds
		\\\nonumber
			&\quad
			+p\int_0^t
			\|u(s)\|_{L^2}^{p-2}
			\bigl(u(s),B(u(s))\bigr)\,dW(s)
		\\\nonumber
			&\quad
			+\frac{p}{2}\int_0^t
			\|u(s)\|_{L^2}^{p-2}
			\|B(u(s))\|_{L^2}^{2}\,ds
			+\frac{p(p-2)}{2}\int_0^t
			\|u(s)\|_{L^2}^{p-4}
			\bigl(u(s),B(u(s))\bigr)^{2}\,ds\\\nonumber
			&:=
			\|u_0\|_{L^2}^{p}
			+
			p\int_0^t\|u(s)\|_{L^2}^{p}\,ds
			\\\nonumber
			&\quad
			+p\int_0^t
			\|u(s)\|_{L^2}^{p-2}
			\bigl(u(s),B(u(s))\bigr)\,dW(s)
			\\\nonumber
			&\quad
			+ I_1 + I_2 + I_3.
		\end{align}
		
		We estimate $I_1, I_2$ and $I_3$ as follows.
		First, using integration by parts and then using the Cauchy--Schwarz
		inequality, we obtain
		\begin{align*}
			I_1
			&=
			p\chi\int_0^t
			\|u(s)\|_{L^2}^{p-2}
			\bigl(u(s)\nabla v(s),\nabla u(s)\bigr)\,ds  \\
			&=
			-\frac{p\chi}{2}\int_0^t
			\|u(s)\|_{L^2}^{p-2}
			\bigl(\Delta v(s),u(s)^2\bigr)\,ds \\
			&\le
			\frac{p\chi}{2}
			\int_0^t
			\|u(s)\|_{L^2}^{p-2}
			\|\Delta v(s)\|_{L^2}
			\|u(s)\|_{L^4}^{2}\,ds .
		\end{align*}
		
		Using the elliptic estimate \eqref{elliptic_estimate}, which implies
		$\|\Delta v\|_{L^2}\le C\|u\|_{L^2}$, we have
		\begin{align*}
			I_1
			&\le
			 Cp\chi
			\int_0^t
			\|u(s)\|_{L^2}^{p-1}
			\|u(s)\|_{L^4}^{2}\,ds  \\
			&\le
			\frac{p}{4}
			\int_0^t
			\|u(s)\|_{L^2}^{p-2}
			\|u(s)\|_{L^4}^{4}\,ds
			+
			C p\chi^2
			\int_0^t
			\|u(s)\|_{L^2}^{p}\,ds .
		\end{align*}
		
		Next, using Assumption~\eqref{Assump_LinearGrowth} with $m=0$ and the
		Cauchy--Schwarz inequality, we also have
		\begin{align*}
			I_2+I_3
			&:=
			\frac{p}{2}\int_0^t
			\|u(s)\|_{L^2}^{p-2}\|B(u(s))\|_{L^2}^{2}\,ds
			+
			\frac{p(p-2)}{2}\int_0^t
			\|u(s)\|_{L^2}^{p-4}
			\bigl(u(s),B(u(s))\bigr)^2\,ds  \\
			&\le
			\frac{p}{2}\int_0^t
			\|u(s)\|_{L^2}^{p-2}\|B(u(s))\|_{L^2}^{2}\,ds
			+
			\frac{p(p-2)}{2}\int_0^t
			\|u(s)\|_{L^2}^{p-2}
			\|B(u(s))\|_{L^2}^{2}\,ds  \\
			&=
			\frac{p(p-1)}{2}\int_0^t
			\|u(s)\|_{L^2}^{p-2}\|B(u(s))\|_{L^2}^{2}\,ds  \\
			&\le
			C_p c_B \int_0^t
			\|u(s)\|_{L^2}^{p-2}
			\left(1+\|u(s)\|_{L^2}^{2}\right)\,ds  \\
			&\le
			C_{p,c_B}\int_0^t
			\left(1+\|u(s)\|_{L^2}^{p}\right)\,ds .
		\end{align*}
		
		Substituting the estimates for $I_1$, $I_2$, and $I_3$ back into the
		previous identity and absorbing the first term in the estimate of $I_1$ into
		the left-hand side, we obtain
		\begin{align}
			\|u(t)\|_{L^2}^{p}
			&+
			p\int_0^t
			\|u(s)\|_{L^2}^{p-2}\|\nabla u(s)\|_{L^2}^{2}\,ds
			+
			\frac{3p}{4}\int_0^t
			\|u(s)\|_{L^2}^{p-2}\|u(s)\|_{L^4}^{4}\,ds
			\nonumber\\
			&\le
			\|u_0\|_{L^2}^{p}
			+
			C_{p,\chi,c_B}\int_0^t
			\left(1+\|u(s)\|_{L^2}^{p}\right)\,ds
			\nonumber\\
			&\quad
			+
			p\int_0^t
			\|u(s)\|_{L^2}^{p-2}
			\bigl(u(s),B(u(s))\bigr)\,dW(s).
		\end{align}
		Taking the supremum over $t\in[0,T]$ and then taking expectation, we get
		\begin{align}
			\mE\Bigg[
			\sup_{0\le t\le T}\|u(t)\|_{L^2}^{p}
			&+
			p\int_0^T
			\|u(s)\|_{L^2}^{p-2}\|\nabla u(s)\|_{L^2}^{2}\,ds
			\nonumber\\
			&+
			\frac{3p}{4}\int_0^T
			\|u(s)\|_{L^2}^{p-2}\|u(s)\|_{L^4}^{4}\,ds
			\Bigg]
			\nonumber\\
			&\le
			\mE\left[\|u_0\|_{L^2}^{p}\right]
			+
			C_{p,\chi,c_B}
			\mE\left[
			\int_0^T
			\left(1+\|u(s)\|_{L^2}^{p}\right)\,ds
			\right]
			\nonumber\\
			&\quad
			+
			p\,\mE\left[
			\sup_{0\le t\le T}
			\left|
			\int_0^t
			\|u(s)\|_{L^2}^{p-2}
			\bigl(u(s),B(u(s))\bigr)\,dW(s)
			\right|
			\right].
		\end{align}
		For the stochastic term, by the Burkholder--Davis--Gundy inequality,
		the Cauchy--Schwarz inequality, and Assumption~\eqref{Assump_LinearGrowth},
		we have
	\begin{align}
		& p\,\mE\left[
		\sup_{0\le t\le T}
		\left|
		\int_0^t
		\|u(s)\|_{L^2}^{p-2}
		\bigl(u(s),B(u(s))\bigr)\,dW(s)
		\right|
		\right]
		\nonumber\\
		&\quad\le
		C_p\mE\left[
		\left(
		\int_0^T
		\|u(s)\|_{L^2}^{2p-2}
		\|B(u(s))\|_{L^2}^{2}\,ds
		\right)^{1/2}
		\right]
		\nonumber\\
		&\quad\le
		C_{p,c_B}\mE\left[
		\left(
		\sup_{0\le s\le T}\|u(s)\|_{L^2}^{p}
		\int_0^T
		\|u(s)\|_{L^2}^{p-2}
		\left(1+\|u(s)\|_{L^2}^{2}\right)\,ds
		\right)^{1/2}
		\right]
		\nonumber\\
		&\quad\le
		\frac12\mE\left[
		\sup_{0\le s\le T}\|u(s)\|_{L^2}^{p}
		\right]
		+
		C_{p,c_B}\mE\left[
		\int_0^T
		\left(1+\|u(s)\|_{L^2}^{p}\right)\,ds
		\right].
	\end{align}
	Therefore, absorbing the first term on the right-hand side into the left-hand
	side yields
	\begin{align}
		\mE\Bigg[
		\frac12\sup_{0\le t\le T}\|u(t)\|_{L^2}^{p}
		&+
		p\int_0^T
		\|u(s)\|_{L^2}^{p-2}\|\nabla u(s)\|_{L^2}^{2}\,ds
		\nonumber\\
		&+
		\frac{3p}{4}\int_0^T
		\|u(s)\|_{L^2}^{p-2}\|u(s)\|_{L^4}^{4}\,ds
		\Bigg]
		\nonumber\\
		&\le
		\mE\left[\|u_0\|_{L^2}^{p}\right]
		+
		C_{p,\chi,c_B}
		\int_0^T
		\left(
		1+
		\mE\left[
		\sup_{0\le r\le s}\|u(r)\|_{L^2}^{p}
		\right]
		\right)\,ds .
	\end{align}
	The proof is complete by using the Gronwall inequality.

\bigskip
(b) To prove \eqref{eq2.7}, applying It\^o's formula to
\[
\Phi(u(t))=\|u(t)\|_{L^4}^{4}=\int_D |u(t)|^4\,dx,
\]
we obtain
\begin{align}
	\|u(t)\|_{L^4}^{4}
	&=
	\|u_0\|_{L^4}^{4}
	+
	4\int_0^t
	\bigl(u^3(s),
	\Delta u(s)
	-\chi\nabla\cdot(u(s)\nabla v(s))
	+g(u(s))\bigr)\,ds
	\nonumber\\
	&\quad
	+
	4\int_0^t
	\bigl(u^3(s),B(u(s))\bigr)\,dW(s)
	\nonumber\\
	&\quad
	+
	6\int_0^t
	\bigl(u^2(s)B(u(s)),B(u(s))\bigr)\,ds,
\end{align}
which, together with integration by parts implies that
\begin{align}\label{eq2.27}
	\|u(t)\|_{L^4}^{4} + 12 \int_{0}^t \|u(s)\nab u(s)\|^2_{L^2}\, ds + 4\int_{0}^t \|u(s)\|^6_{L^6}\, ds
	&=
	\|u_0\|_{L^4}^{4}
	+
	4\int_0^t \|u(s)\|^4_{L^4}\,ds \\\nonumber
	&\quad + 12\chi\int_{0}^t \bigl(u^3(s)\nab u(s), \nab v(s)\bigr) \, ds
	\\\nonumber
	&\quad
	+
	4\int_0^t
	\bigl(u^3(s),B(u(s))\bigr)\,dW(s)
\\\nonumber
	&\quad
	+
	6\int_0^t
	\bigl(u^2(s)B(u(s)),B(u(s))\bigr)\,ds.
\end{align}

We next estimate the chemotaxis nonlinear term. Using integration by parts and the
homogeneous Neumann boundary condition, we obtain
\begin{align*}
	12\chi\bigl(u^3\nabla u,\nabla v\bigr)
	&=
	3\chi\bigl(\nabla(u^4),\nabla v\bigr)  \\
	&=
	-3\chi\bigl(\Delta v,u^4\bigr).
\end{align*}
Therefore, using H\"older's inequality and the elliptic estimate from \eqref{elliptic_estimate}, which implies that $\|\Delta v\|_{L^4}\le C\|u\|_{L^4},$ we obtain
\begin{align*}
	12\chi\bigl(u^3\nabla u,\nabla v\bigr)
	&\le
	3\chi\|\Delta v\|_{L^4}\|u\|_{L^4}\|u^3\|_{L^2} \\
	&=
	3\chi\|\Delta v\|_{L^4}\|u\|_{L^4}\|u\|_{L^6}^{3}\\
	&\le
	C\chi \|u\|_{L^4}^{2}\|u\|_{L^6}^{3}\\
	&\le
	\frac14\|u\|_{L^6}^{6}
	+
	C\|u\|_{L^4}^{4}.
\end{align*}

Next, using the assumption \eqref{Assump_LinearGrowth} with $m = 0,1$, we estimate the noise term as follows.
\begin{align*}
	6\int_0^t
	\bigl(u^2(s)B(u(s)),B(u(s))\bigr)\,ds &\le 6\int_{0}^t\|u(s)\|^2_{L^4}\|B(u(s))\|^2_{L^4}\, ds\\\nonumber
	&\le 6C\int_{0}^t\|u(s)\|^2_{L^4}\|B(u(s))\|_{L^2}\|\nab B(u(s))\|_{L^2}\, ds\\\nonumber
	&\le 6Cc^2_{B}\int_{0}^t\|u(s)\|^2_{L^4}\|u(s)\|_{L^2}\|\nab u(s)\|_{L^2}\, ds\\\nonumber
	&\le 3Cc^2_{B}\int_{0}^t\|u(s)\|^4_{L^4}\|u(s)\|^2_{L^2}\, ds + 3Cc^2_{B}\int_{0}^t\|\nab u(s)\|^2_{L^2}\, ds.
\end{align*}

Substituting the above estimates into \eqref{eq2.27}, we obtain
\begin{align}\label{eq2.28}
	\|u(t)\|_{L^4}^{4}
	&+
	12 \int_0^t \|u(s)\nabla u(s)\|_{L^2}^{2}\,ds
	+
	\frac{15}{4}\int_0^t \|u(s)\|_{L^6}^{6}\,ds
	\nonumber\\
	&\le
	\|u_0\|_{L^4}^{4}
	+
	C\int_0^t \|u(s)\|_{L^4}^{4}\,ds
	+
	C c_B^2\int_0^t
	\|u(s)\|_{L^4}^{4}\|u(s)\|_{L^2}^{2}\,ds
	\nonumber\\
	&\quad
	+
	C c_B^2\int_0^t
	\|\nabla u(s)\|_{L^2}^{2}\,ds
	+
	4\int_0^t
	\bigl(u^3(s),B(u(s))\bigr)\,dW(s).
\end{align}

Taking the supremum over $t\in[0,T]$ and then taking expectation to \eqref{eq2.28}, we have
\begin{align}\label{eq2.29}
	&\mE\Bigg[
	\sup_{0\le t\le T}\|u(t)\|_{L^4}^{4}
	+
	12 \int_0^T\|u(s)\nabla u(s)\|_{L^2}^{2}\,ds
	+
	\frac{15}{4}\int_0^T\|u(s)\|_{L^6}^{6}\,ds
	\Bigg]
	\nonumber\\
	&\quad\le
	\mE\left[\|u_0\|_{L^4}^{4}\right]
	+
	C\mE\left[
	\int_0^T
	\|u(s)\|_{L^4}^{4}
	\left(1+\|u(s)\|_{L^2}^{2}\right)\,ds
	\right]
	\nonumber\\
	&\qquad
	+
	Cc_B^2
	\mE\left[
	\int_0^T
	\|\nabla u(s)\|_{L^2}^{2}\,ds
	\right]
	\nonumber\\
	&\qquad
	+
	4\mE\left[
	\sup_{0\le t\le T}
	\left|
	\int_0^t
	\bigl(u^3(s),B(u(s))\bigr)\,dW(s)
	\right|
	\right].
\end{align}

Now, we control the last term on the right-hand side of \eqref{eq2.29}.
By using the Burkholder-Davis-Gundy inequality, the Cauchy-Schwarz inequality,
and Assumption~\eqref{Assump_LinearGrowth} with $m=0$, we have
\begin{align}
	&4\mE\left[
	\sup_{0\le t\le T}
	\left|
	\int_0^t
	\bigl(u^3(s),B(u(s))\bigr)\,dW(s)
	\right|
	\right]
\\\nonumber
	&\quad\le
	C\mE\left[
	\left(
	\int_0^T
	\|u(s)\|_{L^6}^{6}\|B(u(s))\|_{L^2}^{2}\,ds
	\right)^{1/2}
	\right]
\\\nonumber
	&\quad\le
	C\mE\left[
	\left(
	\int_0^T
	\|u(s)\|_{L^6}^{6}
	\left(1+\|u(s)\|_{L^2}^{2}\right)\,ds
	\right)^{1/2}
	\right]\\\nonumber
	&\quad\le
	C\mE\left[
	\left(
	1+\sup_{0\le s\le T}\|u(s)\|_{L^2}^{2}
	\right)^{1/2}
	\left(
	\int_0^T\|u(s)\|_{L^6}^{6}\,ds
	\right)^{1/2}
	\right]
\\\nonumber
	&\quad\le
	\frac14
	\mE\left[
	\int_0^T\|u(s)\|_{L^6}^{6}\,ds
	\right]
	+
	C\mE\left[
	1+\sup_{0\le s\le T}\|u(s)\|_{L^2}^{2}
	\right].
\end{align}

Substituting this estimate into \eqref{eq2.29}, and absorbing the like-terms
into the left-hand side, we obtain
\begin{align}\label{eq2.31}
	&\mE\Bigg[
	\sup_{0\le t\le T}\|u(t)\|_{L^4}^{4}
	+
	12 \int_0^T\|u(s)\nabla u(s)\|_{L^2}^{2}\,ds
	+
	\frac{7}{2}\int_0^T\|u(s)\|_{L^6}^{6}\,ds
	\Bigg]
	\nonumber\\
	&\quad\le
	\mE\left[\|u_0\|_{L^4}^{4}\right]
	+
	C\mE\left[
	\int_0^T
	\|u(s)\|_{L^4}^{4}
	\left(1+\|u(s)\|_{L^2}^{2}\right)\,ds
	\right]
	\nonumber\\
	&\qquad
	+
	Cc_B^2
	\mE\left[
	\int_0^T
	\|\nabla u(s)\|_{L^2}^{2}\,ds
	\right]
	+
	C\mE\left[
	1+\sup_{0\le s\le T}\|u(s)\|_{L^2}^{2}
	\right].
\end{align}

The proof is complete by using Part (a) to control the right-hand side of \eqref{eq2.31}.

	\end{proof}
	
	\subsection{Variational mixed formulation}
	To control the second order nonlinear term of \eqref{eq_2.1a}, we follow the approach of \cite{vo2025analysis} by introducing the auxiliary variable $\vsigma = \nab v \in \vH_{\vsigma}^1(D)$ in \eqref{eq_2.1b}, we obtain
	\begin{align}\label{eq2.1}
		\bigl(\vsigma,\vvarphi\bigr) + \bigl(v,\nab\cdot \vvarphi\bigr) &= 0,\\
		\label{eq2.2}	\bigl(\nab\cdot\vsigma, \psi\bigr) + (u,\psi)&= \bigl(v,\psi\bigr),
	\end{align}
	for $\vvarphi\in \vH_{\vsigma}^1(D)$, and $\psi\in L^2(D)$. Next, choosing $\psi = \nab\cdot\vvarphi$ in $\eqref{eq2.2}$ and replacing the result into $\eqref{eq2.1}$, and also adding a zero term $\bigl(rot\, \vsigma,\, rot \,\vvarphi \bigr)$ (this is zero because of the fact that $rot \,\vsigma = 0$) to $\eqref{eq2.1}$, we obtain the following mixed variational formulation: Find $\bigl(u,\vsigma, v\bigr) \in  L^2(\Ome; C(0,T; L^2(D)))\cap L^2(\Ome; L^2(0,T; H^1(D))) \times L^2(\Ome; C(0,T; \vH_{\vsigma}^1(D))) \times L^2(\Ome; C(0,T; L^2(D)))$ such that $\mP$-a.s.
	\begin{subequations}\label{Weak_formulation}
		\begin{align}
			\label{eq_sigma}	&\bigl(\vsigma, \vvarphi\bigr) + \bigl(\nab\cdot \vsigma, \nab\cdot\vvarphi\bigr) + \bigl(rot \,\vsigma, rot \,\vvarphi\bigr) = -\bigl(u,\nab\cdot\vvarphi\bigr)\\
			\label{eq_u}	&\bigl(u(t), \varphi\bigr) +  \int_0^t \bigl(\nab u, \nab\varphi\bigr)  \, ds = \bigl(u_0, \varphi\bigr)+\chi \int_0^t \bigl(u\vsigma,\nab \varphi\bigr)\, ds + \int_{0}^t \bigl(g(u), \varphi\bigr)\, ds \\\nonumber&\qquad\qquad\qquad\qquad\qquad\qquad\qquad+ \label{eq_c}\Bigl(\int_0^t B(u)dW(s),\varphi\Bigr),\\
			&\bigl(v,\psi\bigr)=	\bigl(\nab\cdot\vsigma, \psi\bigr) + (u,\psi),
		\end{align}
	\end{subequations}
	for all $\bigl(\varphi,\vvarphi, \psi\bigr) \in {H}^1(D) \times \vH_{\vsigma}^1(D) \times L^2(D)$.

		\begin{lemma}\label{Lemma_Holder}
	Let $(u,\vsigma,v)$ be the solution of \eqref{Weak_formulation}, where
	$\vsigma=\nabla v$. Assume that $B$ satisfies the assumption \eqref{Assump_LinearGrowth}. Then the following estimates hold.
	
	\begin{enumerate}
		\item[(a)] If $u \in L^{2p}(\Omega;C([0,T]; H^1(D)))\cap L^{6p}(\Ome; L^{6p}(0,T ; H^1(D)))$ for all $p\geq 1$, then for all
		$0\le s\le t\le T$,
		\begin{align}
			\mE\left[\|u(t)-u(s)\|_{L^2}^{2p}\right]
			&\le C_{Holder,1}|t-s|^{p}, \label{eq_holder_L2}\\
			\mE\left[\|\vsigma(t)-\vsigma(s)\|_{H^1}^{2p}\right]
			&\le C_{Holder,1}\,|t-s|^{p}, \label{eq_holder_sigma}
		\end{align}
		where 
		\begin{align}
			C_{Holder,1}&:= \left(\sup_{s\in[0,T]}\mE\left[\|u(s)\|^{2p}_{H^1}\right] + \int_0^T \mE\left[\|u\|^{2p}_{L^4}
			\|\nabla u\|^{2p}_{L^2}+
			\|u\|_{L^4}^{4p}\right]\, dr \right.
			\\\nonumber
			&\qquad\left.+\int_0^T\mE\left[(c_B+1)\left\|u(r)\right\|^{2p}_{L^2} + \|u(r)\|^{6p}_{L^6}\right]\, dr \right).
		\end{align}
		
		\item[(b)] If $u\in L^2(\Omega;L^{\infty}(0,T;H^2(D)))\cap L^6(\Ome;L^{\infty}(0,T;H^1(D)))$, then for all
		$0\le s\le t\le T$,
		\begin{align}
			\mE\left[\|\nabla(u(t)-u(s))\|_{L^2}^{2}\right]
			&\le C_{Holder,2}\,|t-s| \label{eq_holder_H1_2},
		\end{align}
		where 
		\begin{align*}
			C_{Holder,2} &:= \left(4\mE\left[\sup_{s\in[0,T]}\|\Delta u(s)\|^2_{L^2}\right] + \mE\left[4C\chi^2\sup_{r\in[0,T]} \|u(r)\|^2_{L^4} (\|\nab u(r)\|^2_{L^2} + \|u(r)\|^2_{L^4})\right] \right.\\\nonumber
			&\qquad\left.+ \mE\left[8\sup_{r\in[0,T]}(\|u(r)\|^2_{L^2} + \|u(r)\|^6_{L^6})\right] + \mE\left[c_B \sup_{r\in[0,T]}\|\nab u(r)\|^2_{L^2}\right] \right)
		\end{align*}
	\end{enumerate}
\end{lemma}
\begin{proof}
	(a) To prove \eqref{eq_holder_L2}, we recall the mild solution \eqref{mild_sol}
	\begin{align}\label{eq:mild}
		u(t)
		=& e^{-(t-s)A}u(s)
		-\chi\int_s^t e^{-(t-r)A}
		\nabla\cdot\bigl(u(r)\nabla v(r)\bigr)\,dr\\\nonumber 
		&+\int_s^t e^{-(t-r)A}g(u(r))\,dr
		+\int_s^t e^{-(t-r)A}B(u(r))\,dW(r),
		\qquad t\ge s,
	\end{align}
	which implies that
	\begin{align}
		u(t) - u(s) = I + II + III + IV,
	\end{align}
	where $I, II, III$ and $IV$ will be specified as we estimate them.
	
	First, we estimate $I$ as follow. Using the property 
	\begin{align}
		\mE[\|I\|^{2p}_{L^2}] &= \mE\left[\|(e^{-(t-s)A} - I)u(s)\|^{2p}_{L^2}\right]\\\nonumber
		&= \mE\left[\|A^{-1/2}(e^{-(t-s)A} - I) A^{1/2}u(s)\|^{2p}_{L^2}\right]\\\nonumber
		&\leq \mE\left[\|A^{-1/2}(e^{-(t-s)A} - I) \|^{2p}_{\mathcal{L}(L^2)}\|A^{1/2}u(s)\|^{2p}_{L^2}\right]\\\nonumber
		&\leq C(t-s)^p\mE\left[\|u(s)\|^{2p}_{H^1}\right].
	\end{align}
	
	Next, using the H\"older inequality, we estimate $II$ as follows: 
	\begin{align}\label{eq2.38}
		\mE\left[\|II\|^{2p}_{L^2}\right] &= \mE\left[\left\|\chi\int_s^t e^{-(t-r)A}
		\nabla\cdot\bigl(u(r)\nabla v(r)\bigr)\,dr\right\|^{2p}_{L^2}\right]\\\nonumber
		&\leq \chi^{2p} (t-s)^{2p-1} \int_s^t \mE\left[\|e^{-(t-r)A}
		\nabla\cdot\bigl(u(r)\nabla v(r)\bigr)\|^{2p}_{L^2}\right]\, dr\\\nonumber
		&\leq C (t-s)^{2p-1} \int_s^t \mE\left[\|
		\nabla\cdot\bigl(u(r)\nabla v(r)\bigr)\|^{2p}_{L^2}\right]\, dr.
	\end{align}
	
	Observe that
	
	\[
	\nabla\cdot(u\nabla v)
	=
	\nabla u\cdot\nabla v
	+
	u\Delta v.
	\]
	
	Hence,
	
	\[
	\begin{aligned}
		\|\nabla\cdot(u\nabla v)\|_{L^2}
		&\le
		\|\nabla u\cdot\nabla v\|_{L^2}
		+
		\|u\Delta v\|_{L^2}
		\\
		&\le
		\|\nabla u\|_{L^2}
		\|\nabla v\|_{L^\infty}
		+
		\|u\|_{L^4}
		\|\Delta v\|_{L^4}.
	\end{aligned}
	\]
	
	Using the elliptic estimate $\|\nabla v\|_{L^\infty}
	+
	\|\Delta v\|_{L^4}
	\le
	C\|u\|_{L^4},$
	we obtain
	\begin{align}\label{eq2.32}
	\|\nabla\cdot(u\nabla v)\|_{L^2}
	&\le
	C
	\|\nabla u\|_{L^2}
	\|u\|_{L^4}
	+
	C
	\|u\|_{L^4}^2
	\\\nonumber
	&=
	C
	\|u\|_{L^4}
	\left(
	\|\nabla u\|_{L^2}
	+
	\|u\|_{L^4}
	\right).
\end{align}
	
	With this, we update the right-hand side of \eqref{eq2.38} as follows.
	\begin{align}
		\mE\left[\|II\|^{2p}_{L^2}\right] &= \mE\left[\left\|\chi\int_s^t e^{-(t-r)A}
		\nabla\cdot\bigl(u(r)\nabla v(r)\bigr)\,dr\right\|^{2p}_{L^2}\right]\\\nonumber
		&\leq \chi^{2p} (t-s)^{2p-1} \int_s^t \mE\left[\|e^{-(t-r)A}
		\nabla\cdot\bigl(u(r)\nabla v(r)\bigr)\|^{2p}_{L^2}\right]\, dr\\\nonumber
		&\leq C (t-s)^{2p-1} \int_s^t \mE\left[\|u\|^{2p}_{L^4}
		\left(
		\|\nabla u\|_{L^2}
		+
		\|u\|_{L^4}
		\right)^{2p}\right]\, dr\\\nonumber
		&\leq C (t-s)^{2p-1} \int_s^t \mE\left[\|u\|^{2p}_{L^4}
		\|\nabla u\|^{2p}_{L^2}
		+
		\|u\|_{L^4}^{4p}\right]\, dr.
	\end{align}
	
	Now, we turn to estimate $III$. Using the H\"older inequality, we obtain
	\begin{align}
		\mE\left[\|III\|^{2p}_{L^2}\right] &= \mE\left[\left\|\int_s^t e^{-(t-r)A}g(u(r))\, dr\right\|^{2p}_{L^2}\right]\\\nonumber
		&\leq C(t-s)^{2p-1}\int_s^t \mE\left[\left\|g(u(r))\right\|^{2p}_{L^2}\right]\, dr\\\nonumber
		&\leq C(t-s)^{2p-1}\int_s^t \mE\left[\left\|u(r)\right\|^{2p}_{L^2} + \|u(r)\|^{6p}_{L^6}\right]\, dr.
	\end{align}
	
	Using the Burkholder-Davis-Gundy inequality and the H\"older inequality and then the semigroup property and the assumption \eqref{Assump_LinearGrowth} with $m =0$, we estimate $IV$ as follows.
	\begin{align}
		\mE\left[\|IV\|^{2p}_{L^2}\right] &\leq \mE\left[\left\|\int_s^t e^{-(t-r)A}B(u(r))\,dW(r)\right\|^{2p}_{L^2}\right]\\\nonumber
		&\leq C\mE\left[\left(\int_s^t \left\|e^{-(t-r)A}B(u(r))\right\|^{2}_{L^2}\,dr\right)^p\right]\\\nonumber
		&\leq C(t-s)^p\int_s^t \mE\left[\left\|e^{-(t-r)A}B(u(r))\right\|^{2p}_{L^2}\right]\,dr\\\nonumber
		&\leq Cc_B(t-s)^p\int_s^t \mE\left[\left\|u(r)\right\|^{2p}_{L^2}\right]\,dr.
	\end{align}
	
	Collecting all the estimates from $I, II, III$ and $IV$, we obtain
	\begin{align}
		\mE\left[\|u(t) - u(s)\|^{2p}_{L^2}\right] \leq &C(t-s)^p\left(\mE\left[\|u(s)\|^{2p}_{H^1}\right] + \int_s^t \mE\left[\|u\|^{2p}_{L^4}
		\|\nabla u\|^{2p}_{L^2}+
		\|u\|_{L^4}^{4p}\right]\, dr \right.
		\\\nonumber
		&\qquad\left.+\int_s^t\mE\left[\left\|u(r)\right\|^{2p}_{L^2} + \|u(r)\|^{6p}_{L^6}\right]\, dr + \int_s^t \mE\left[c_B\left\|u(r)\right\|^{2p}_{L^2}\right]\,dr\right)\\\nonumber
		\leq &C(t-s)^p\left(\sup_{s\in[0,T]}\mE\left[\|u(s)\|^{2p}_{H^1}\right] + \int_0^T \mE\left[\|u\|^{2p}_{L^4}
		\|\nabla u\|^{2p}_{L^2}+
		\|u\|_{L^4}^{4p}\right]\, dr \right.
		\\\nonumber
		&\qquad\left.+\int_0^T\mE\left[(c_B+1)\left\|u(r)\right\|^{2p}_{L^2} + \|u(r)\|^{6p}_{L^6}\right]\, dr \right).
	\end{align}

	(b) To prove \eqref{eq_holder_H1_2}, we apply the It\^o formula to $\Phi(u(t)) = \|\nab (u(t) - u(s))\|^2_{L^2}$ as follows: 
	\begin{align*}
		\|\nab(u(t) - u(s))\|^2_{L^2} &= -2\int_s^t \bigl(\Delta u(r), \Delta(u(r) - u(s))\bigr)\, dr \\\nonumber
		&\qquad- 2\chi\int_s^t \bigl(\nab\cdot(u(r)\nab v(r)), \Delta(u(r) - u(s))\bigr)\, dr\\\nonumber
		&\qquad + 2\int_s^t \bigl( g(u(r)), \Delta(u(r) - u(s))\bigr)\, dr\\\nonumber
		&\qquad + 2\int_s^t \bigl(\nab B(u(r)), \nab(u(r) - u(s))\bigr)\, dW(r)+ \int_s^t \|\nab B(u(r))\|^2_{L^2}\, dr,
	\end{align*}
	which implies that
	\begin{align}\label{eq3.43}
		\|\nab(u(t) - u(s))\|^2_{L^2} &+2\int_s^t \|\Delta(u(r) - u(s))\|^2_{L^2}\, dr\\\nonumber &= - 2\int_s^t \bigl(\Delta u(s), \Delta(u(r) - u(s))\bigr)\, dr \\\nonumber
		&\qquad- 2\chi\int_s^t \bigl(\nab\cdot(u(r)\nab v(r)), \Delta(u(r) - u(s))\bigr)\, dr\\\nonumber
		&\qquad + 2\int_s^t \bigl( g(u(r)), \Delta(u(r) - u(s))\bigr)\, dr\\\nonumber
		&\qquad + 2\int_s^t \bigl(\nab B(u(r)), \nab(u(r) - u(s))\bigr)\, dW(r)+ \int_s^t \|\nab B(u(r))\|^2_{L^2}\, dr\\\nonumber
		&:= T_1 + T_2 + T_3 + T_4 + T_5.
	\end{align}
	
	Now, we estimate $T_1,..., T_5$ as follows. First, using the Cauchy-Schwarz inequality, we obtain
	\begin{align*}
		T_1 &\leq \frac14\int_s^t \|\Delta(u(r) - u(s))\|^2_{L^2}\, dr + 4\|\Delta u(s)\|^2_{L^2} |t-s|.
	\end{align*} 
	
	Similarly, using the Cauchy-Schwarz inequality, we also get
	\begin{align*}
		T_2 &\leq \frac14\int_s^t \|\Delta(u(r) - u(s))\|^2_{L^2}\, dr + 4\chi^2\int_s^t \|\nab\cdot(u\nab v)\|^2_{L^2}\, dr\\\nonumber
		&\leq \frac14\int_s^t \|\Delta(u(r) - u(s))\|^2_{L^2}\, dr + 4C\chi^2\int_s^t \|u(r)\|^2_{L^4} (\|\nab u(r)\|^2_{L^2} + \|u(r)\|^2_{L^4})\, dr\\\nonumber
		&\leq \frac14\int_s^t \|\Delta(u(r) - u(s))\|^2_{L^2}\, dr + 4C\chi^2\sup_{r\in[0,T]} \|u(r)\|^2_{L^4} (\|\nab u(r)\|^2_{L^2} + \|u(r)\|^2_{L^4}) |t-s|,
	\end{align*}
	where the second inequality of $T_2$ is obtained by using the estimate from \eqref{eq2.32}. 
	
	Also, using the definition of $g$ and the Cauchy-Schwarz inequality we get
	\begin{align}
		T_3 &\leq \frac14 \int_s^t \|\Delta(u(r) - u(s))\|^2_{L^2}\, dr + 8\sup_{r\in[0,T]}(\|u(r)\|^2_{L^2} + \|u(r)\|^6_{L^6})|t-s|.
	\end{align}
	
	Since $T_4$ is an It\^o integral, its expectation equals $0$. Additionally, using the assumption of \eqref{Assump_LinearGrowth} with $m=1$, we get
	\begin{align*}
	T_5 &\leq c_B\int_s^t \|\nab u(r)\|^2_{L^2}\, dr \leq c_B \sup_{r\in[0,T]}\|\nab u(r)\|^2_{L^2} |t-s|.
	\end{align*}
	
	Collecting all the estimates from $T_1, T_2, ..., T_5$ and then substituting them into \eqref{eq3.43} with the expectation, we obtain
	\begin{align*}
			&\mE\left[\|\nab(u(t) - u(s))\|^2_{L^2} + \int_s^t \|\Delta(u(r) - u(s))\|^2_{L^2}\, dr\right]\\\nonumber 
			&\leq \left(4\mE\left[\sup_{s\in[0,T]}\|\Delta u(s)\|^2_{L^2}\right] + \mE\left[4C\chi^2\sup_{r\in[0,T]} \|u(r)\|^2_{L^4} (\|\nab u(r)\|^2_{L^2} + \|u(r)\|^2_{L^4})\right] \right.\\\nonumber
			&\qquad\left.+ \mE\left[8\sup_{r\in[0,T]}(\|u(r)\|^2_{L^2} + \|u(r)\|^6_{L^6})\right] + \mE\left[c_B \sup_{r\in[0,T]}\|\nab u(r)\|^2_{L^2}\right] \right)|t-s|.
	\end{align*}
	
	The proof is complete.
	
\end{proof}

	\section{Fully discrete splitting mixed finite element method}\label{sec4}
	
\subsection{Formulation of the finite element method}

In this section, we present the fully discrete splitting mixed finite element method for the mixed formulation \eqref{Weak_formulation}. The temporal discretization is based on the implicit Euler scheme, while the spatial discretization is carried out using a splitting mixed finite element method.

Let $k=T/M$ be the uniform time step associated with the partition $\{t_m\}_{m=0}^{M}$ of the interval $[0,T]$. Furthermore, let $\{\mathcal{T}_h\}_{h>0}$ be a family of quasi-uniform triangulations of $D\subset\mathbb{R}^2$ with mesh size $0<h\ll1$. We consider the finite element spaces
\[
X_u^h\times X_{\vsigma}^h\times X_v^h
\subset
H^1(D)\times\mathbf{H}^1(D)\times H^1(D),
\]
consisting of continuous piecewise polynomial functions of degrees at most $r_1$, $r_2$, and $r_3$, respectively, where $r_1,r_2,r_3\ge1$.

Unlike standard mixed finite element methods, the proposed splitting formulation does not require the Ladyzhenskaya--Babu\v{s}ka--Brezzi (LBB) stability condition. Consequently, the finite element spaces can be chosen independently. Throughout this paper, we employ continuous piecewise linear ($P_1$) finite element spaces for all unknowns. Although the analysis is presented only for linear finite element spaces, it extends in a straightforward manner to higher-order polynomial finite element spaces.

Next, we introduce the projection operators
\begin{align}
	&\mathcal{P}_u:L^2(D)\rightarrow X_u^h,\qquad
	\mathcal{P}_{\vsigma}:\mathbf{H}^1(D)\rightarrow X_{\vsigma}^h,\nonumber\\
	&\mathcal{P}_v:L^2(D)\rightarrow X_v^h,
\end{align}
defined by
\begin{align}
	\label{Pu}
	(v-\mathcal{P}_uv,\varphi_h)&=0,
	\qquad\forall\,\varphi_h\in X_u^h,\\
	\label{Ps}
	(\mathbf{v}-\mathcal{P}_{\vsigma}\mathbf{v},\vvarphi_h)
	+(\nabla\cdot(\mathbf{v}-\mathcal{P}_{\vsigma}\mathbf{v}),\nabla\cdot\vvarphi_h)
	+(rot(\mathbf{v}-\mathcal{P}_{\vsigma}\mathbf{v}),rot\,\vvarphi_h)
	&=0,
	\qquad\forall\,\vvarphi_h\in X_{\vsigma}^h,\\
	\label{Pc}
	(v-\mathcal{P}_vv,\psi_h)&=0,
	\qquad\forall\,\psi_h\in X_v^h.
\end{align}

The above projection operators satisfy the standard interpolation estimates
\begin{align}\label{ineq_interpolation}
	\|v-\mathcal{P}_uv\|_{L^2}+h\|v-\mathcal{P}_uv\|_{H^1}
	&\le Ch^2\|v\|_{H^2},
	\qquad \forall\,v\in H^2(D),\nonumber\\
	\|\vvarphi-\mathcal{P}_{\vsigma}\vvarphi\|_{L^2}
	+h\|\vvarphi-\mathcal{P}_{\vsigma}\vvarphi\|_{\mathbf{H}^1}
	&\le Ch^2\|\vvarphi\|_{\mathbf{H}^2},
	\qquad \forall\,\vvarphi\in\mathbf{H}^2(D),\nonumber\\
	\|v-\mathcal{P}_vv\|_{L^2}
	+h\|v-\mathcal{P}_vv\|_{H^1}
	&\le Ch^2\|v\|_{H^2},
	\qquad \forall\,v\in H^2(D).
\end{align}

We are now ready to present the fully discrete splitting mixed finite element method for \eqref{Weak_formulation}.

\bigskip

\noindent
{\bf Main Algorithm.}
Let $u_h^0=\mathcal{P}_u u_0$ and $\vsigma_h^0=\mathcal{P}_{\vsigma}\vsigma_0$. For each time level $m=0,1,\ldots,M-1$, find $\bigl(u_h^{m+1},\vsigma_h^{m+1},v_h^{m+1}\bigr)\in X_u^h\times X_{\vsigma}^h\times X_v^h$ such that, $\mathbb{P}$-almost surely,
\begin{align}
	\label{eq_discrete_sigma}
	&(\vsigma_h^{m+1},\vvarphi_h)
	+(\nabla\cdot\vsigma_h^{m+1},\nabla\cdot\vvarphi_h)
	+(rot\,\vsigma_h^{m+1},rot\,\vvarphi_h)
	=-(u_h^m,\nabla\cdot\vvarphi_h),\\
	\label{eq_discrete_u}
	&(u_h^{m+1}-u_h^m,\varphi_h)
	+k(\nabla u_h^{m+1},\nabla\varphi_h)
	=k\chi(u_h^{m+1}\vsigma_h^{m+1},\nabla\varphi_h)
	+k(g(u_h^{m+1}),\varphi_h)\\\nonumber
	&\qquad\qquad\qquad\qquad\qquad\qquad\qquad\qquad\qquad+(B(u_h^m)\Delta W_m,\varphi_h),\\
	\label{eq_discrete_c}
	&(v_h^{m+1},\psi_h)
	=(\nabla\cdot\vsigma_h^{m+1},\psi_h)
	+(u_h^{m+1},\psi_h),
\end{align}
for all $(\varphi_h,\vvarphi_h,\psi_h)\in X_u^h\times X_{\vsigma}^h\times X_v^h$, where $\Delta W_m=W(t_{m+1})-W(t_m)\sim\mathcal{N}(0,k)$.

\medskip

The proposed scheme naturally decouples the computation at each time step into three successive subproblems. First, the auxiliary variable $\vsigma_h^{m+1}$ is computed from \eqref{eq_discrete_sigma} using the previously computed solution $u_h^m$. Next, the cell density $u_h^{m+1}$ is obtained by solving \eqref{eq_discrete_u}. Finally, the chemical concentration $v_h^{m+1}$ is recovered from the linear relation \eqref{eq_discrete_c}. Consequently, the original coupled stochastic Keller--Segel system is replaced by a sequence of simpler subproblems, eliminating the need to solve a fully coupled mixed system at each time step while retaining a direct approximation of the chemical gradient. This decoupled structure significantly improves the computational efficiency of the method, particularly for large-scale Monte Carlo simulations that require the computation of a large number of independent sample paths.

	Since we choose $u^0_h = \mathcal{P}_u u_0$ and $\vsigma_h^0 = \cP_{\vsigma} \vsigma_0$ and \eqref{ineq_interpolation}, without loss of generality, we assume that $u^0_h = u_0$ and $\vsigma_h^0  = \vsigma_0$ in all the proofs of the stability and error estimates of the Main Algorithm.

	Now, we present the stability estimates of the Main Algorithm. 
	
	\smallskip
	
	\begin{lemma}\label{lemma_fem_stability} Let $(u^m_h,\vsigma_h^m,v^m_h)$ be the solution of the Main Algorithm. Assume that $u_0 \in L^2(\Ome; L^2(D))$. Then, there exist a constant $\tilde{C}>0$ and $k_0>0$ such that for all $k\in(0,k_0)$
		\begin{enumerate}
			\item[(a)] $\displaystyle \mE\left[\max_{1\leq m \leq M}\|u_h^m\|^2_{L^2} +  k\sum_{m=1}^{M}\|\nab u_h^{m}\|^2_{L^2}\right] \leq \tilde{C}$,
			\item[(b)] $\displaystyle \mE\left[\max_{1\leq m \leq M}\|\vsigma_h^m\|^2_{H^1}\right] \leq \tilde{C}$,
			\item[(c)] $\displaystyle \mE\left[\max_{1\leq m \leq M} \|v_h^m\|^2_{L^2}\right]  \leq \tilde{C}$.
		\end{enumerate}
	\end{lemma}
	\begin{proof}
		We proceed to prove (a) as follows. Taking $\varphi_h= u_h^{m+1}$ in \eqref{eq_discrete_u} and using the identity $2a(a-b) = a^2 - b^2 + (a-b)^2$, we obtain
		\begin{align}\label{eq_3.6}
			&\frac{1}{2}\bigl[\|u_h^{m+1}\|^2_{L^2} - \|u_h^{m}\|^2_{L^2} + \|u_h^{m+1} - u_h^m\|^2_{L^2}\bigr] +  k\|\nab u_h^{m+1}\|^2_{L^2}\\\nonumber
			& = \chi k\bigl(u_h^{m+1}\vsigma_h^{m+1}, \nab u_h^{m+1}\bigr)+ k\bigl(g(u^{m+1}_h), u_h^{m+1}\bigr) +\bigl(B(u_h^m)\Delta W_m,  u_h^{m+1} - u_h^m\bigr)\\\nonumber
			&\qquad +\bigl(B(u_h^m)\Delta W_m,  u_h^m\bigr)\\\nonumber
			&:= I + II + III + IV.
		\end{align}
		
		First, we notice that 
		\begin{align*}
			II =  k\|u_h^{m+1}\|^2_{L^2} - k\|u_h^{m+1}\|^4_{L^4}.
		\end{align*}
		
		Next, using integration by parts and the  Cauchy-Schwarz inequality, we obtain
		\begin{align}\label{eq_3.7}
			I&=-\chi k \bigl((u_h^{m+1})^2, \nab\cdot\vsigma_h^{m+1}\bigr)\\\nonumber
			&\leq \chi k \|u_h^{m+1}\|^2_{L^4}\|\nab\cdot \vsigma_h^{m+1}\|_{L^2}.
		\end{align}
		
		Next, choosing $\vvarphi_h = \vsigma_h^{m+1}$ in \eqref{eq_discrete_sigma}, we obtain
		\begin{align}\label{eq_3.12}
			\|\vsigma_h^{m+1}\|^2_{H^1} \leq \|u_h^{m}\|^2_{L^2},
		\end{align}
		which together with \eqref{eq_3.7} give us that 
		\begin{align}\label{eq_3.8}
			I &\leq   \chi k \|u_h^{m+1}\|^2_{L^4}\|u_h^{m}\|_{L^2}\\\nonumber
			&\leq \frac{k}{4} \| u_h^{m+1}\|^4_{L^4} + Ck\|u_h^{m}\|^2_{L^2}.
		\end{align}
		
		Next, we estimate $III$ and $IV$ with the expectation. First, using the independence property of the increments $\Delta W_m$, we immediately have that $\mE\left[IV\right] = 0$. Then, using the Cauchy-Schwarz inequality and the assumption of $B$ as well as the fact that $\mE\left[|\Delta W_m|^2\right] = k$, we obtain
		\begin{align*}
			\mE\left[III\right] &\leq \frac14\mE\left[\|u_h^{m+1} - u_h^m\|^2_{L^2}\right] + C_Bk \mE\left[\|u_h^m\|^2_{L^2}\right].
		\end{align*}
		
		Substituting all the estimates from $I, ..., IV$ into \eqref{eq_3.6} in the expectation, we arrive at
		\begin{align}\label{eq_3.9}
			&	\frac12\mE\left[\|u_h^{m+1}\|^2_{L^2} - \|u_h^m\|^2_{L^2} \right] + \frac14\mE\left[\|u_h^{m+1} - u_h^m\|^2_{L^2}\right]  + \frac34\mE\left[ k\|u_h^{m+1}\|^4_{L^4} + k \|\nab u_h^{m+1}\|^2_{L^2}\right]\\\nonumber
			&\leq Ck\mE\left[\|u_h^{m+1}\|^2_{L^2} + \|u_h^m\|^2_{L^2}\right].
		\end{align}
		
		Next, applying the summation $\sum_{m=0}^{\ell}$ to \eqref{eq_3.9} for any $1\leq \ell \leq M-1$, we obtain
		\begin{align}\label{eq_3.13}
			&	\frac12\mE\left[\|u_h^{\ell+1}\|^2_{L^2}\right] + \mE\left[\sum_{m=0}^{\ell}\left(\frac14\|u_h^{m+1} - u_h^m\|^2_{L^2} + \frac{3k }{4}\|u_h^{m+1}\|^4_{L^4} +  k\|\nab u_h^{m+1}\|^2_{L^2}\right)\right]\\\nonumber
			&\leq \frac{1}{2}\mE\left[\|u_h^0\|^2_{L^2}\right] + Ck\sum_{m=0}^{\ell}\mE\left[\|u_h^{m+1}\|^2_{L^2} + \|u_h^m\|^2_{L^2}\right].
		\end{align}
		
		With this and using the discrete Gronwall inequality, for sufficiently small $k\in (0,1)$, we obtain the estimate in Part (a) after taking the maximum over $0\leq \ell\leq M-1$.
		
		Lastly, it is clear that Part (b) is obtained by taking $\vvarphi_h= \vsigma_h^{m+1}$ and using Part (a). Similarly, Part (c) is a consequence of Parts (a) and (b). The proof is complete.
	\end{proof}
	
	Next, we state and prove the error estimates of the Main Algorithm. 
	
	\subsection{Error estimates} In this part, we analyze and derive the error estimates for the Main Algorithm. To control the nonlinearity, we introduce the following sequence of subsets of the sample space
	\begin{align}\label{omega}
		{\Omega}_{\rho, m} := \left\{\omega \in \Omega; \, \sup_{t \leq t_m} (\|u(t)\|^4_{L^4} + \|u(t)\|^2_{L^2}) \leq \rho \right\} \cap \left\{\omega \in \Omega; \, \sup_{0\leq n \leq m} \|u_h^n\|^2_{L^2} \leq \rho \right\},
	\end{align}
	where $u$ is the variational solution from \eqref{Weak_formulation} and for some $\rho>0$ specified later. We observe that ${\Omega}_{\rho,0} \supset {\Omega}_{\rho,1} \supset ... \supset {\Omega}_{\rho,\ell}$. 
	
	\begin{remark}
		In the error estimate stated in Theorem~\ref{Thm_global_error}, we choose
		\begin{align}
			\rho(k) := \frac{\ln(\ln(1/k))}{\widehat{C}} >0,
		\end{align}
		where $\widehat{C}$ is determined the proof of of Theorem \ref{Thm_global_error}.
		
		With this choice of $\rho(k)$, an application of the Markov inequality together with Lemma~\ref{Stability_PDE} and Lemma \ref{lemma_fem_stability} yields
		\begin{align*}
			\mP(\Omega_{\rho,M}^c)
			&\le \frac{\widehat{C}}{\ln(\ln(1/k))}
			\mE\!\left[
			\sup_{0\le t\le T}
			(\|u(t)\|^4_{L^4} + \|u(t)\|^2_{L^2})+ \sup_{0\leq n \leq M}\|u_h^n\|^2_{L^2}
			\right].
		\end{align*}
		Since the expectation on the right-hand side is finite, we conclude that
		\[
		\mP(\Omega_{\rho,M}^c) \longrightarrow 0
		\qquad \text{as } k \to 0,
		\]
		and hence
		\[
		\mP(\Omega_{\rho,M}) \longrightarrow 1
		\qquad \text{as } k \to 0.
		\]
		
		Therefore, the convergence of the numerical solutions implied by Theorem~\ref{Thm_global_error} is convergence in probability.
		
		\end{remark}
		
		\medskip
	
	\begin{theorem}\label{Thm_global_error}
		Let $(u,\vsigma,v)$ and $(u_h^m,\vsigma_h^m,v_h^m)$ be the variational solution of \eqref{Weak_formulation} and the approximate solution generated by the Main Algorithm, respectively. Assume that $u\in L^2(\Omega;L^{\infty}(0,T;H^2(D)))\cap L^6(\Ome;L^{\infty}(0,T;H^1(D)))$ and assume that $B$ satisfies \eqref{Assump_LinearGrowth} and \eqref{Assump_Lipschitz}.
		Then there exists $k_1>0$ such that, for every $k\in(0,k_1)$, there exists a constant $C\equiv C(u_0,\chi,T)>0,$
		independent of $k$ and $h$, for which the following error estimates hold:
		\begin{align}
			\label{estimate_u_fem}
			\max_{1\leq m\leq M}
			\mE\bigl[\mathbf{1}_{\Omega_{\rho,m}}
			\|u(t_m)-u_h^m\|_{L^2}^2
			\bigr]
			&+
			k
			\sum_{m=1}^{M}
			\mE\bigl[\mathbf{1}_{\Omega_{\rho,m}}
			\|\nab(u(t_m)-u_h^m)\|_{L^2}^2
			\bigr]
		\\\nonumber
		&\qquad	\leq
			C\ln(1/k)(k+h^2),
			\\
			\label{estimate_sigma_fem}
			\max_{1\leq m\leq M}
			\mE\bigl[\mathbf{1}_{\Omega_{\rho,m}}
			\|\vsigma(t_m)-\vsigma_h^m\|_{H^1}^2
			\bigr]
			&\leq
			C\ln(1/k)(k+h^2),
			\\
			\label{estimate_c_fem}
			\max_{1\leq m\leq M}
			\mE\bigl[\mathbf{1}_{\Omega_{\rho,m}}
			\|v(t_m)-v_h^m\|_{L^2}^2
			\bigr]
			&\leq
			C\ln(1/k)(k+h^2).
		\end{align}

	\end{theorem}

	\begin{proof} We only focus on giving the proof of \eqref{estimate_u_fem}. The proofs of \eqref{estimate_sigma_fem} and \eqref{estimate_c_fem} can be deduced directly from \eqref{estimate_u_fem}. First, denote
		\begin{align*}
			e_{u}^m &= u(t_m) - u^m_h  = u(t_m) - \cP_{u} u(t_m) + \cP_{u} u(t_m) - u_h^m:= \theta_{u}^m + \varepsilon_{u}^m  \\
			e_{\vsigma}^m &= \vsigma(t_m) - \vsigma_h^m = \vsigma(t_m) - \cP_{\vsigma}\vsigma(t_m) + \cP_{\vsigma} \vsigma(t_m) - \vsigma_h^m:= \vtheta_{\vsigma}^m + \e_{\vsigma}^m. 
		\end{align*}
		
		Subtracting \eqref{eq_u} from \eqref{eq_discrete_u}, we obtain the following error equation.
		\begin{align}\label{eq_3.27}
			&\bigl(e_u^{m+1} - e^m_u,\varphi_h\bigr) + k\bigl(\nab e_u^{m+1}, \nab \varphi_h\bigr) \\\nonumber
			&= \nu \int_{t_m}^{t_{m+1}} \bigl(\nab u(t_{m+1}) - \nab u(s), \nab \varphi_h\bigr)\, ds \\\nonumber
			&\qquad+ \chi\int_{t_m}^{t_{m+1}} \bigl(u(s)\vsigma(s) - u^{m+1}_h\vsigma^{m+1}_h, \nab \varphi_h\bigr)\, ds \\\nonumber
			&\qquad+\int_{t_m}^{t_{m+1}} \bigl(g(u(s)) - g(u_h^{m+1}), \varphi_h\bigr)\, ds \\\nonumber
			&\qquad+ \Bigl(\int_{t_m}^{t_{m+1}} B(u(s)) dW(s) - B(u_h^m)\Delta W_m, \varphi_h\Bigr).
		\end{align}
		
		Taking $v_h = \varepsilon_u^{m+1}$ in \eqref{eq_3.27} and using the orthogonal property of $\cP_u$ \eqref{Pu}, we obtain the following error equation:
		\begin{align}\label{eq_3.28}
			&\frac12\bigl[\|\varepsilon^{m+1}_u\|^2_{L^2} - \|\varepsilon_u^m\|^2_{L^2}\bigr] +\frac12\|\varepsilon_u^{m+1} - \varepsilon_u^m\|^2_{L^2}+  k\|\nab\varepsilon_{u}^{m+1}\|^2_{L^2}\\\nonumber
			&= -  k\bigl(\nab\theta_u^{m+1},\nab\varepsilon_u^{m+1}\bigr) +   \int_{t_m}^{t_{m+1}} \bigl(\nab u(t_{m+1}) - \nab u(s), \nab \varepsilon_u^{m+1}\bigr)\, ds \\\nonumber
			&\qquad+ \chi\int_{t_m}^{t_{m+1}} \bigl(u(s)\vsigma(s) - u^{m+1}_h\vsigma^{m+1}_h, \nab \varepsilon_u^{m+1}\bigr)\, ds \\\nonumber
			&\qquad+\int_{t_m}^{t_{m+1}} \bigl(g(u(s)) - g(u_h^{m+1}), \varepsilon_{u}^{m+1}\bigr)\, ds \\\nonumber
			&\qquad+ \Bigl(\int_{t_m}^{t_{m+1}} B(u(s)) dW(s) - B(u_h^m)\Delta W_m,\varepsilon_{u}^{m+1}\Bigr)\\\nonumber
			&:= I + II + III + IV + V.
		\end{align}
		
		Now, we begin to estimate the right-hand side of \eqref{eq_3.28}. First, denote 
		\begin{align*}
			\theta_{u}(t) = u(t) - \cP_{u} u(t),\qquad \theta_{\vsigma}(t) = \vsigma(t) - \cP_{\vsigma} \vsigma(t).
		\end{align*}
		 Using Cauchy-Schwarz's inequality and \eqref{ineq_interpolation}, we get
		\begin{align}
			I &=-  \int_{t_m}^{t_{m+1}}\bigl(\nab(\theta_u^{m+1} - \theta_{u}(s)),\nab\varepsilon_u^{m+1}\bigr)\, ds -  \int_{t_m}^{t_{m+1}}\bigl(\nab\theta_{u}(s),\nab\varepsilon_u^{m+1}\bigr)\, ds\\\nonumber&\leq \frac{ k}{16} \|\nab \varepsilon_u^{m+1}\|^2_{L^2} + Ch^2 \int_{t_{m}}^{t_{m+1}}\|u(s)\|^2_{H^2}\, ds + C\int_{t_{m}}^{t_{m+1}}\|\nab(u(t_{m+1}) - u(s))\|^2_{L^2}\, ds.
		\end{align}
		
		Using the Cauchy-Schwarz inequality, we obtain
		\begin{align*}
			II &\leq C\int_{t_m}^{t_{m+1}}\|\nab(u(t_{m+1}) - u(s))\|^2_{L^2}\, ds + \frac{ k}{16}\|\nab \varepsilon_u^{m+1}\|^2_{L^2}.
		\end{align*}
		
		Now, we turn to estimate $III$ by adding and subtracting the term $u(t_{m+1})\vsigma(t_{m+1})$:
		\begin{align}
			III &= \chi \int_{t_m}^{t_{m+1}} \bigl(u(s)\vsigma(s) - u(t_{m+1})\vsigma(t_{m+1}),\nab \varepsilon_u^{m+1}\bigr)\, ds \\\nonumber
			&\qquad+ \chi k \bigl(u(t_{m+1})\vsigma(t_{m+1}) - u^{m+1}_h\vsigma^{m+1}_h,\nab \varepsilon_u^{m+1}\bigr)\\\nonumber
			&:= III_1 + III_2,
		\end{align}
		where the terms $III_1$ and $III_2$ are estimated as follows. By using the Cauchy-Schwarz inequality and then the Young inequality, we have
		\begin{align*}
			III_1 &=\chi \int_{t_m}^{t_{m+1}}[\bigl(u(s)[\vsigma(s) - \vsigma(t_{m+1})] + [u(s) - u(t_{m+1})]\vsigma(t_{m+1}), \nab \varepsilon_{u}^{m+1}\bigr)]\, ds \\\nonumber
			&\leq C\int_{t_m}^{t_{m+1}}[\|u(s)[\vsigma(s) - \vsigma(t_{m+1})]\|^2_{L^2} + \|[u(s) - u(t_{m+1})]\vsigma(t_{m+1})\|^2_{L^2}]\, ds \\\nonumber
			&\qquad\qquad\qquad+ \frac{k}{16}\|\nab \varepsilon_u^{m+1}\|^2_{L^2}\\\nonumber
			&\leq C\int_{t_m}^{t_{m+1}}[\|u(s)\|^2_{L^4}\|[\vsigma(s) - \vsigma(t_{m+1})]\|^2_{L^4} + \|[u(s) - u(t_{m+1})]\|^2_{L^4}\|\vsigma(t_{m+1})\|^2_{L^4}]\, ds \\\nonumber
			&\qquad\qquad\qquad+ \frac{k}{16}\|\nab \varepsilon_u^{m+1}\|^2_{L^2}.
		\end{align*}

		Now, we turn to control $III_2$ by adding and subtracting $u(t_{m+1})\vsigma_h^{m+1}$ to get
		\begin{align*}
			III_2 &= \chi k \bigl(u(t_{m+1})e_{\vsigma}^{m+1},\nab \varepsilon_u^{m+1}\bigr) + \chi k \bigl(e_{u}^{m+1}\vsigma_h^{m+1},\nab \varepsilon_u^{m+1}\bigr)\\\nonumber
			&= \chi k \bigl(u(t_{m+1})\varepsilon_{\vsigma}^{m+1},\nab \varepsilon_u^{m+1}\bigr) + \chi k \bigl(\varepsilon_{u}^{m+1}\vsigma_h^{m+1},\nab \varepsilon_u^{m+1}\bigr)\\\nonumber
			&\qquad+ \chi k \bigl(u(t_{m+1})\theta_{\vsigma}^{m+1},\nab \varepsilon_u^{m+1}\bigr) + \chi k \bigl(\theta_{u}^{m+1}\vsigma_h^{m+1},\nab \varepsilon_u^{m+1}\bigr)\\\nonumber
			&:= III_{2,1} + III_{2,2} + III_{2,3} + III_{2,4}.
		\end{align*}	
		
		To estimate $III_{2,1}$, we have to control $\varepsilon_{\vsigma}^{m+1}$ as follows.	Subtracting \eqref{eq_sigma} to \eqref{eq_discrete_sigma}, we obtain
		\begin{align}\label{eq_3.31}
			&	\bigl(e_{\vsigma}^{m+1}, \vvarphi_h\bigr) + \bigl(\nab\cdot e_{\vsigma}^{m+1},\nab\cdot\vvarphi_h\bigr) + \bigl(rot\, e_{\vsigma}^{m+1},\, rot\, \vvarphi_h\bigr) \\\nonumber
			&= -\bigl(u(t_{m+1}) - u(t_m),\nab\cdot\vvarphi_h\bigr) - \bigl(e_{u}^m,\nab\cdot\vvarphi_h\bigr).
		\end{align}
		Taking $\vvarphi_h = \varepsilon_{\vsigma}^{m+\frac12}$ in \eqref{eq_3.31} and using the orthogonality of $\cP_{\vsigma}$ from \eqref{Ps}, we arrive at
		\begin{align}\label{eq3.32}
			\|\varepsilon_{\vsigma}^{m+1}\|^2_{H^1} & \leq C\bigl[\|u(t_{m+1}) - u(t_m)\|^2_{L^2} \bigr] + C\|e_u^{m}\|^2_{L^2}\\\nonumber
			& \leq C\bigl[\|u(t_{m+1}) - u(t_m)\|^2_{L^2} \bigr] + C\|\varepsilon_u^{m}\|^2_{L^2} + Ch^2\| u(t_{m})\|^2_{H^1}.
		\end{align}
		
		Using  \eqref{eq3.32} and the Ladyzhenskaya inequality to bound $III_{2,1}$, we obtain 
		\begin{align}
			III_{2,1} &\leq \chi k \|u(t_{m+1})\|_{L^4}\|\varepsilon_{\vsigma}^{m+1}\|_{L^4}\|\nab \varepsilon_{u}^{m+1}\|_{L^2}\\\nonumber
			&\leq \frac{ k}{16}\|\nab \varepsilon_{u}^{m+1}\|^2_{L^2} + Ck\|u(t_{m+1})\|^2_{L^4} \|\varepsilon_{\vsigma}^{m+1}\|^2_{H^1}\\\nonumber
			&\leq  \frac{ k}{16}\|\nab \varepsilon_{u}^{m+1}\|^2_{L^2} + Ck\|u(t_{m+1})\|^2_{L^4}\|u(t_{m+1}) - u(t_m)\|^2_{L^2}  \\\nonumber
			&\qquad+ Ck\|u(t_{m+1})\|^2_{L^4}\|\varepsilon_u^{m}\|^2_{L^2} + Ckh^2\|u(t_{m+1})\|^2_{L^4}\|u(t_{m})\|^2_{H^1}.
		\end{align}

		Next, using the Cauchy-Schwarz inequality, the Ladyzhenskaya inequality and then Lemma \ref{lemma_fem_stability}, we have
		\begin{align}
			III_{2,2} &=-\chi k ((\varepsilon_{u}^{m+1})^2,\nab\cdot\vsigma_h^{m+1}) \\\nonumber
			&\leq \chi k \|\varepsilon_{u}^{m+1}\|^2_{L^4}\|\vsigma_h^{m+1}\|_{H^1}\\\nonumber
			&\leq Ck \|\varepsilon_{u}^{m+1}\|_{L^2}\|\nab\varepsilon_{u}^{m+1}\|_{L^2}\|u_h^m\|_{L^2}\\\nonumber
			&\leq \frac{k}{16} \|\nab \varepsilon_{u}^{m+1}\|^2_{L^2} + Ck\|u_h^m\|^2_{L^2}\|\varepsilon_{u}^{m+1}\|^2_{L^2}.
		\end{align}
		
		To estimate the terms $III_{2,3}$ and $III_{2,4}$, we use the Cauchy-Schwarz inequality, the inequality \eqref{ineq_interpolation} as follows:
		\begin{align}
			III_{2,3} + III_{2,4} &\leq \frac{k}{16}\|\nab \varepsilon_{u}^{m+1}\|^2_{L^2} + Ck \|u(t_{m+1})\|^2_{L^4}\|\theta_{\vsigma}^{m+1}\|^2_{L^4} \\\nonumber
			&\quad+ C\int_{t_{m}}^{t_{m+1}} \|\theta_{u}^{m+1}- \theta_{u}(s)\|^2_{L^4}\|\vsigma_h^{m+1}\|^2_{L^4}\, ds  + C\int_{t_{m}}^{t_{m+1}} \| \theta_{u}(s)\|^2_{L^4}\|\vsigma_h^{m+1}\|^2_{L^4}\, ds\\\nonumber
			&\leq \frac{ k}{16}\|\nab \varepsilon_{u}^{m+1}\|^2_{L^2} + Ckh^3\|u(t_{m+1})\|^2_{L^4}\|\vsigma(t_{m+1})\|^2_{H^2} \\\nonumber
			&\quad+ C h^3\int_{t_{m}}^{t_{m+1}}\|u_h^{m}\|^2_{L^2}\|u(s)\|^2_{H^2}\, ds + C\int_{t_{m}}^{t_{m+1}} \|\nab(u(t_{m+1})- u(s))\|^2_{L^2}\|u_h^{m}\|^2_{L^2}\, ds.
		\end{align}

		Next, we proceed to control $IV$ as follows. Using the fact that $(g(a) - g(b))(a-b) \leq (a-b)^2$, we have
		\begin{align}
			IV &= \int_{t_{m}}^{t_{m+1}} \bigl(g(u(s)) - g(u(t_{m+1})), \varepsilon_{u}^{m+1}\bigr)\, ds + k \bigl(g(u(t_{m+1})) - g(\cP_{u}u(t_{m+1})), \varepsilon_{u}^{m+1}\bigr) \\\nonumber
			&\qquad+ k \bigl(g(\cP_{u}u(t_{m+1})) - g(u_h^{m+1}), \varepsilon_{u}^{m+1}\bigr) \\\nonumber
			&\leq C\int_{t_{m}}^{t_{m+1}}\|g(u(s)) - g(u(t_{m+1}))\|^2_{L^2}\, ds + Ck\|g(u(t_{m+1})) - g(\cP_{u}u(t_{m+1}))\|^2_{L^2} \\\nonumber
			&\qquad+ k \|\varepsilon_{u}^{m+1}\|^2_{L^2} +  k \|\varepsilon_{u}^{m+1}\|^2_{L^2}.
		\end{align}
		
		Using the identity $a^3-b^3=(a-b)(a^2+ab+b^2),$
		we have
		\begin{align}
			|g(a)-g(b)|
			\le
			|a-b|\left(1+|a|^2+|b|^2\right).
		\end{align}
		
		With this and then using H\"older's inequality and the Sobolev embedding
		$H^1(D)\hookrightarrow L^q(D)$ holds for all $1\le q<\infty$, we obtain
		\begin{align}
			IV
			&\le
			\int_{t_m}^{t_{m+1}}
			\|u(s)-u(t_{m+1})\|_{L^2}
			\left\|1+|u(s)|^2+|u(t_{m+1})|^2\right\|_{L^4}\|\varepsilon_{u}^{m+1}\|_{L^4}\,ds
			\nonumber\\
			&\quad
			+
			k
			\|\theta_{u}^{m+1}\|_{L^2}
			\left\|1+|u(t_{m+1})|^2+|\mathcal P_u u(t_{m+1})|^2\right\|_{L^4}\|\varepsilon_{u}^{m+1}\|_{L^4}
			+
			k\|\varepsilon_u^{m+1}\|_{L^2}^{2}\\\nonumber
				&\le
			C\int_{t_m}^{t_{m+1}}
			\|u(s)-u(t_{m+1})\|_{L^2}
			\left\|1+|u(s)|^2+|u(t_{m+1})|^2\right\|_{L^4}\|\varepsilon_{u}^{m+1}\|_{H^1}\,ds
			\nonumber\\
			&\quad
			+
			Ck
			\|\theta_{u}^{m+1}\|_{L^2}
			\left\|1+|u(t_{m+1})|^2+|\mathcal P_u u(t_{m+1})|^2\right\|_{L^4}\| \varepsilon_{u}^{m+1}\|_{H^1}
			+
			k\|\varepsilon_u^{m+1}\|_{L^2}^{2}\\\nonumber
			&\le
			C\int_{t_m}^{t_{m+1}}
			\|u(s)-u(t_{m+1})\|^2_{L^2} (1 + \|u(s)\|^4_{H^1} + \|u(t_{m+1})\|^4_{H^1})\,ds
			\nonumber\\
			&\quad
			+
			Ckh^2
			\|u(t_{m+1})\|^2_{H^1}(1 + \|u(t_{m+1})\|^4_{H^1}) + \frac{k}{16}(\|\varepsilon_{u}^{m+1}\|^2_{L^2} + \|\nab \varepsilon_{u}^{m+1}\|^2_{L^2})
			+
			k\|\varepsilon_u^{m+1}\|_{L^2}^{2}\\\nonumber
		\end{align}
		
		Finally, we estimate the noise term $V$ as follows. 
		\begin{align*}
			V &=  \left(\int_{t_m}^{t_{m+1}} (B(u(s)) - B(u_h^m))\, dW(s), \varepsilon_u^{m+1} - \varepsilon_u^m\right) +  \left(\int_{t_m}^{t_{m+1}} (B(u(s)) - B(u_h^m))\, dW(s),  \varepsilon_u^m\right)\\\nonumber
			&=  \left(\int_{t_m}^{t_{m+1}} (B(u(s)) - B(u(t_{m})))\, dW(s), \varepsilon_u^{m+1} - \varepsilon_u^m\right) +\bigl((B(u(t_m)) - B(u_h^m))\Delta W_m, \varepsilon_u^{m+1} - \varepsilon_u^m\bigr) \\\nonumber
			&\qquad+  \left(\int_{t_m}^{t_{m+1}} (B(u(s)) - B(u^m_h))\, dW(s),  \varepsilon_u^m\right)\\\nonumber
			&\leq  2\left\|\int_{t_m}^{t_{m+1}} (B(u(s)) - B(u(t_{m})))\, dW(s)\right\|^2_{L^2} +2\left\|(B(u(t_m)) - B(u_h^m))\Delta W_m\right\|^2_{L^2} + \frac14\|\varepsilon_u^{m+1} - \varepsilon_u^{m}\|^2_{L^2} \\\nonumber
			&\qquad+  \left(\int_{t_m}^{t_{m+1}} (B(u(s)) - B(u^m_h))\, dW(s),  \varepsilon_u^m\right).
		\end{align*}
		
		Substituting all estimates from $I, ..., V$ into \eqref{eq_3.28}, we obtain
	\begin{align}\label{eq3.33}
		&\frac12\left[\|\varepsilon^{m+1}_u\|^2_{L^2}
		-\|\varepsilon_u^m\|^2_{L^2}\right]
		+\frac14\|\varepsilon_u^{m+1}-\varepsilon_u^m\|^2_{L^2}
		+\frac{9k}{16}\|\nabla\varepsilon_u^{m+1}\|^2_{L^2}
		\nonumber\\
		&\le Ch^2 \int_{t_{m}}^{t_{m+1}}\|u(s)\|^2_{H^2}\, ds + C\int_{t_{m}}^{t_{m+1}}\|\nab(u(t_{m+1}) - u(s))\|^2_{L^2}\, ds\\\nonumber
		&\qquad +C\int_{t_m}^{t_{m+1}}[\|u(s)\|^2_{L^4}\|[\vsigma(s) - \vsigma(t_{m+1})]\|^2_{L^4} + \|[u(s) - u(t_{m+1})]\|^2_{L^4}\|\vsigma(t_{m+1})\|^2_{L^4}]\, ds \\\nonumber
		&\qquad+ Ck\|u(t_{m+1})\|^2_{L^4}\|u(t_{m+1}) - u(t_m)\|^2_{L^2}  + Ck\|u_h^m\|^2_{L^2}\|\varepsilon_{u}^{m+1}\|^2_{L^2}\\\nonumber
		&\qquad+ Ck\|u(t_{m+1})\|^2_{L^4}\|\varepsilon_u^{m}\|^2_{L^2} + Ckh^2\|u(t_{m+1})\|^2_{L^4}\|u(t_{m})\|^2_{H^1} \\\nonumber
		&\qquad+  Ckh^3\|u(t_{m+1})\|^2_{L^4}\|\vsigma(t_{m+1})\|^2_{H^2} \\\nonumber
		&\qquad+ C h^3\int_{t_{m}}^{t_{m+1}}\|u_h^{m}\|^2_{L^2}\|u(s)\|^2_{H^2}\, ds + C\int_{t_{m}}^{t_{m+1}} \|\nab(u(t_{m+1})- u(s))\|^2_{L^2}\|u_h^{m}\|^2_{L^2}\, ds\\\nonumber
		&\qquad + C\int_{t_m}^{t_{m+1}}
		\|u(s)-u(t_{m+1})\|^2_{L^2} (1 + \|u(s)\|^4_{H^1} + \|u(t_{m+1})\|^4_{H^1})\,ds
		\\\nonumber
		&\qquad
		+
		Ckh^2
		\|u(t_{m+1})\|^2_{H^1}(1 + \|u(t_{m+1})\|^4_{H^1}) + \frac{17k}{16}\|\varepsilon_{u}^{m+1}\|^2_{L^2}
		\\\nonumber
		&\qquad+ 2\left\|\int_{t_m}^{t_{m+1}} (B(u(s)) - B(u(t_{m})))\, dW(s)\right\|^2_{L^2} +2\left\|(B(u(t_m)) - B(u_h^m))\Delta W_m\right\|^2_{L^2}  \\\nonumber
		&\qquad+  \left(\int_{t_m}^{t_{m+1}} (B(u(s)) - B(u^m_h))\, dW(s),  \varepsilon_u^m\right).
	\end{align}
	
	Now, let us consider $\Omega_{\rho, m}$ from \eqref{omega} with $\rho=\frac{\ln\ln(1/k)}{C}$. 
Multiplying \eqref{eq3.33} by $\mathbf 1_{\Omega_{\rho,m}}$, 
we obtain
\begin{align}
	&\mathbf 1_{\Omega_{\rho,m}}
	\left[
	\|\varepsilon_u^{m+1}\|_{L^2}^{2}
	-
	\|\varepsilon_u^{m}\|_{L^2}^{2}
	+
	\|\varepsilon_u^{m+1}-\varepsilon_u^{m}\|_{L^2}^{2}
	+
	 k\|\nabla\varepsilon_u^{m+1}\|_{L^2}^{2}
	\right]
	\nonumber\\
	&\le C\mathbf 1_{\Omega_{\rho,m}} \Biggl[h^2 \int_{t_{m}}^{t_{m+1}}\|u(s)\|^2_{H^2}\, ds + \int_{t_{m}}^{t_{m+1}}\|\nab(u(t_{m+1}) - u(s))\|^2_{L^2}\, ds\\\nonumber
	&\qquad +\int_{t_m}^{t_{m+1}}\left[\|u(s)\|^2_{L^4}\|\vsigma(s) - \vsigma(t_{m+1})\|^2_{L^4} + \|u(s) - u(t_{m+1})\|^2_{L^4}\|u(t_{m+1})\|^2_{L^2}\right]\, ds \\\nonumber
	&\qquad+ k\|u(t_{m+1})\|^2_{L^4}\|u(t_{m+1}) - u(t_m)\|^2_{L^2}  + k\rho\|\varepsilon_{u}^{m+1}\|^2_{L^2}\\\nonumber
	&\qquad+ k\sqrt{\rho}\|\varepsilon_u^{m}\|^2_{L^2} + k \|u(t_{m+1}) - u(t_m)\|^2_{L^4}\|\varepsilon_{u}^m\|^2_{L^2} + Ckh^2\|u(t_{m+1})\|^2_{L^4}\|u(t_{m})\|^2_{H^1} \\\nonumber
	&\qquad+  kh^3\|u(t_{m+1})\|^2_{L^4}\|u(t_{m+1})\|^2_{H^1} +  h^3\int_{t_{m}}^{t_{m+1}}\rho\|u(s)\|^2_{H^2}\, ds \\\nonumber
	&\qquad+ \rho\int_{t_{m}}^{t_{m+1}} \|\nab(u(t_{m+1})- u(s))\|^2_{L^2}\, ds\\\nonumber
	&\qquad + \int_{t_m}^{t_{m+1}}
	\|u(s)-u(t_{m+1})\|^2_{L^2} (1 + \|u(s)\|^4_{H^1} + \|u(t_{m+1})\|^4_{H^1})\,ds
	\\\nonumber
	&\qquad
	+
	Ckh^2
	\|u(t_{m+1})\|^2_{H^1}(1 + \|u(t_{m+1})\|^4_{H^1}) + \frac{17k}{16}\|\varepsilon_{u}^{m+1}\|^2_{L^2}
	\\\nonumber
	&\qquad+ 2\left\|\int_{t_m}^{t_{m+1}} (B(u(s)) - B(u(t_{m})))\, dW(s)\right\|^2_{L^2} +2\left\|(B(u(t_m)) - B(u_h^m))\Delta W_m\right\|^2_{L^2}  \\\nonumber
	&\qquad+  \left(\int_{t_m}^{t_{m+1}} (B(u(s)) - B(u^m_h))\, dW(s),  \varepsilon_u^m\right)\Biggr]
\end{align}

Since $\Omega_{\rho,m+1}\subset\Omega_{\rho,m},$
we have
\[
\mathbf 1_{\Omega_{\rho,m}}
\ge
\mathbf 1_{\Omega_{\rho,m+1}} .
\]
Hence,
\begin{align}
	&\mathbf 1_{\Omega_{\rho,m}}
	\left(
	\|\varepsilon_u^{m+1}\|_{L^2}^{2}
	-
	\|\varepsilon_u^m\|_{L^2}^{2}
	\right)
	\ge
	\mathbf 1_{\Omega_{\rho,m+1}}
	\|\varepsilon_u^{m+1}\|_{L^2}^{2}
	-
	\mathbf 1_{\Omega_{\rho,m}}
	\|\varepsilon_u^m\|_{L^2}^{2}.
\end{align}

Taking expectation and using the martingale property of the It\^o integral, the It\^o isometry and the assumption \eqref{Assump_Lipschitz}, we obtain
\begin{align}
	&
	\mE\left[
	\mathbf 1_{\Omega_{\rho,m+1}}
	\|\varepsilon_u^{m+1}\|_{L^2}^{2}
	\right]
	-
	\mE\left[
	\mathbf 1_{\Omega_{\rho,m}}
	\|\varepsilon_u^{m}\|_{L^2}^{2}
	\right]
	+
	 k
	\mE\left[
	\mathbf 1_{\Omega_{\rho,m}}
	\|\nabla\varepsilon_u^{m+1}\|_{L^2}^{2}
	\right]
	\nonumber\\
	&\le
	C\mE\left[
	\mathbf 1_{\Omega_{\rho,m}}
	\mathcal{E}_m
	\right]
	+
	Ck(1+ \sqrt{\rho}+ \rho)
	\mE\left[
	\mathbf 1_{\Omega_{\rho,m}}
	\|\varepsilon_u^{m+1}\|_{L^2}^{2} + \mathbf 1_{\Omega_{\rho,m}}
	\|\varepsilon_u^{m}\|_{L^2}^{2} 
	\right] \\\nonumber
	&\qquad+ C\mE\left[\mathbf 1_{\Omega_{\rho,m}}\left\|\int_{t_m}^{t_{m+1}} (B(u(s)) - B(u(t_{m})))\, dW(s)\right\|^2_{L^2} \right]\\\nonumber
	&\qquad+C\mE\left[\mathbf 1_{\Omega_{\rho,m}}\left\|(B(u(t_m)) - B(u_h^m))\Delta W_m\right\|^2_{L^2}\right] \\\nonumber
	&=
	C\mE\left[
	\mathbf 1_{\Omega_{\rho,m}}
	\mathcal{E}_m
	\right]
	+
	Ck(1+ \sqrt{\rho}+ \rho)
	\mE\left[
	\mathbf 1_{\Omega_{\rho,m}}
	\|\varepsilon_u^{m+1}\|_{L^2}^{2} + \mathbf 1_{\Omega_{\rho,m}}
	\|\varepsilon_u^{m}\|_{L^2}^{2} 
	\right] \\\nonumber
	&\qquad+ C\mE\left[\mathbf 1_{\Omega_{\rho,m}}\int_{t_m}^{t_{m+1}} \|B(u(s)) - B(u(t_{m}))\|^2_{L^2}\, ds \right]\\\nonumber
	&\qquad+Ck\mE\left[\mathbf 1_{\Omega_{\rho,m}}\left\|B(u(t_m)) - B(u_h^m)\right\|^2_{L^2}\right]\\\nonumber
	&\leq
	C\mE\left[
	\mathbf 1_{\Omega_{\rho,m}}
	\mathcal{E}_m
	\right]
	+
	Ck(1+ \sqrt{\rho}+ \rho)
	\mE\left[
	\mathbf 1_{\Omega_{\rho,m}}
	\|\varepsilon_u^{m+1}\|_{L^2}^{2} + \mathbf 1_{\Omega_{\rho,m}}
	\|\varepsilon_u^{m}\|_{L^2}^{2} 
	\right] \\\nonumber
	&\qquad+ CC_B\mE\left[\mathbf 1_{\Omega_{\rho,m}}\int_{t_m}^{t_{m+1}} \|u(s) - u(t_{m})\|^2_{L^2}\, ds \right]\\\nonumber
	&\qquad+CC_Bk\mE\left[\mathbf 1_{\Omega_{\rho,m}}\left\|\varepsilon_{u}^m\right\|^2_{L^2}\right] + CC_Bkh^2\mE\left[\mathbf 1_{\Omega_{\rho,m}}\left\|u(t_m)\right\|^2_{H^1}\right]
\end{align}
where \begin{align*}
	\mathcal{E}_m &= h^2 \int_{t_{m}}^{t_{m+1}}\|u(s)\|^2_{H^2}\, ds + \int_{t_{m}}^{t_{m+1}}\|\nab(u(t_{m+1}) - u(s))\|^2_{L^2}\, ds\\\nonumber
	&\qquad +\int_{t_m}^{t_{m+1}}\left[\|u(s)\|^2_{L^4}\|\vsigma(s) - \vsigma(t_{m+1})\|^2_{L^4} + \|u(s) - u(t_{m+1})\|^2_{L^4}\|u(t_{m+1})\|^2_{L^2}\right]\, ds \\\nonumber
	&\qquad+ k\|u(t_{m+1})\|^2_{L^4}\|u(t_{m+1}) - u(t_m)\|^2_{L^2}  \\\nonumber
	&\qquad+ 4\rho k \|u(t_{m+1}) - u(t_m)\|^2_{L^4} + Ckh^2\|u(t_{m+1})\|^2_{L^4}\|u(t_{m})\|^2_{H^1} \\\nonumber
	&\qquad+  kh^3\|u(t_{m+1})\|^2_{L^4}\|u(t_{m+1})\|^2_{H^1} +  h^3\int_{t_{m}}^{t_{m+1}}\rho\|u(s)\|^2_{H^2}\, ds \\\nonumber
	&\qquad+ \rho\int_{t_{m}}^{t_{m+1}} \|\nab(u(t_{m+1})- u(s))\|^2_{L^2}\, ds\\\nonumber
	&\qquad + \int_{t_m}^{t_{m+1}}
	\|u(s)-u(t_{m+1})\|^2_{L^2} (1 + \|u(s)\|^4_{H^1} + \|u(t_{m+1})\|^4_{H^1})\,ds
	\nonumber\\
	&\qquad
	+
	Ckh^2
	\|u(t_{m+1})\|^2_{H^1}(1 + \|u(t_{m+1})\|^4_{H^1}).
\end{align*}

For $k>0$ sufficiently small such that
\[
Ck(1+\sqrt{\rho} +\rho)\le \frac12,
\]
the last term on the right-hand side can be absorbed into the left-hand side.
Therefore,
\begin{align}
	&
	\mE\left[
	\mathbf 1_{\Omega_{\rho,m+1}}
	\|\varepsilon_u^{m+1}\|_{L^2}^{2}
	\right]
	-
	\mE\left[
	\mathbf 1_{\Omega_{\rho,m}}
	\|\varepsilon_u^{m}\|_{L^2}^{2}
	\right]
	+
	 k
	\mE\left[
	\mathbf 1_{\Omega_{\rho,m}}
	\|\nabla\varepsilon_u^{m+1}\|_{L^2}^{2}
	\right]
	\nonumber\\
	&\le
	C\mE\left[
	\mathcal{E}_m
	\right]
	+
	Ck(1+\sqrt{\rho}+\rho)
	\mE\left[
	\mathbf 1_{\Omega_{\rho,m}}
	\|\varepsilon_u^{m}\|_{L^2}^{2}
	\right].
\end{align}

Summing the above inequality from $m=0$ to $\ell$ gives
\begin{align}
	&
	\mE\left[
	\mathbf 1_{\Omega_{\rho,\ell+1}}
	\|\varepsilon_u^{\ell+1}\|_{L^2}^{2}
	\right]
	+
	k
	\sum_{m=0}^{\ell}
	\mE\left[
	\mathbf 1_{\Omega_{\rho,m}}
	\|\nabla\varepsilon_u^{m+1}\|_{L^2}^{2}
	\right]
	\nonumber\\
	&\le
	C\sum_{m=0}^{\ell}
	\mE\left[
	\mathcal{E}_m
	\right]
	+
	Ck(1+ \sqrt{\rho}+ \rho)
	\sum_{m=0}^{\ell}
	\mE\left[
	\mathbf 1_{\Omega_{\rho,m}}
	\|\varepsilon_u^{m}\|_{L^2}^{2}
	\right].
\end{align}

Now, we analyze the terms in $\mathcal{E}_m$ as follows: 
\begin{align*}
\sum_{m=0}^{\ell}
\mE\left[
\mathcal{E}_m
\right] &=T_1 + T_2,
\end{align*}
where 
\begin{align*}
	T_1 &=\sum_{m=0}^{\ell}\mE\Biggl[h^2 \int_{t_{m}}^{t_{m+1}}\|u(s)\|^2_{H^2}\, ds + Ckh^2\|u(t_{m+1})\|^2_{L^4}\|u(t_{m})\|^2_{H^1} \\\nonumber
	&\qquad+  kh^3\|u(t_{m+1})\|^2_{L^4}\|u(t_{m+1})\|^2_{H^1} +  h^3\int_{t_{m}}^{t_{m+1}}\rho\|u(s)\|^2_{H^2}\, ds 
	\\\nonumber
	&\qquad
	+
	Ckh^2
	\|u(t_{m+1})\|^2_{H^1}(1 + \|u(t_{m+1})\|^4_{H^1})\Biggr],\\\nonumber
	T_2&=\sum_{m=0}^{\ell}
	\mE\Biggl[ \int_{t_{m}}^{t_{m+1}}\|\nab(u(t_{m+1}) - u(s))\|^2_{L^2}\, ds\\\nonumber
	&\qquad +\int_{t_m}^{t_{m+1}}\left[\|u(s)\|^2_{L^4}\|\vsigma(s) - \vsigma(t_{m+1})\|^2_{L^4} + \|u(s) - u(t_{m+1})\|^2_{L^4}\|u(t_{m+1})\|^2_{L^2}\right]\, ds \\\nonumber
	&\qquad+ k\|u(t_{m+1})\|^2_{L^4}\|u(t_{m+1}) - u(t_m)\|^2_{L^2}  + 4\rho k \|u(t_{m+1}) - u(t_m)\|^2_{L^4}  \\\nonumber
	&\qquad+ \rho\int_{t_{m}}^{t_{m+1}} \|\nab(u(t_{m+1})- u(s))\|^2_{L^2}\, ds\\\nonumber
	&\qquad + \int_{t_m}^{t_{m+1}}
	\|u(s)-u(t_{m+1})\|^2_{L^2} (1 + \|u(s)\|^4_{H^1} + \|u(t_{m+1})\|^4_{H^1})\,ds
	\Biggr].
\end{align*}

Using the Sobolev embedding inequality $\|u\|_{L^p} \leq C\|u\|_{H^1}$ for any $1\geq p <\infty$, and the hypothesis of the solution $u\in L^2(\Omega;L^{\infty}(0,T;H^2(D)))\cap L^6(\Ome;L^{\infty}(0,T;H^1(D)))$, we obtain
\begin{align*}
	T_1 & \leq Ch^2\mE\left[\sup_{s\in[0,T]}\|u(s)\|^2_{H^2} + \sup_{s\in[0,T]}\|u(s)\|^6_{H^1}\right] \leq Ch^2.
\end{align*}

Using the Cauchy-Schwarz inequality and Lemma \ref{Lemma_Holder}, we also obtain
\begin{align*}
	T_2&\leq \sum_{m=0}^{\ell}
\Biggl( \int_{t_{m}}^{t_{m+1}}\mE\left[\|\nab(u(t_{m+1}) - u(s))\|^2_{L^2}\right]\, ds\\\nonumber
	&\qquad +\int_{t_m}^{t_{m+1}}\left(\mE\left[\|u(s)\|^4_{L^4}\right]\right)^{\frac12}\left(\mE\left[\|\vsigma(s) - \vsigma(t_{m+1})\|^4_{L^4}\right]\right)^{\frac12}\, ds \\\nonumber
	&\qquad+\int_{t_m}^{t_{m+1}} \left(\mE\left[\|u(s) - u(t_{m+1})\|^4_{L^4}\right]\right)^{\frac12} \left(\mE\left[\|u(t_{m+1})\|^4_{L^2}\right]\right)^{\frac12}\, ds \\\nonumber
	&\qquad+ k\left(\mE\left[\|u(t_{m+1})\|^4_{L^4}\right]\right)^{\frac12}\left(\mE\left[\|u(t_{m+1}) - u(t_m)\|^4_{L^2}\right]\right)^{\frac12}  \\\nonumber
	&\qquad+ 4\rho k \mE\left[\|u(t_{m+1}) - u(t_m)\|^2_{L^4}\right]+ \rho\int_{t_{m}}^{t_{m+1}}\mE\left[\|\nab(u(t_{m+1})- u(s))\|^2_{L^2}\right]\, ds\\\nonumber
	&\qquad + \int_{t_m}^{t_{m+1}}
	\left(\mE\left[\|u(s)-u(t_{m+1})\|^4_{L^2}\right]\right)^{\frac12} \left(\mE\left[(1 + \|u(s)\|^4_{H^1} + \|u(t_{m+1})\|^4_{H^1})^2\right]\right)^{\frac12}\,ds
	\Biggr)\\\nonumber
	&\leq Ck + Ck \sup_{s\in[0,T]}\left\{\left(\mE\left[\|u(s)\|^4_{L^4}\right]\right)^{\frac12} + \left(\mE\left[\|u(s)\|^8_{L^4}\right]\right)^{\frac12}\right\} \leq Ck.
\end{align*}

Hence,
\begin{align}
	&
	\mE\left[
	\mathbf 1_{\Omega_{\rho,\ell+1}}
	\|\varepsilon_u^{\ell+1}\|_{L^2}^{2}
	\right]
	+
	 k
	\sum_{m=0}^{\ell}
	\mE\left[
	\mathbf 1_{\Omega_{\rho,m}}
	\|\nabla\varepsilon_u^{m+1}\|_{L^2}^{2}
	\right]
	\nonumber\\
	&\le
	C(h^2+k)
	+
	Ck(1+\sqrt{\rho}+\rho)
	\sum_{m=0}^{\ell}
	\mE\left[
	\mathbf 1_{\Omega_{\rho,m}}
	\|\varepsilon_u^{m}\|_{L^2}^{2}
	\right].
\end{align}

Applying the discrete Gronwall inequality yields
\begin{align}
	\mE\left[
	\mathbf 1_{\Omega_{\rho,\ell +1}}
	\|\varepsilon_u^{\ell +1}\|_{L^2}^{2}
	\right]
	+
	k
	\sum_{m=0}^{\ell}
	\mE\left[
	\mathbf 1_{\Omega_{\rho,m}}
	\|\nabla\varepsilon_u^{m+1}\|_{L^2}^{2}
	\right]
	&\le
	Ce^{C(1+ \sqrt{\rho}+ \rho)T}(h^2+k)\\\nonumber
&	\le
	C\ln(1/k)(h^2+k).
\end{align}

The proof is complete by using the triangle inequality and \eqref{ineq_interpolation}.

	\end{proof}
	
	\begin{corollary}\label{cor:probability}
		Under the assumptions of Theorem~\ref{Thm_global_error}, the numerical solutions
		$\{(u_h^m,\vsigma_h^m,v_h^m)\}_{m=1}^M$
		converge in probability. More precisely, for every
		$\alpha\in(0,\frac12)$, there holds
		\begin{align*}
			\lim_{C\to\infty}
			\lim_{k,h\to0}
			\mathbb P
			\Bigg(
			\max_{1\le m\le M}
			\|u(t_m)-u_h^m\|_{L^2}
			+
			\Bigg(
			k\sum_{m=1}^{M}
			\|\nabla(u(t_m)-u_h^m)\|_{L^2}^{2}
			\Bigg)^{\frac12}
			\ge
			C(k+h^2)^{\alpha}
			\Bigg)
			=0,
		\end{align*}
		provided that the localization parameter $\rho$ is chosen sufficiently large.
		Furthermore,
		\begin{align*}
			\lim_{C\to\infty}
			\lim_{k,h\to0}
			\mathbb P
			\Bigg(
			\max_{1\le m\le M}
			\|\vsigma(t_m)-\vsigma_h^m\|_{H^1}
			\ge
			C(k+h^2)^{\alpha}
			\Bigg)
			=0,
		\end{align*}
		and
		\begin{align*}
			\lim_{C\to\infty}
			\lim_{k,h\to0}
			\mathbb P
			\Bigg(
			\max_{1\le m\le M}
			\|v(t_m)-v_h^m\|_{L^2}
			\ge
			C(k+h^2)^{\alpha}
			\Bigg)
			=0.
		\end{align*}
	\end{corollary}
	\begin{proof}
		We only prove the first assertion, since the proofs of the remaining two assertions are identical.
		
		Let
		\[
		E_k:=
		\max_{1\le m\le M}
		\|u(t_m)-u_h^m\|_{L^2}
		+
		\left(
		k\sum_{m=1}^{M}
		\|\nabla(u(t_m)-u_h^m)\|_{L^2}^{2}
		\right)^{\frac12}.
		\]
		Then
		\begin{align*}
			\mathbb P
			\Bigl(
			E_k
			\ge
			C(k+h^2)^{\alpha}
			\Bigr)
			&\le
			\mathbb P
			\Bigl(
			E_k
			\ge
			C(k+h^2)^{\alpha},
			\,
			\Omega_{\rho,M}
			\Bigr)
			+
			\mathbb P(\Omega\setminus\Omega_{\rho,M}).
		\end{align*}
		Applying Chebyshev's inequality together with Theorem~\ref{Thm_global_error}, we obtain
		\begin{align*}
			\mathbb P
			\Bigl(
			E_k
			\ge
			C(k+h^2)^{\alpha},
			\,
			\Omega_{\rho,M}
			\Bigr)
			&\le
			\frac{
				\mathbb E
				\!\left[
				\mathbf 1_{\Omega_{\rho,M}}
				E_k^2
				\right]
			}
			{C^2(k+h^2)^{2\alpha}}
			\\
			&\le
			\frac{
				C\ln(1/k)(k+h^2)
			}
			{C^2(k+h^2)^{2\alpha}}
			\\
			&=
			\frac{C}{C^2}
			\ln(1/k)
			(k+h^2)^{1-2\alpha}.
		\end{align*}
		Since $\alpha<\frac12$,
		\[
		\lim_{k,h\to0}
		\ln(1/k)(k+h^2)^{1-2\alpha}
		=0.
		\]
		Therefore,
		\[
		\lim_{k,h\to0}
		\mathbb P
		\Bigl(
		E_k
		\ge
		C(k+h^2)^{\alpha}
		\Bigr)
		\le
		\mathbb P(\Omega\setminus\Omega_{\rho,M}).
		\]
		Finally, letting $\rho\to\infty$ yields
		\[
		\lim_{\rho\to\infty}
		\mathbb P(\Omega\setminus\Omega_{\rho,M})
		=0,
		\]
		and hence
		\[
		\lim_{C\to\infty}
		\lim_{k,h\to0}
		\mathbb P
		\Bigl(
		E_k
		\ge
		C(k+h^2)^{\alpha}
		\Bigr)
		=0.
		\]
		This proves the first assertion. The remaining two assertions follow by applying the error estimates
		\eqref{estimate_sigma_fem} and \eqref{estimate_c_fem} in exactly the same way.
	\end{proof}

	\section{Numerical experiments}\label{sec5}

In this section, we investigate the numerical performance of the proposed splitting finite element method and verify the theoretical error estimates established in Theorem~\ref{Thm_global_error}. Throughout all experiments in Tests 1 and 2, the computational domain is taken as $D=(0,1)^2\subset\mathbb{R}^2$, the final time is $T=1$, and the chemotactic sensitivity parameter is $\chi=1$. Unless otherwise specified, the multiplicative noise is chosen as $B(u)=\frac12u$. The stochastic forcing is driven by a real-valued Wiener process $W(t)$ in \eqref{eq1.1}. To ensure consistency across different temporal discretizations, all Brownian sample paths are generated using the finest time step $k_0=2^{-11}$ and subsequently reused on all coarser temporal grids. The spatial discretization is performed using continuous piecewise linear ($P_1$) finite element spaces for all unknowns, and statistical expectations are approximated by the standard Monte Carlo method with $J =400$ independent realizations. Unless otherwise specified, the initial condition in Tests~1 and~2 is chosen as $u_0(x,y)=\sin(\pi x)\sin(\pi y)$.

Since analytical solutions of \eqref{eq1.1} are unavailable, the numerical errors are computed by comparing the finite element approximation $(\vsigma_h^m(\omega_j),u_h^m(\omega_j),v_h^m(\omega_j))$ with a reference approximation
\[
(\vsigma_{\rm ref}^m(\omega_j),u_{\rm ref}^m(\omega_j),v_{\rm ref}^m(\omega_j)),
\]
computed on a sufficiently refined spatial mesh and temporal partition. The construction of the reference approximation will be specified in each individual experiment.

To quantify the numerical accuracy, we employ the following discrete approximations of the strong error norms appearing in Theorem~\ref{Thm_global_error}:
\begin{align*}
	L^2_{\omega}L^{\infty}_tL^2_x(u)
	&:=
	\Bigl(
	\mathbb E
	\Bigl[
	\max_{1\le m\le M}
	\|u(t_m)-u_h^m\|_{L^2}^2
	\Bigr]
	\Bigr)^{1/2}
	\\
	&\approx
	\left(
	\frac1J
	\sum_{j=1}^{J}
	\max_{1\le m\le M}
	\|u_{\rm ref}^m(\omega_j)-u_h^m(\omega_j)\|_{L^2}^2
	\right)^{1/2},
	\\[1ex]
	L^2_{\omega}L^{2}_tH^1_x(u)
	&:=
	\left(
	\mathbb E
	\left[
	k\sum_{m=1}^{M}
	\|\nabla(u(t_m)-u_h^m)\|_{L^2}^{2}
	\right]
	\right)^{1/2}
	\\
	&\approx
	\left(
	\frac1J
	\sum_{j=1}^{J}
	k\sum_{m=1}^{M}
	\|\nabla(u_{\rm ref}^m(\omega_j)-u_h^m(\omega_j))\|_{L^2}^{2}
	\right)^{1/2},
	\\[1ex]
	L^2_{\omega}L^{\infty}_tL^2_x(v)
	&:=
	\Bigl(
	\mathbb E
	\Bigl[
	\max_{1\le m\le M}
	\|v(t_m)-v_h^m\|_{L^2}^{2}
	\Bigr]
	\Bigr)^{1/2}
	\\
	&\approx
	\left(
	\frac1J
	\sum_{j=1}^{J}
	\max_{1\le m\le M}
	\|v_{\rm ref}^m(\omega_j)-v_h^m(\omega_j)\|_{L^2}^{2}
	\right)^{1/2},
	\\[1ex]
	L^2_{\omega}L^{\infty}_tH^1_x(\vsigma)
	&:=
	\Bigl(
	\mathbb E
	\Bigl[
	\max_{1\le m\le M}
	\|\vsigma(t_m)-\vsigma_h^m\|_{H^1}^{2}
	\Bigr]
	\Bigr)^{1/2}
	\\
	&\approx
	\left(
	\frac1J
	\sum_{j=1}^{J}
	\max_{1\le m\le M}
	\|\vsigma_{\rm ref}^m(\omega_j)-\vsigma_h^m(\omega_j)\|_{H^1}^{2}
	\right)^{1/2},
\end{align*}
where $J$ denotes the number of Monte Carlo samples.

{\bf Test 1.}
In this experiment, we verify the global convergence rates predicted by Theorem~\ref{Thm_global_error}. Since the theoretical estimate is of order $\mathcal{O}(k^{1/2}+h)$, we choose the time step according to $k=h^2$ so that the temporal and spatial discretization errors are balanced, yielding an overall convergence rate of $\mathcal{O}(h)$.

The proposed algorithm is implemented on the sequence of meshes $h=2^{-1},2^{-2},2^{-3},2^{-4}$ with the corresponding time step sizes $k=2^{-2},2^{-4},2^{-6},2^{-8}$. As the exact solution is unavailable, the reference approximation is computed on the refined mesh $(h_{\rm ref},k_{\rm ref})=(h/2,k/4)$, and the numerical errors are evaluated by comparing the computed solutions on two consecutive levels of spatial and temporal refinements. The resulting nonlinear algebraic systems are solved by a fixed-point iteration, which has been widely used for stochastic nonlinear partial differential equations; see, for example, \cite{vo2025analysis,feng2024high,nguyen2026finite,nguyen2025fully,vo2026mixed}. Compared with the classical Newton method, the fixed-point iteration is computationally more efficient while providing satisfactory convergence for the problems considered here.

The computed errors together with the corresponding experimental convergence orders are reported in Table~\ref{tab:test1}. The numerical results show that the discrete $L^2_{\omega}L_t^{2}H_x^1(u)$- and $L^2_{\omega}L_t^{\infty}H_x^1(\vsigma)$-errors converge with approximately first-order accuracy, in excellent agreement with the theoretical prediction of Theorem~\ref{Thm_global_error}. In contrast, the discrete $L^2_{\omega}L_t^{\infty}L_x^2(u)$- and $L^2_{\omega}L_t^{\infty}L_x^2(v)$-errors exhibit convergence rates noticeably higher than one. Since the choice $k=h^2$ balances the temporal and spatial discretization errors, it is difficult to determine whether this improved convergence is due to a higher-order spatial approximation. This observation motivates the next experiment, in which the temporal error is made sufficiently small so that the spatial convergence can be investigated independently.
\begin{table}[ht]
	\begin{tabular}{c|ccc|ccc}
		\hline
		\multirow{2}{*}{Error}
		& \multicolumn{3}{c|}{Spatial convergence}
		& \multicolumn{3}{c}{Temporal convergence} \\
		\cline{2-7}
		& $h$ & Error & Order
		& $k$ & Error & Order \\
		\hline
		
		\multirow{4}{*}{$L^2_{\omega}L_t^{\infty}L_x^2(u)$}
		& $2^{-1}$ & $1.04546\times10^{-2}$ & --    & $2^{-2}$ & $1.04546\times10^{-2}$ & --    \\
		& $2^{-2}$ & $3.78230\times10^{-3}$ & 1.467 & $2^{-4}$ & $3.78230\times10^{-3}$ & 0.734 \\
		& $2^{-3}$ & $1.11964\times10^{-3}$ & 1.756 & $2^{-6}$ & $1.11964\times10^{-3}$ & 0.878 \\
		& $2^{-4}$ & $3.50000\times10^{-4}$ & 1.678 & $2^{-8}$ & $3.50000\times10^{-4}$ & 0.839 \\
		\hline
		
		\multirow{4}{*}{$L^2_{\omega}L_t^{2}H_x^1(u)$}
		& $2^{-1}$ & $1.01138$ & --    & $2^{-2}$ & $1.01138$ & --    \\
		& $2^{-2}$ & $6.10662\times10^{-1}$ & 0.728 & $2^{-4}$ & $6.10662\times10^{-1}$ & 0.364 \\
		& $2^{-3}$ & $3.21143\times10^{-1}$ & 0.927 & $2^{-6}$ & $3.21143\times10^{-1}$ & 0.463 \\
		& $2^{-4}$ & $1.63272\times10^{-1}$ & 0.976 & $2^{-8}$ & $1.63272\times10^{-1}$ & 0.488 \\
		\hline
		
		\multirow{4}{*}{$L^2_{\omega}L_t^{\infty}L_x^2(v)$}
		& $2^{-1}$ & $1.43928\times10^{-1}$ & --    & $2^{-2}$ & $1.43928\times10^{-1}$ & --    \\
		& $2^{-2}$ & $4.18400\times10^{-2}$ & 1.782 & $2^{-4}$ & $4.18400\times10^{-2}$ & 0.891 \\
		& $2^{-3}$ & $1.36386\times10^{-2}$ & 1.617 & $2^{-6}$ & $1.36386\times10^{-2}$ & 0.808 \\
		& $2^{-4}$ & $4.92915\times10^{-3}$ & 1.468 & $2^{-8}$ & $4.92915\times10^{-3}$ & 0.734 \\
		\hline
		
		\multirow{4}{*}{$L^2_{\omega}L_t^{\infty}H_x^1(\vsigma)$}
		& $2^{-1}$ & $2.28653\times10^{-1}$ & --    & $2^{-2}$ & $2.28653\times10^{-1}$ & --    \\
		& $2^{-2}$ & $1.25449\times10^{-1}$ & 0.866 & $2^{-4}$ & $1.25449\times10^{-1}$ & 0.433 \\
		& $2^{-3}$ & $7.15228\times10^{-2}$ & 0.811 & $2^{-6}$ & $7.15228\times10^{-2}$ & 0.405 \\
		& $2^{-4}$ & $4.22027\times10^{-2}$ & 0.761 & $2^{-8}$ & $4.22027\times10^{-2}$ & 0.380 \\
		\hline
	\end{tabular}
	\medskip
	\centering
	\caption{Discrete strong errors and corresponding experimental convergence orders for Test~1 with $k=h^2$.}
	\label{tab:test1}
\end{table}

{\bf Test 2.}
In this experiment, we investigate the spatial convergence behavior of the proposed method. To minimize the influence of the temporal discretization error, the time step is fixed at $k=2^{-9}$ while only the spatial mesh is refined with $h=2^{-1},2^{-2},2^{-3},2^{-4}$. As in Test~1, the reference approximation is computed on the refined mesh $h_{\rm ref} =h/2$.

The computed spatial errors together with the corresponding experimental convergence orders are reported in Table~\ref{tab:spatial}. The results show that the $L^2_{\omega}L_t^{2}H_x^1(u)$- and $L^2_{\omega}L_t^{\infty}H_x^1(\vsigma)$-errors converge with approximately first-order accuracy, which is fully consistent with the theoretical prediction of Theorem~\ref{Thm_global_error} and the results from Test 1. In contrast, the $L^2_{\omega}L_t^{\infty}L_x^2(u)$- and $L^2_{\omega}L_t^{\infty}L_x^2(v)$-errors exhibit convergence rates approaching two. Although Theorem~\ref{Thm_global_error} guarantees only first-order convergence in these norms, the numerical evidence indicates that the optimal spatial convergence rate may be second order. This observation suggests that the current theoretical analysis could potentially be sharpened by employing an Aubin--Nitsche duality argument to derive improved $L^2$-error estimates.

\begin{table}[ht]
	\begin{tabular}{c|ccc}
		\hline
		\multirow{2}{*}{Error}
		& \multicolumn{3}{c}{Spatial convergence} \\
		\cline{2-4}
		& $h$ & Error & Order \\
		\hline
		
		\multirow{4}{*}{$L^2_{\omega}L_t^{\infty}L_x^2(u)$}
		& $2^{-1}$ & $1.28444\times10^{-2}$ & --    \\
		& $2^{-2}$ & $4.68667\times10^{-3}$ & 1.454 \\
		& $2^{-3}$ & $1.37100\times10^{-3}$ & 1.773 \\
		& $2^{-4}$ & $3.64000\times10^{-4}$ & 1.913 \\
		\hline
		
		\multirow{4}{*}{$L^2_{\omega}L_t^{2}H_x^1(u)$}
		& $2^{-1}$ & $1.01358$ & --    \\
		& $2^{-2}$ & $6.14774\times10^{-1}$ & 0.722 \\
		& $2^{-3}$ & $3.22928\times10^{-1}$ & 0.929 \\
		& $2^{-4}$ & $1.63491\times10^{-1}$ & 0.982 \\
		\hline
		
		\multirow{4}{*}{$L^2_{\omega}L_t^{\infty}L_x^2(v)$}
		& $2^{-1}$ & $1.05227\times10^{-1}$ & --    \\
		& $2^{-2}$ & $3.74714\times10^{-2}$ & 1.490 \\
		& $2^{-3}$ & $1.24651\times10^{-2}$ & 1.588 \\
		& $2^{-4}$ & $3.67386\times10^{-3}$ & 1.763 \\
		\hline
		
		\multirow{4}{*}{$L^2_{\omega}L_t^{\infty}H_x^1(\vsigma)$}
		& $2^{-1}$ & $1.86350\times10^{-1}$ & --    \\
		& $2^{-2}$ & $1.22084\times10^{-1}$ & 0.610 \\
		& $2^{-3}$ & $7.11920\times10^{-2}$ & 0.778 \\
		& $2^{-4}$ & $4.22000\times10^{-2}$ & 0.755 \\
		\hline

	\end{tabular}
	\medskip
		\centering
	\caption{Discrete strong errors and experimental spatial convergence orders for Test~2 with the fixed time step $k=2^{-9}$.}
	\label{tab:spatial}
\end{table}

{\bf Test 3.}
In this experiment, we investigate the qualitative behavior of the stochastic Keller--Segel system \eqref{eq1.1} by applying the proposed splitting finite element method. We choose $B(u)=2u$, $\chi=5$, and employ a uniform mesh with mesh size $h=1/40$. The solutions are computed at the final times $T=0.005$, $0.01$, $0.02$, and $0.05$. The computational domain is taken as $D=\left[-\frac12,\frac12\right]\times\left[-\frac12,\frac12\right]$, and the initial condition is chosen as $u_0(x,y)=1+2e^{-40x^2-40y^2}$.

We compute the expected value of the numerical solution using the Monte Carlo method together with three representative sample paths. For comparison, we also compute the corresponding deterministic solution by setting $B(u)=0$.

Figures~\ref{fig5.5}--\ref{fig5.9} illustrate the evolution of the expected solution, the deterministic solution, and three representative sample paths. We observe that all computed solutions remain bounded over the entire simulation interval and no blow-up is observed. This behavior is consistent with the global well-posedness result established in \cite{chen2025well}, where the logistic growth term suppresses excessive chemotactic aggregation and guarantees global solutions. In contrast, for stochastic Keller--Segel equations without the logistic damping term, finite-time blow-up may occur even for nontrivial initial data under linear multiplicative noise \cite{misiats2022global}. Although the two models are different, the present numerical results clearly illustrate the regularizing effect of the logistic term in preventing the excessive aggregation induced by chemotaxis and stochastic perturbations.

Figures~\ref{fig5.5} and~\ref{fig5.6} show that the expected solution and the deterministic solution exhibit very similar qualitative behavior. In particular, the peak value decreases monotonically as the final time increases from $T=0.005$ to $T=0.05$, demonstrating the stabilizing influence of the logistic damping. In contrast, the individual sample paths displayed in Figures~\ref{fig5.7}--\ref{fig5.9} exhibit noticeably different transient dynamics. Depending on the realization of the Wiener process, some sample paths produce stronger chemotactic aggregation than the deterministic solution, whereas others exhibit a faster decay of the peak. These observations demonstrate the variability introduced by the multiplicative stochastic forcing, while the expected solution retains the overall qualitative behavior of the deterministic model.

	\begin{figure}[htp]
	\includegraphics[width=.45\textwidth]{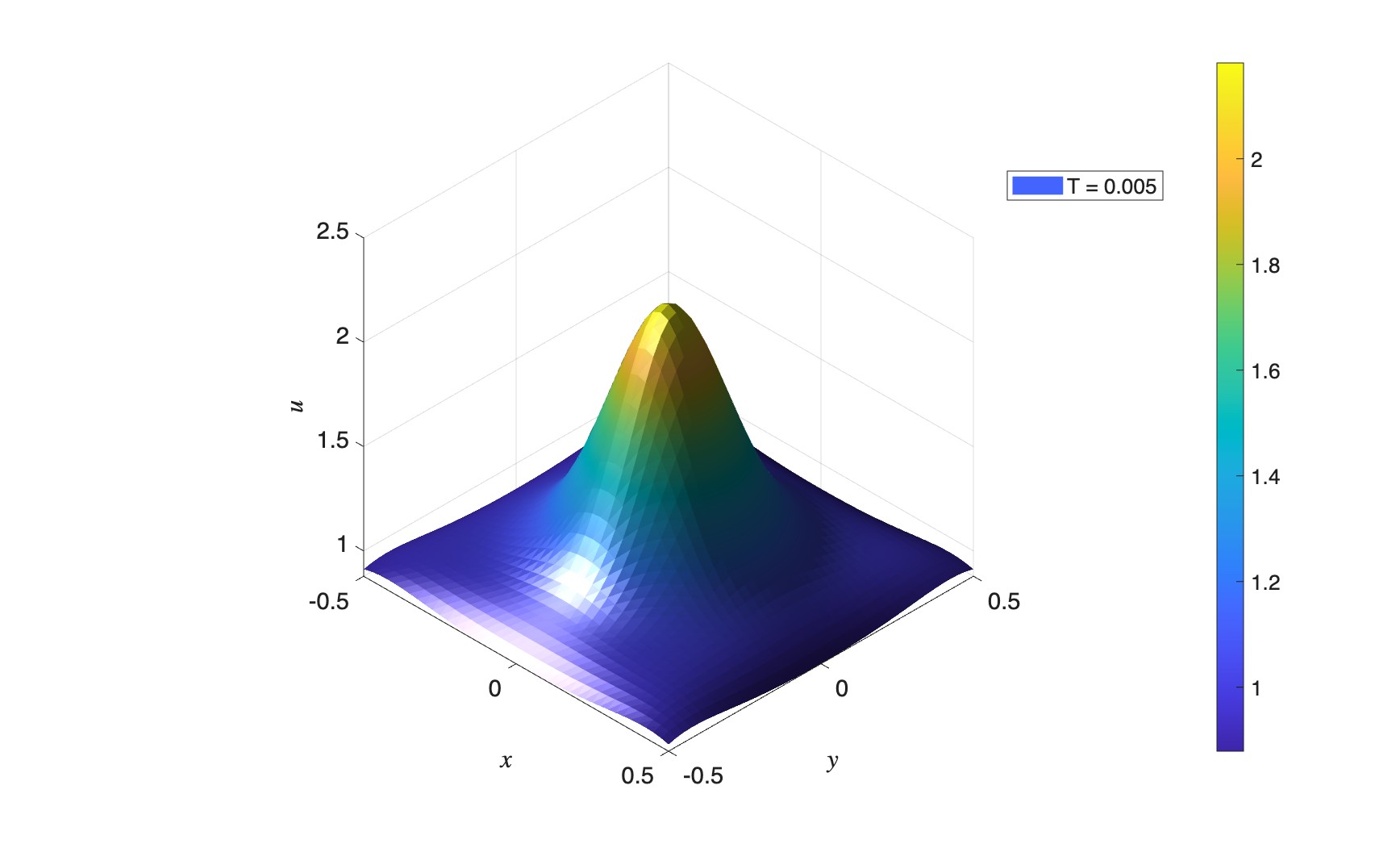}
	\includegraphics[width=.45\textwidth]{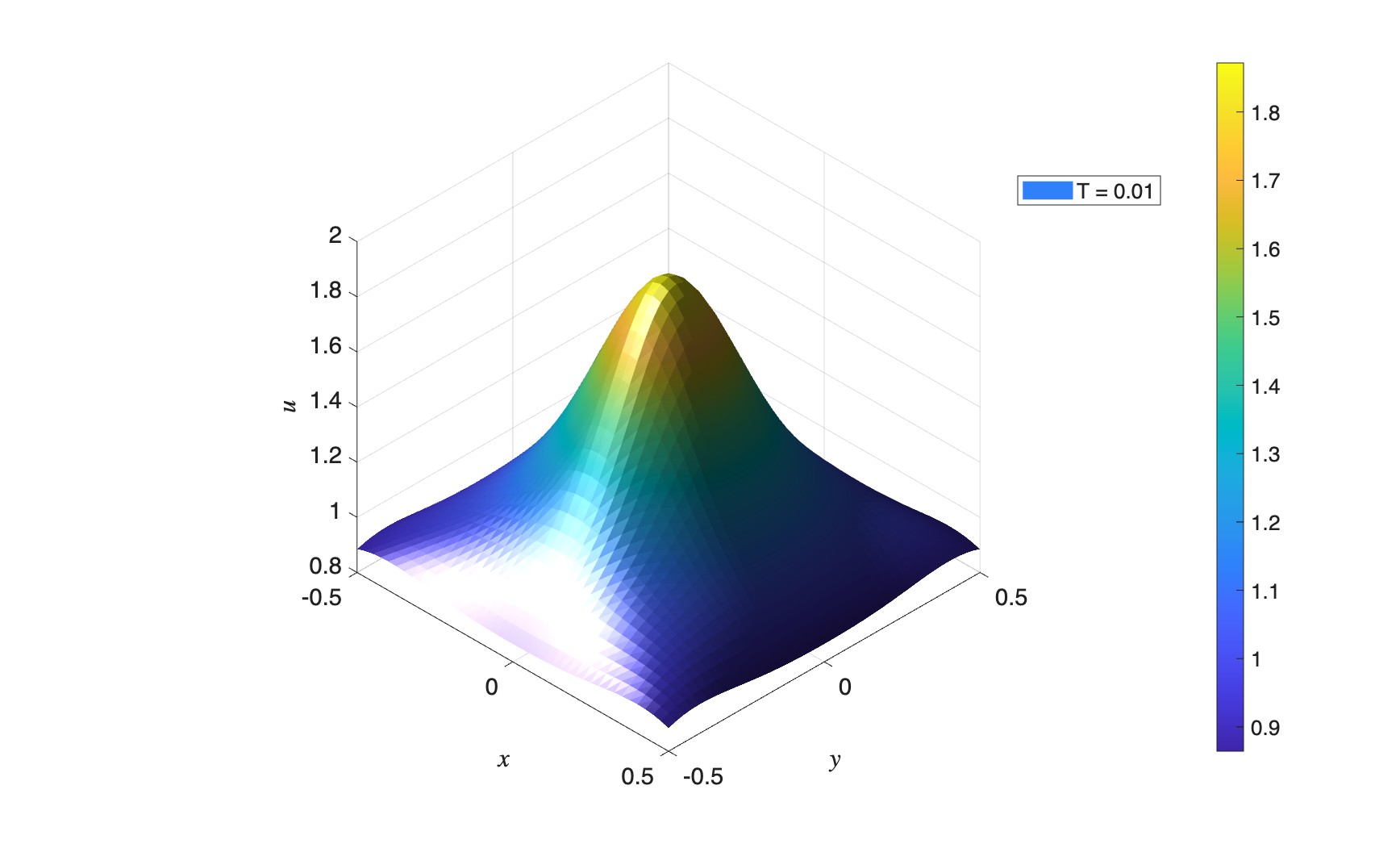}\\
	\includegraphics[width=.45\textwidth]{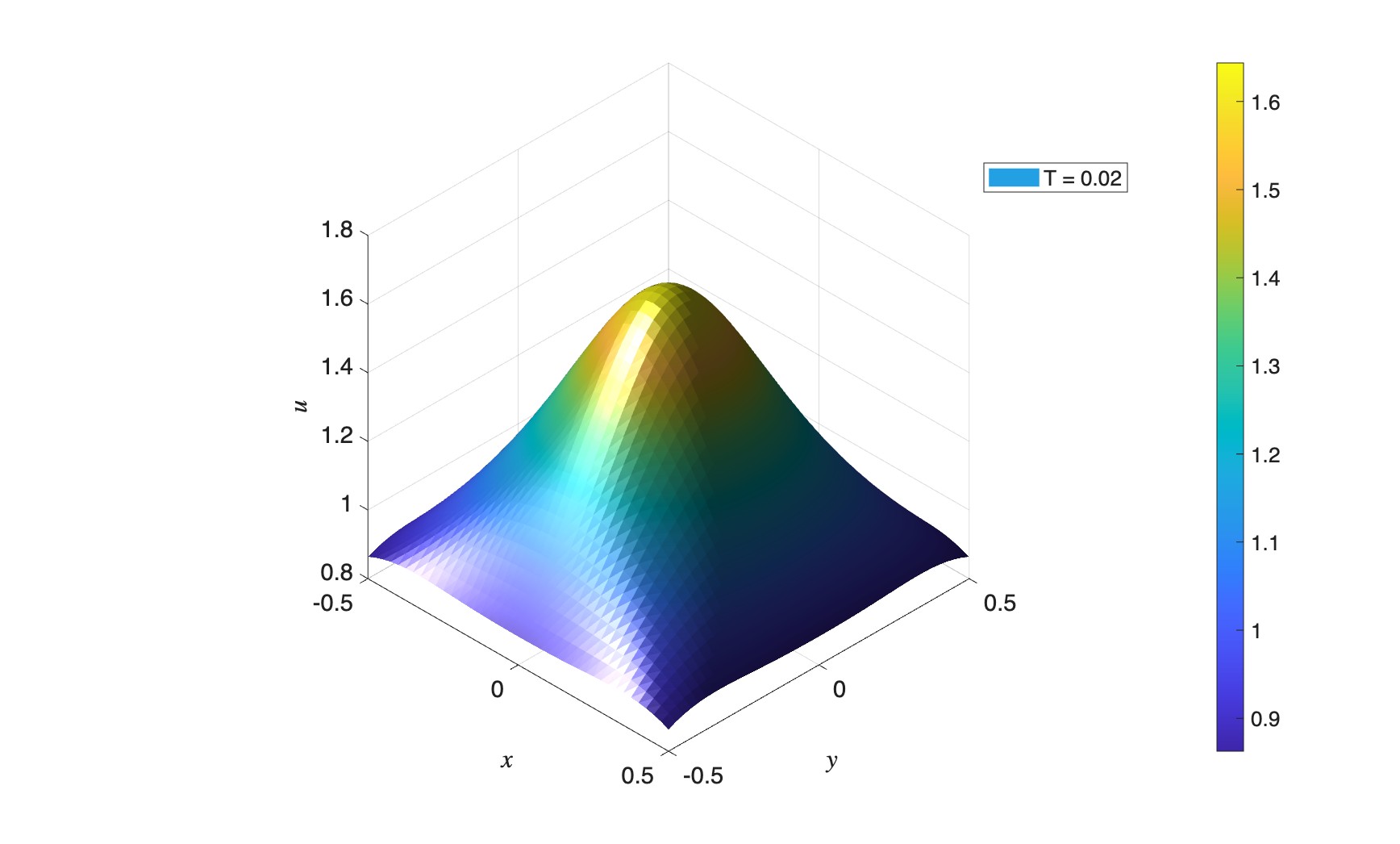}
	\includegraphics[width=.45\textwidth]{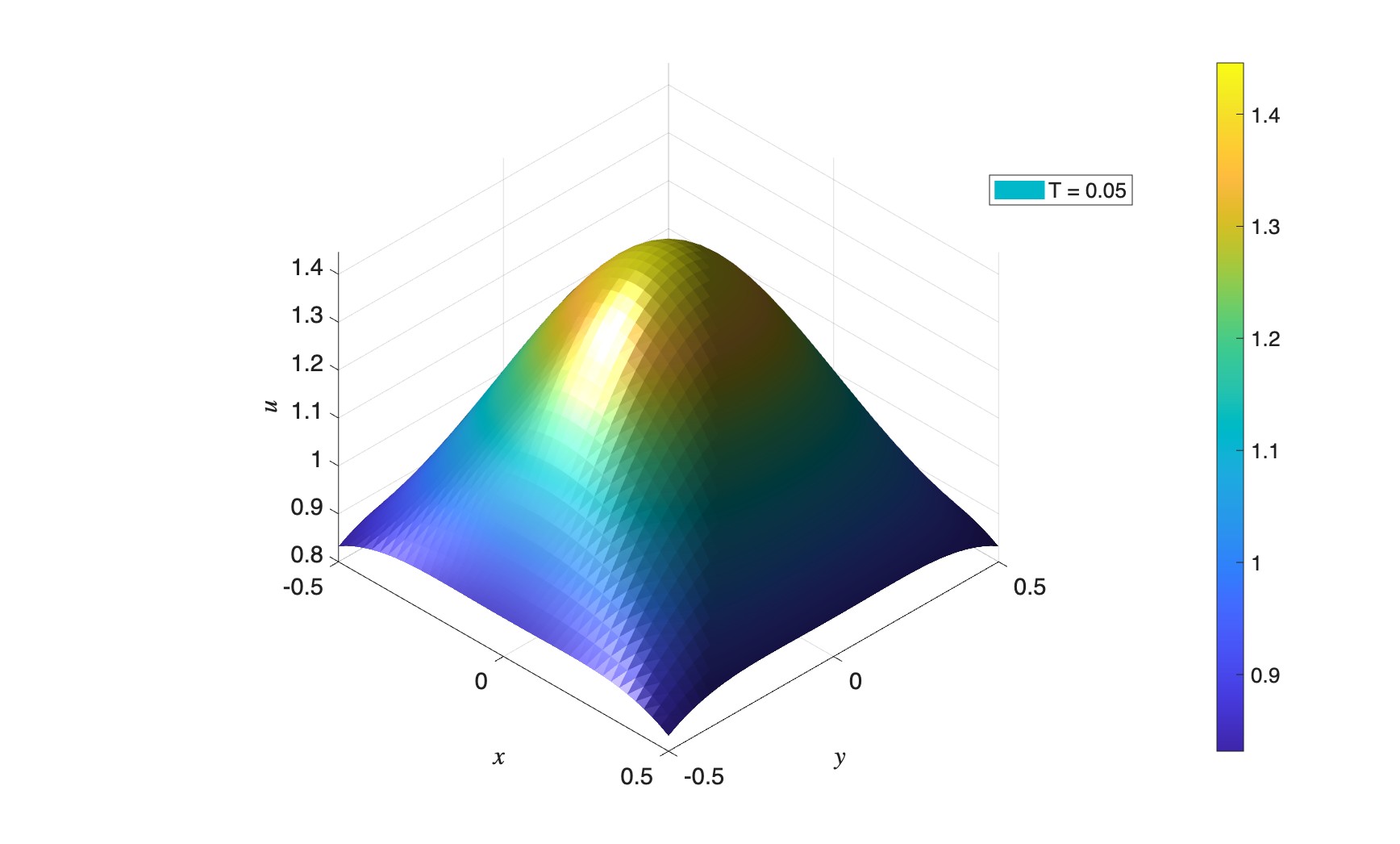}
\caption{Evolution of the expected numerical solution $\mathbb{E}[u_h^M]$ for the stochastic Keller--Segel system with $B(u)=2u$ at the final times $T=0.005$, $0.01$, $0.02$, and $0.05$. The expected solution remains bounded and its peak value decreases monotonically over time.}
\label{fig5.5}
\end{figure}

\begin{figure}[htp]
	\includegraphics[width=.45\textwidth]{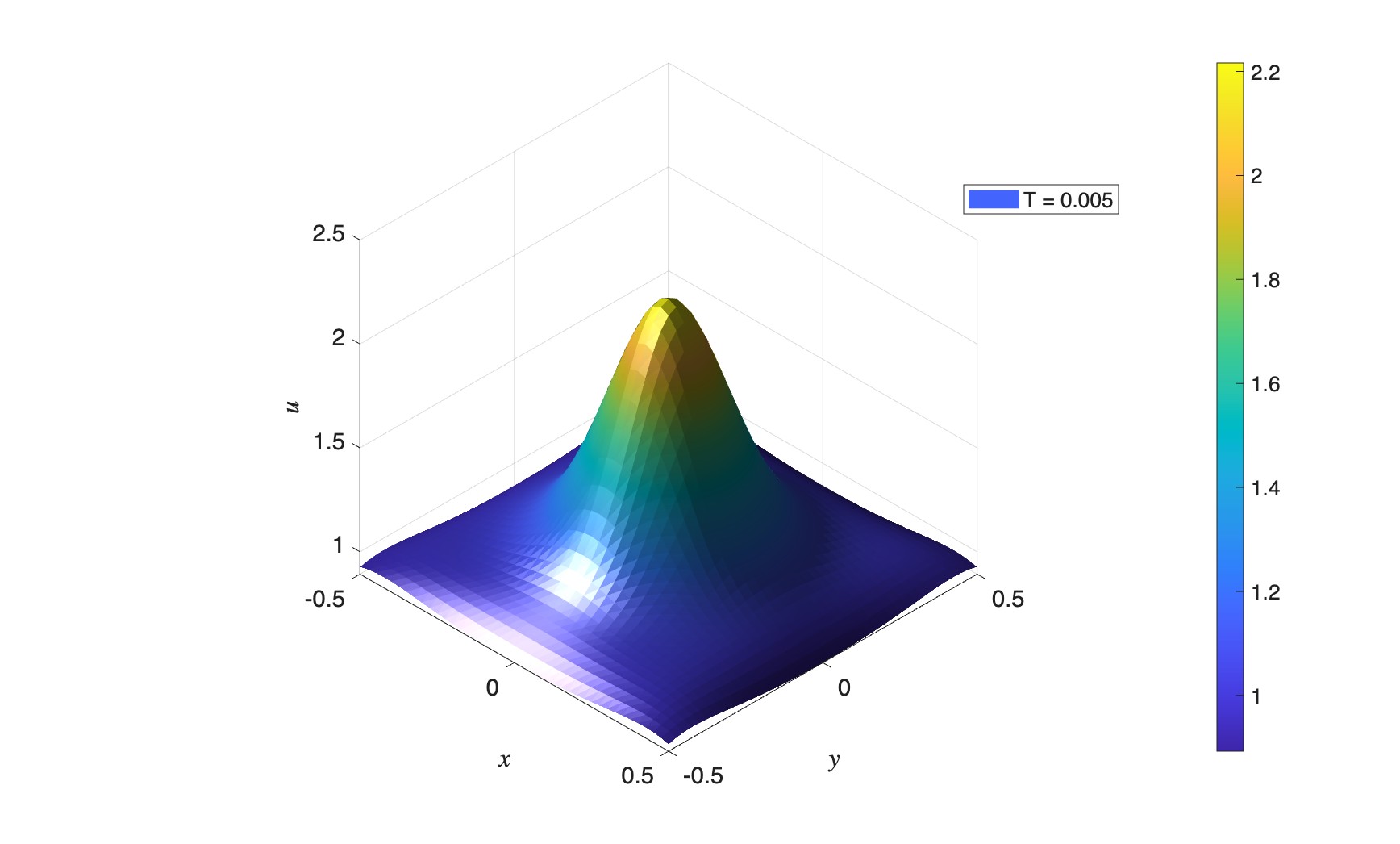}
	\includegraphics[width=.45\textwidth]{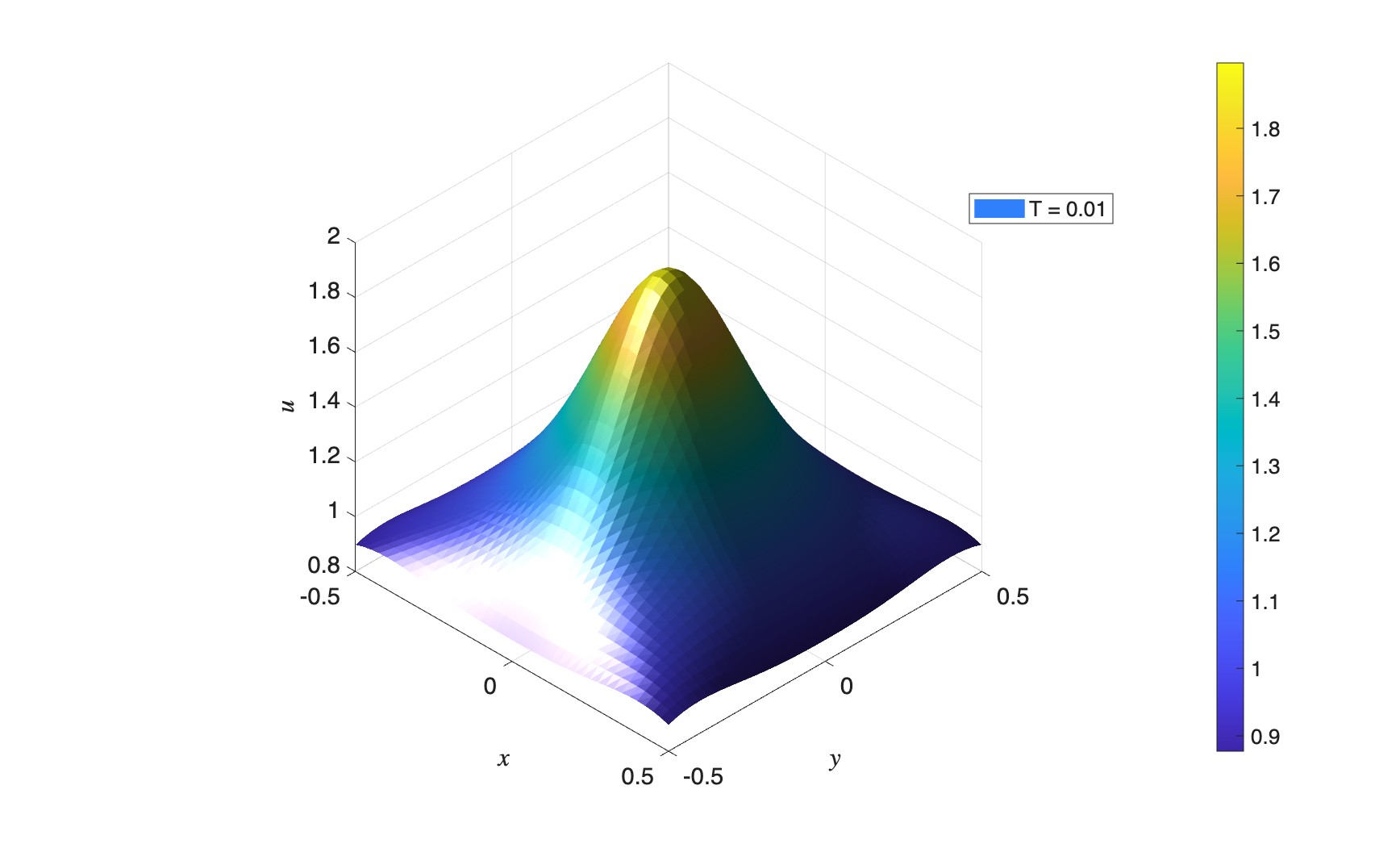}
	\includegraphics[width=.45\textwidth]{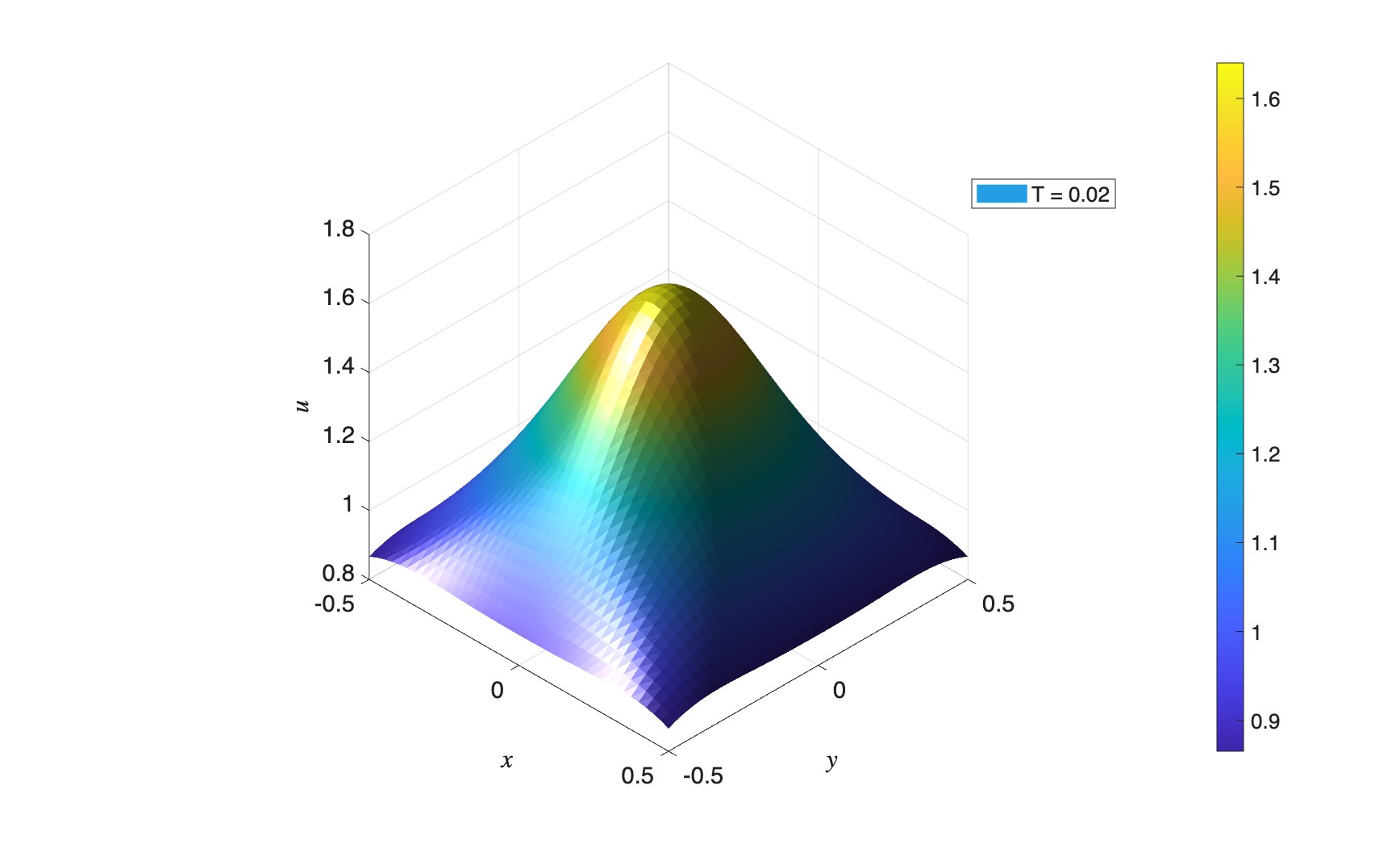}
	\includegraphics[width=.45\textwidth]{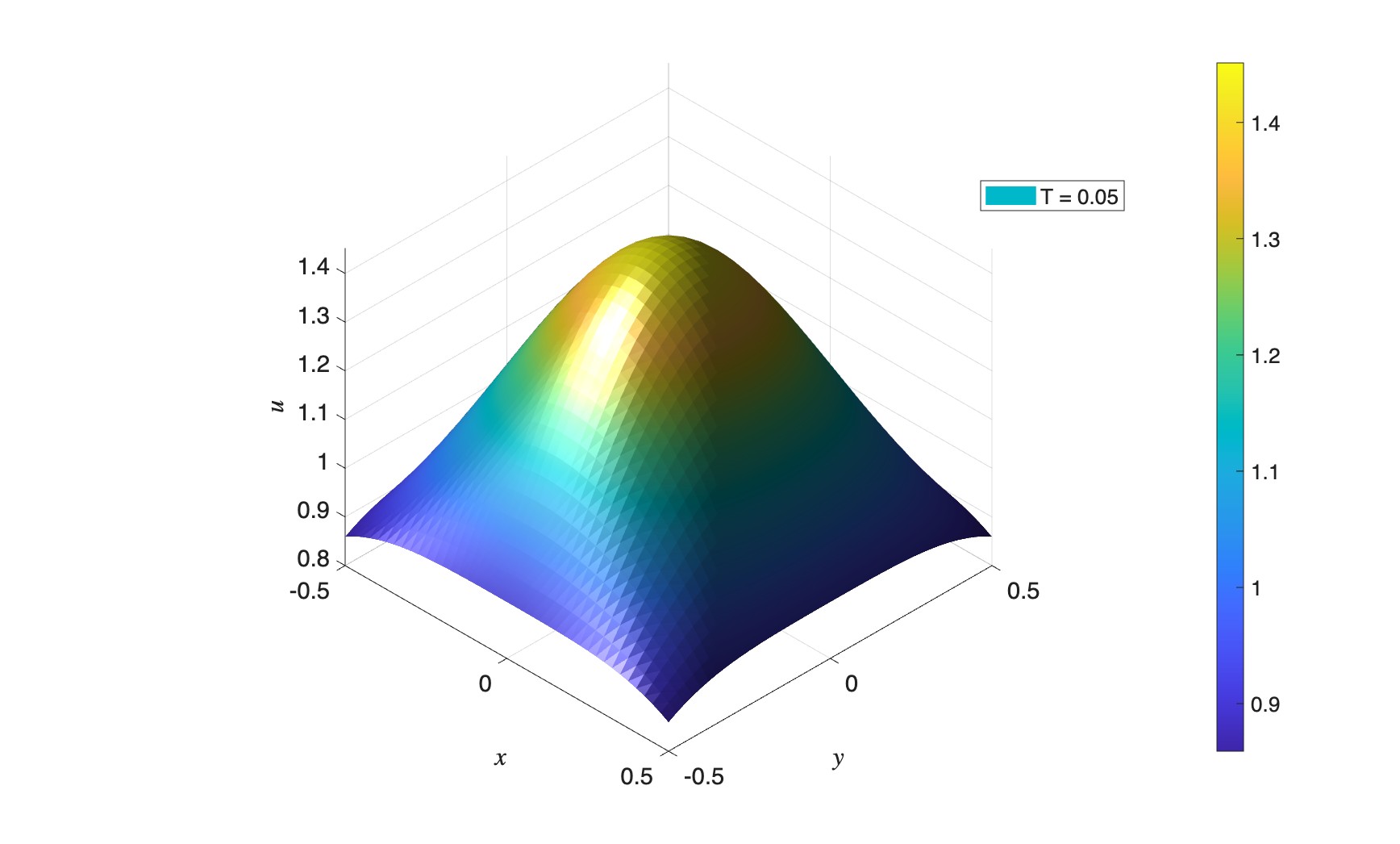}
\caption{Evolution of the deterministic numerical solution obtained by setting $B(u)=0$ at the final times $T=0.005$, $0.01$, $0.02$, and $0.05$. Similar to the expected stochastic solution, the peak value decreases monotonically as time increases.}
\label{fig5.6}
\end{figure}

	\begin{figure}[htp]
	\includegraphics[width=.45\textwidth]{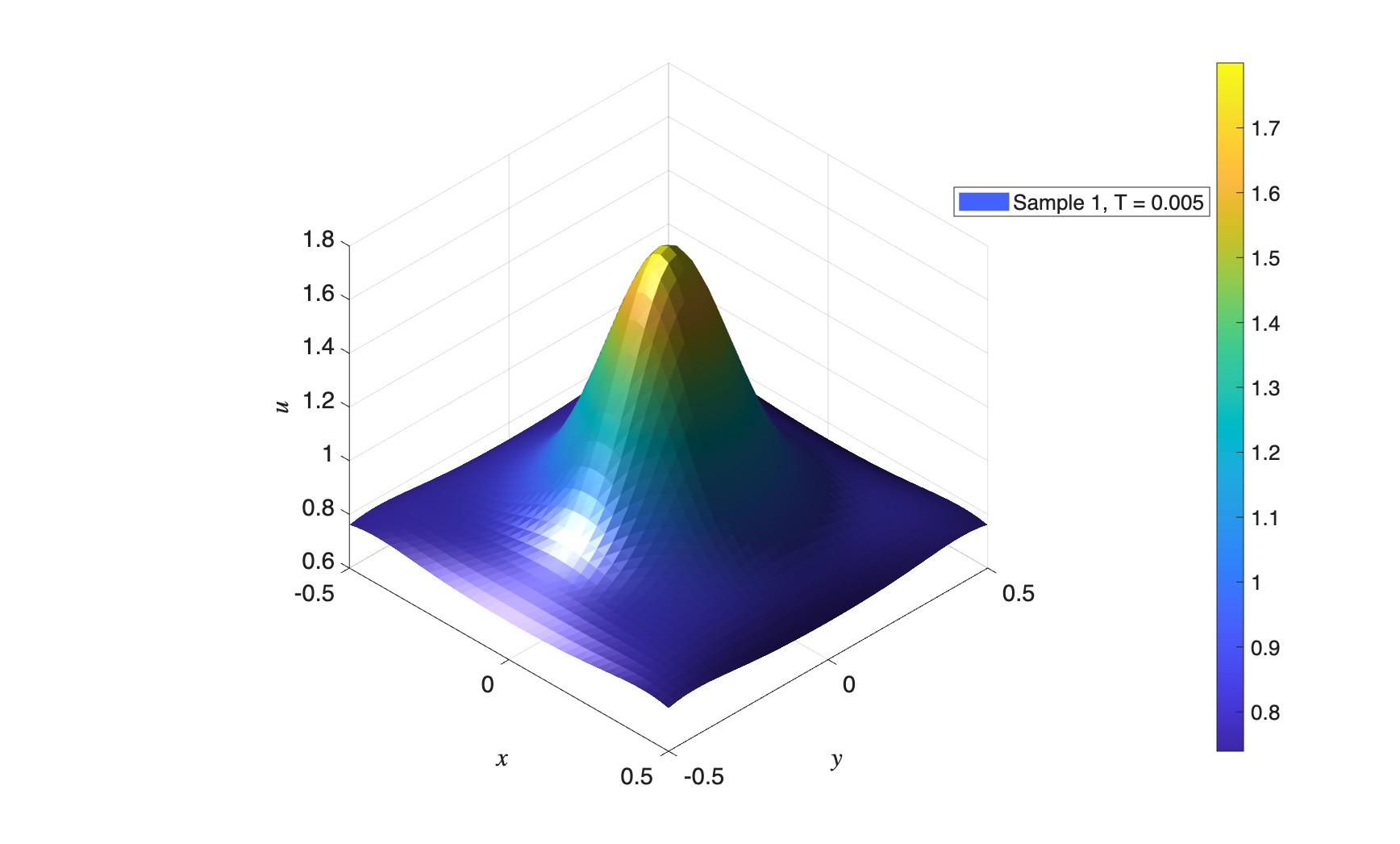}
	\includegraphics[width=.45\textwidth]{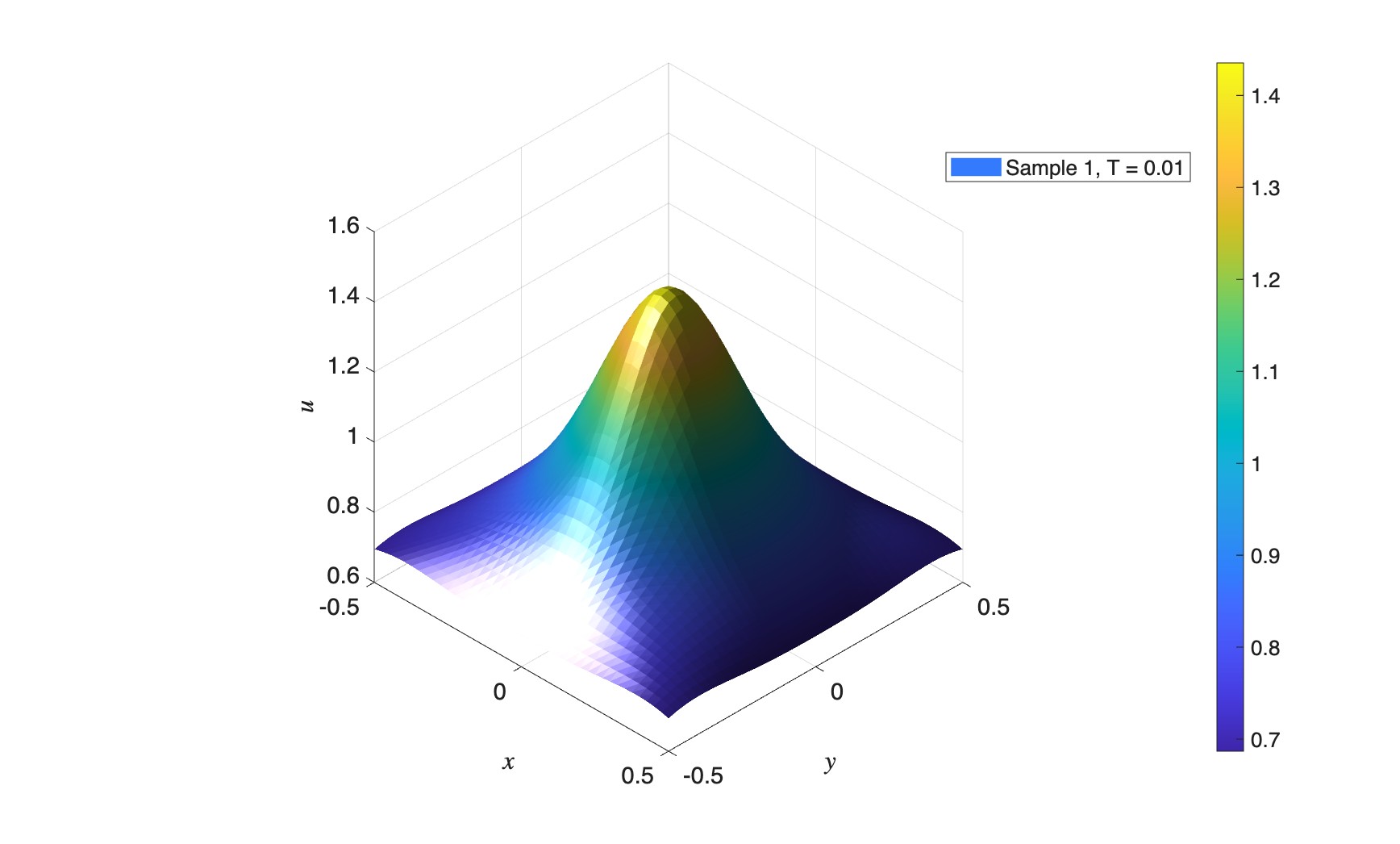}
	\includegraphics[width=.45\textwidth]{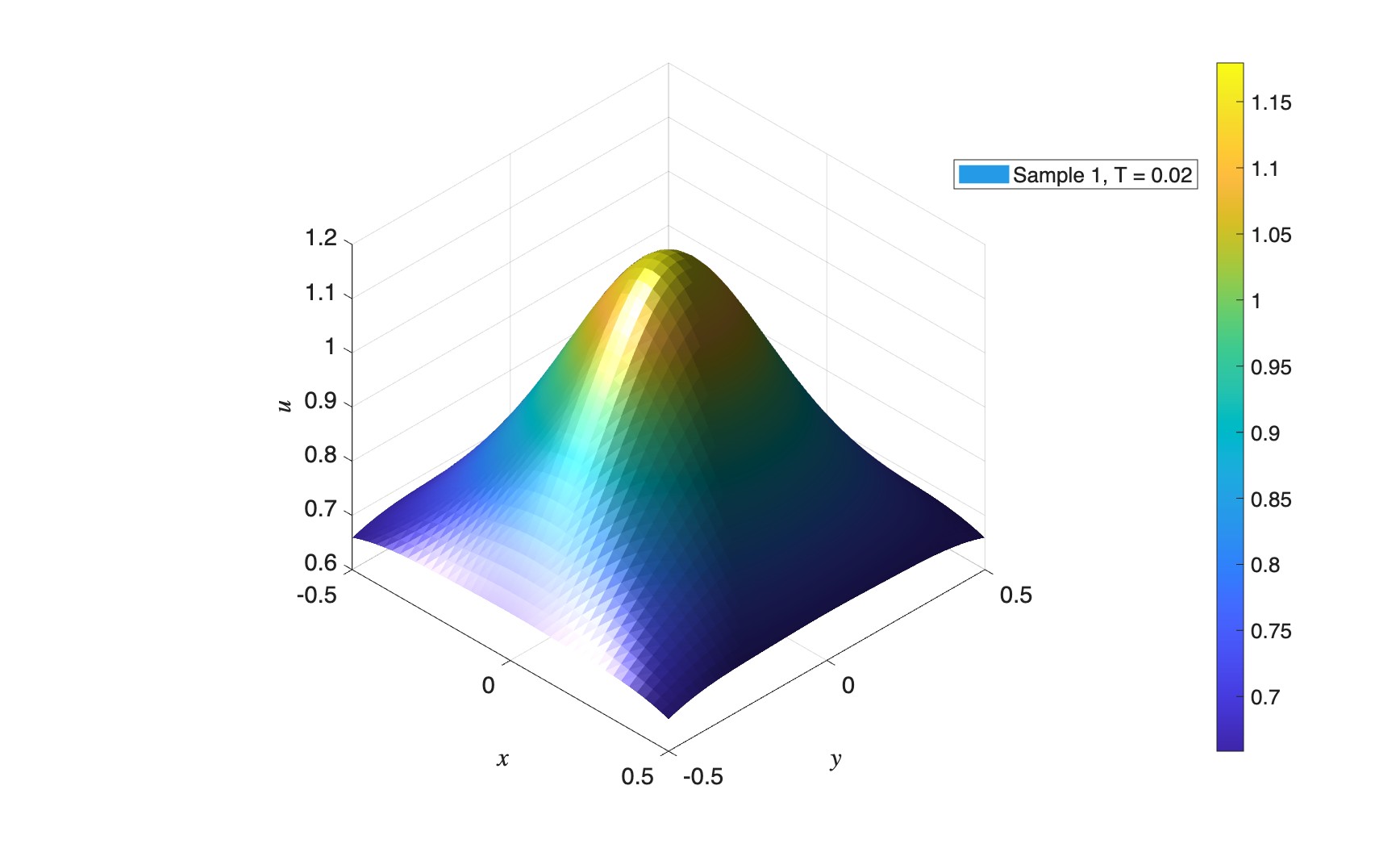}
	\includegraphics[width=.45\textwidth]{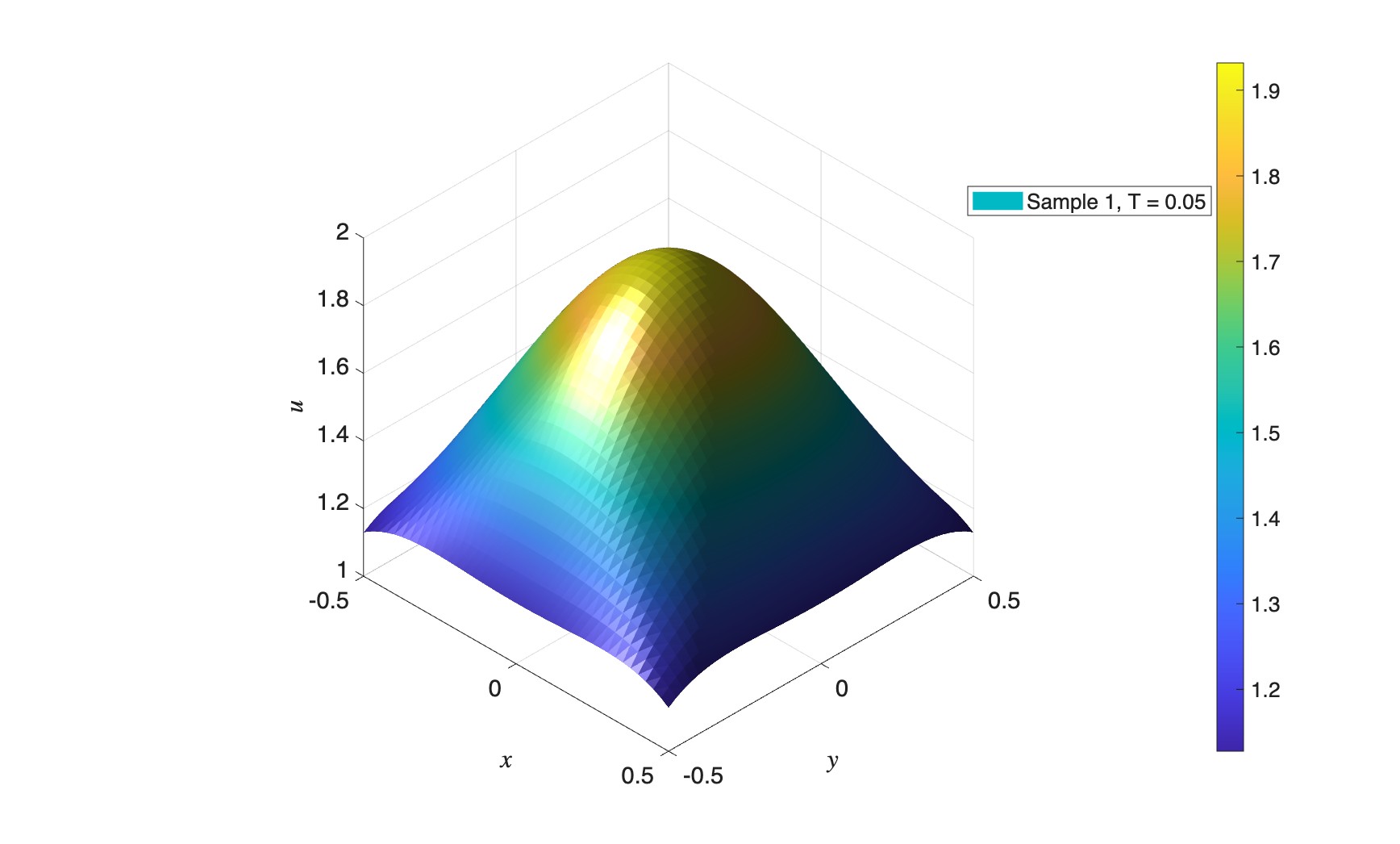}
	\caption{Evolution of the pathwise numerical solution corresponding to the first realization (Sample~1) at the final times $T=0.005$, $0.01$, $0.02$, and $0.05$.}
	\label{fig5.7}
\end{figure}

\begin{figure}[htp]
	\includegraphics[width=.45\textwidth]{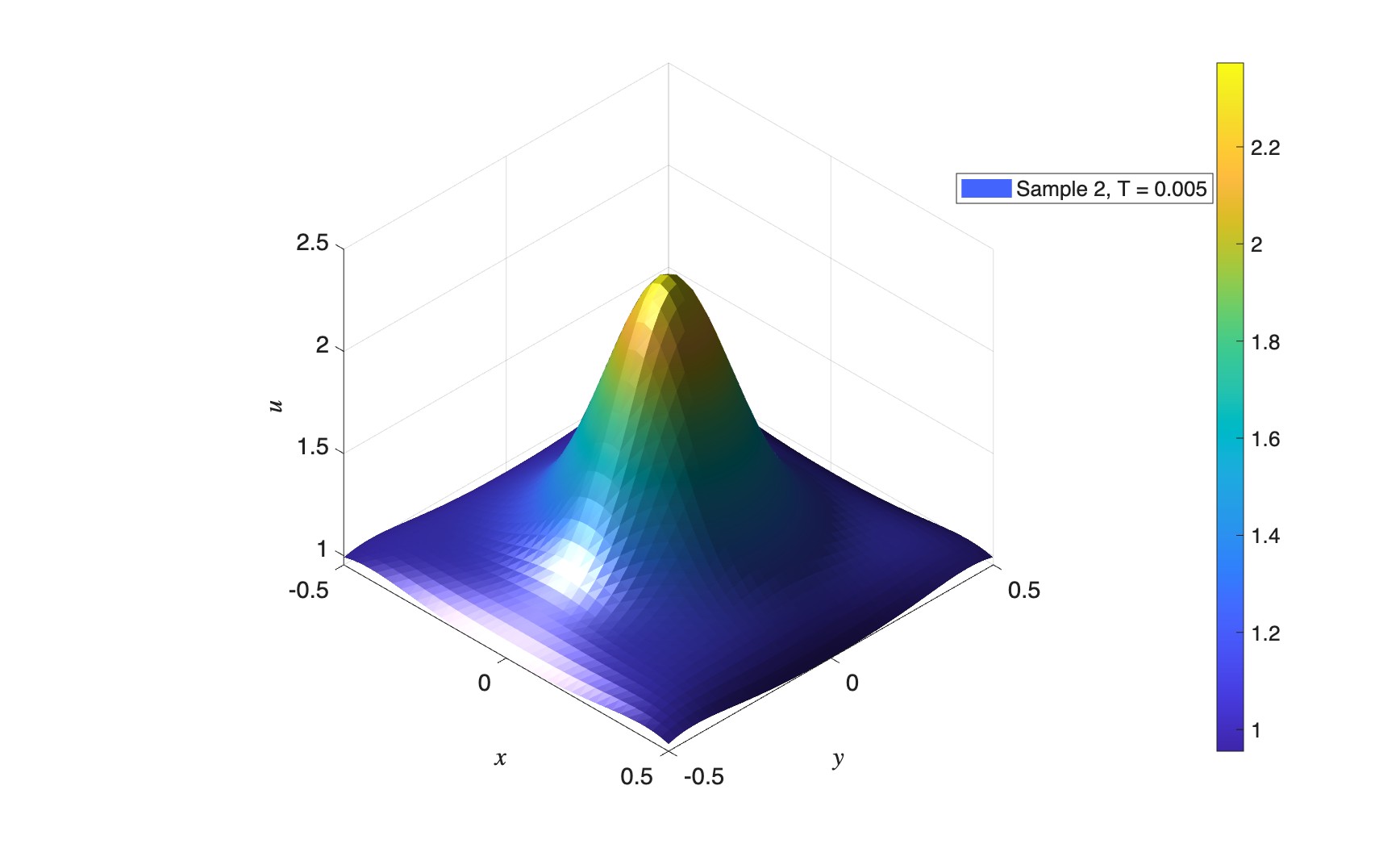}
	\includegraphics[width=.45\textwidth]{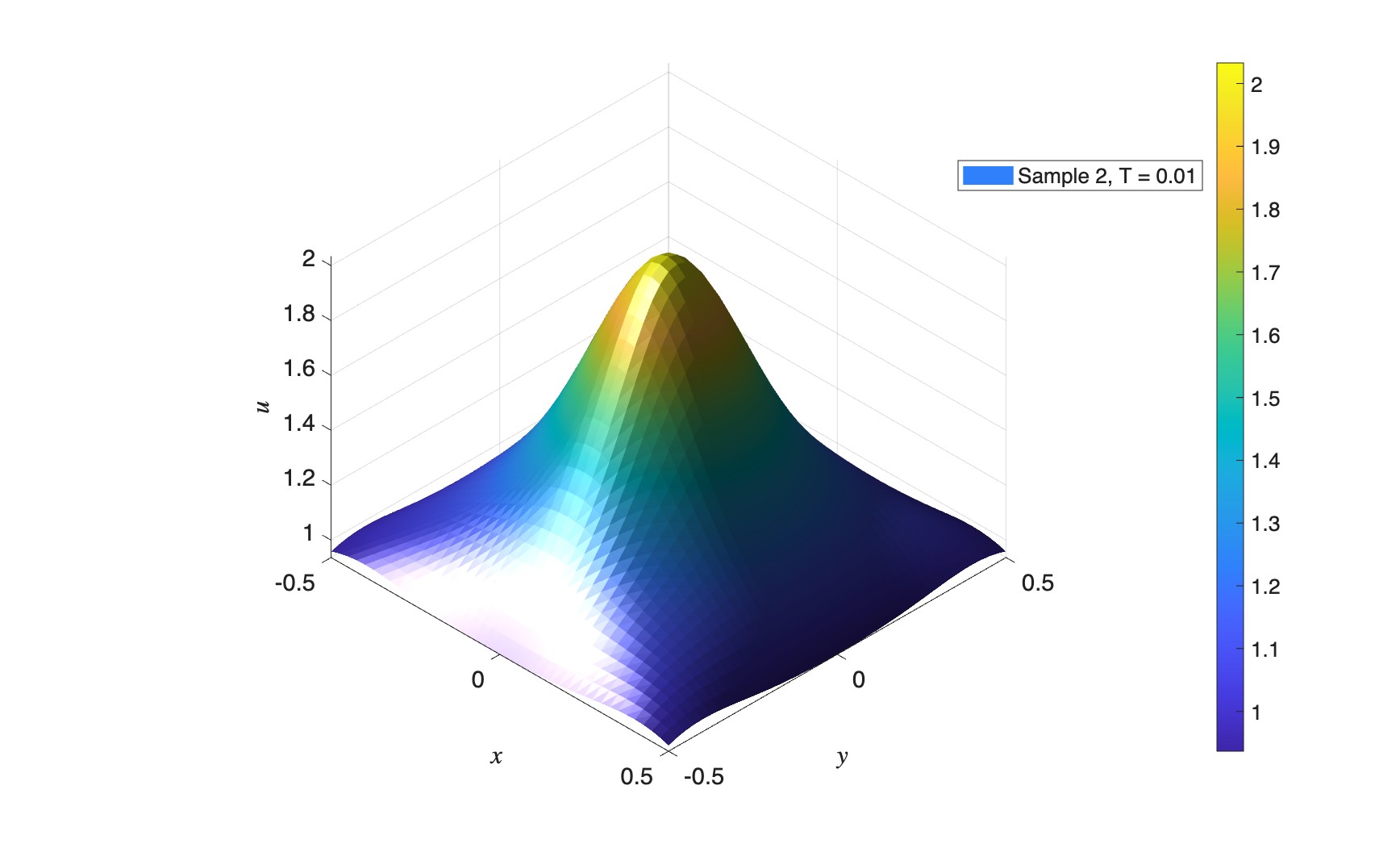}
	\includegraphics[width=.45\textwidth]{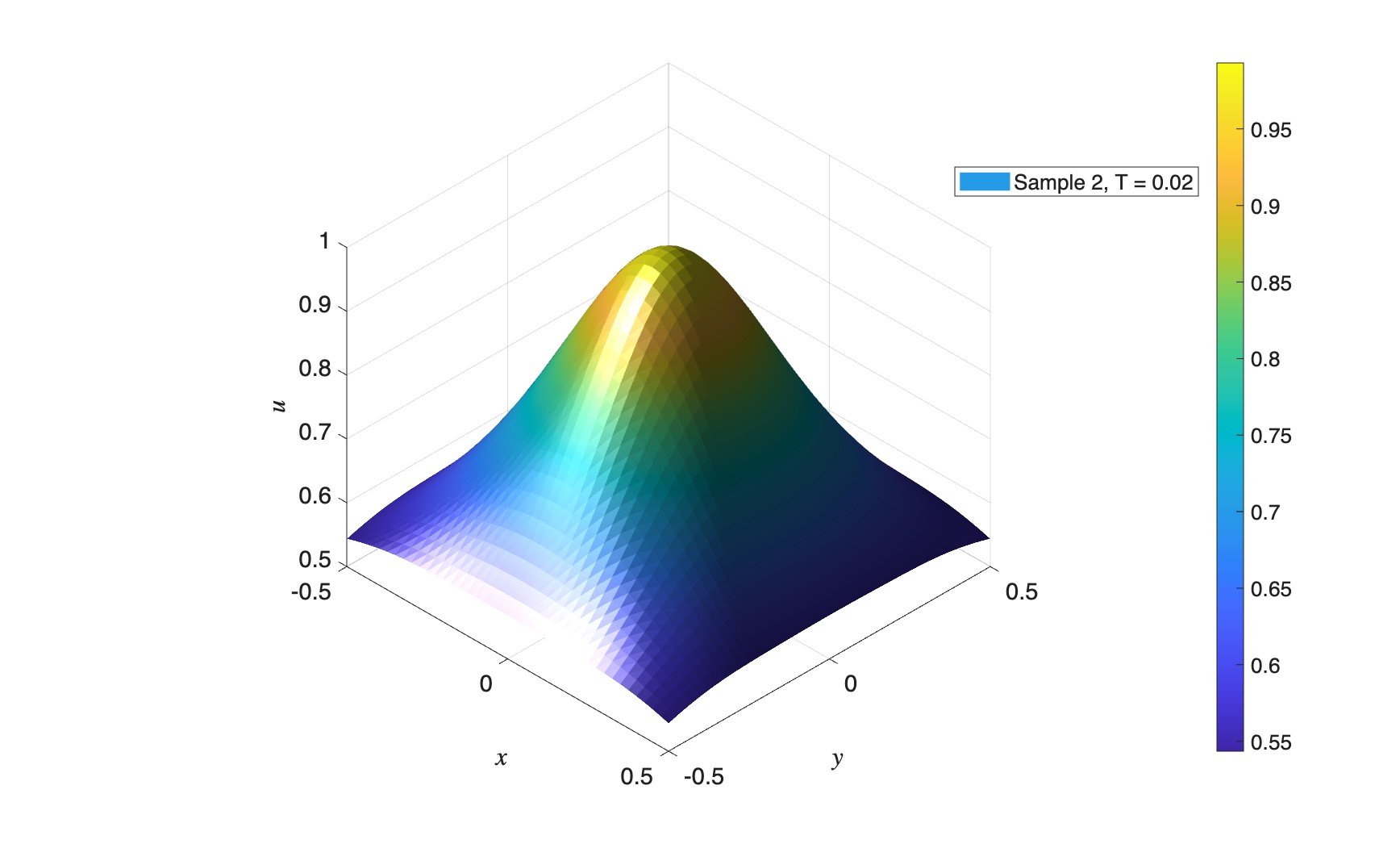}
	\includegraphics[width=.45\textwidth]{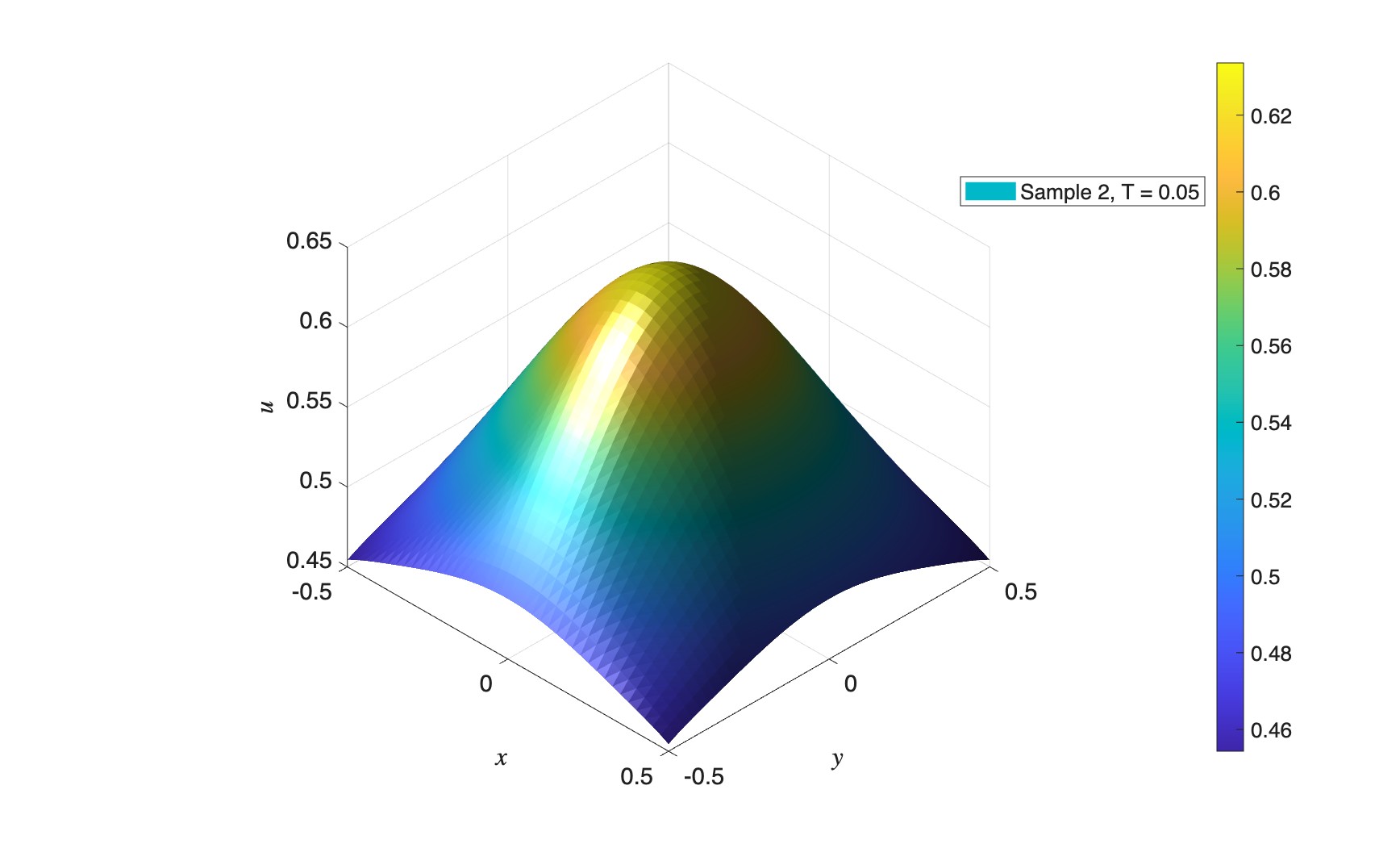}
\caption{Evolution of the pathwise numerical solution corresponding to the second realization (Sample~2) at the final times $T=0.005$, $0.01$, $0.02$, and $0.05$.}
\label{fig5.8}
\end{figure}

\begin{figure}[htp]
	\includegraphics[width=.45\textwidth]{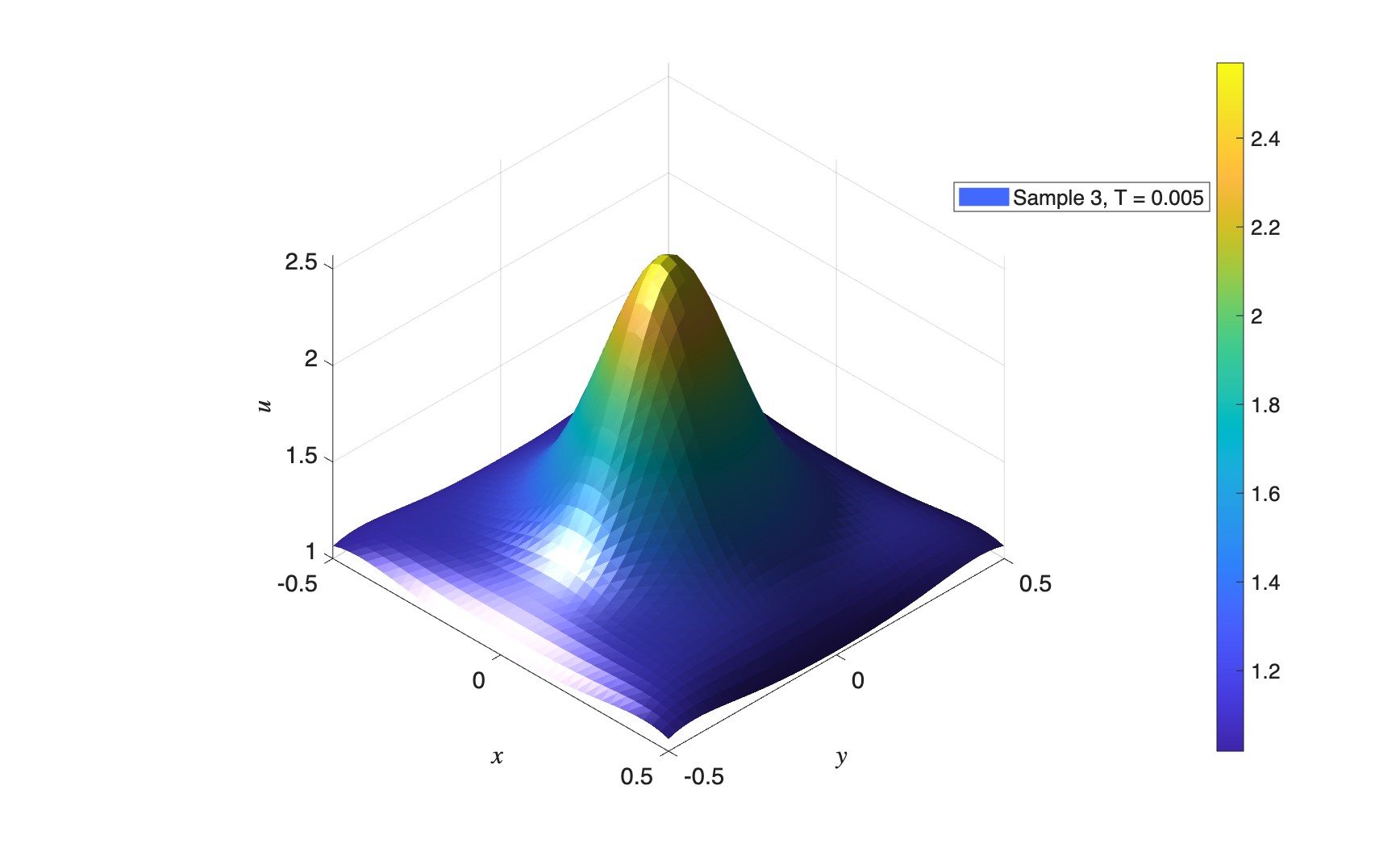}
	\includegraphics[width=.45\textwidth]{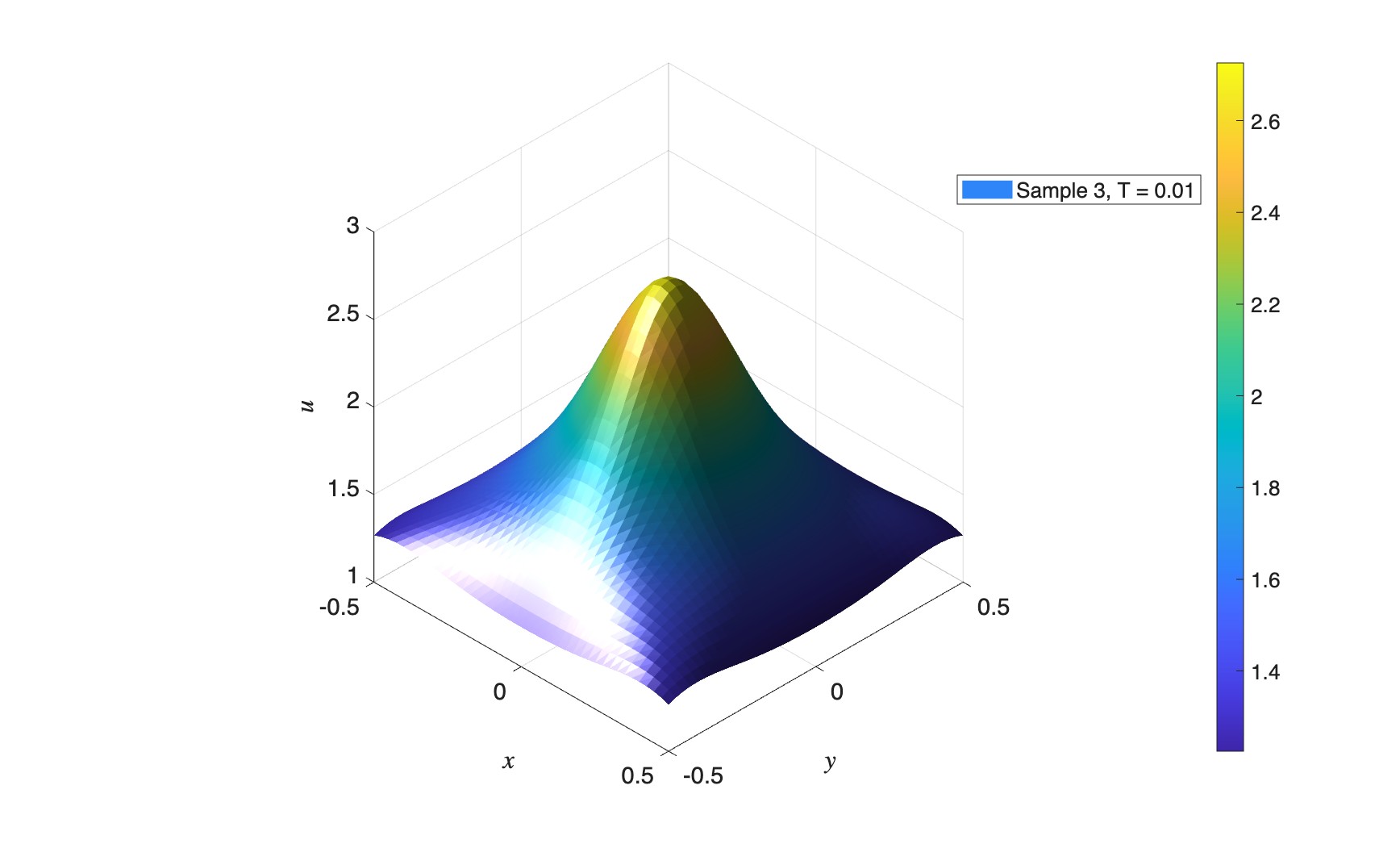}
	\includegraphics[width=.45\textwidth]{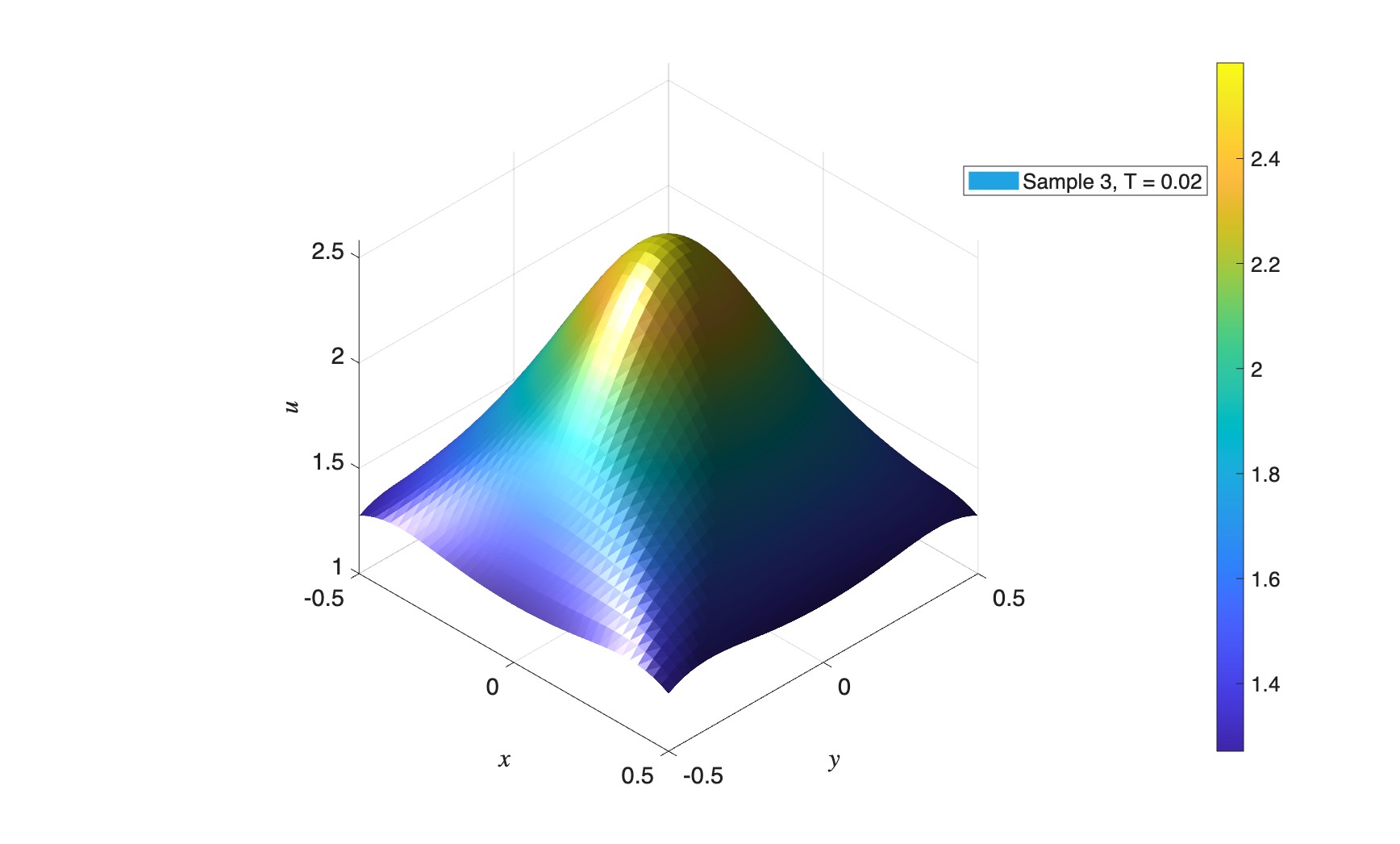}
	\includegraphics[width=.45\textwidth]{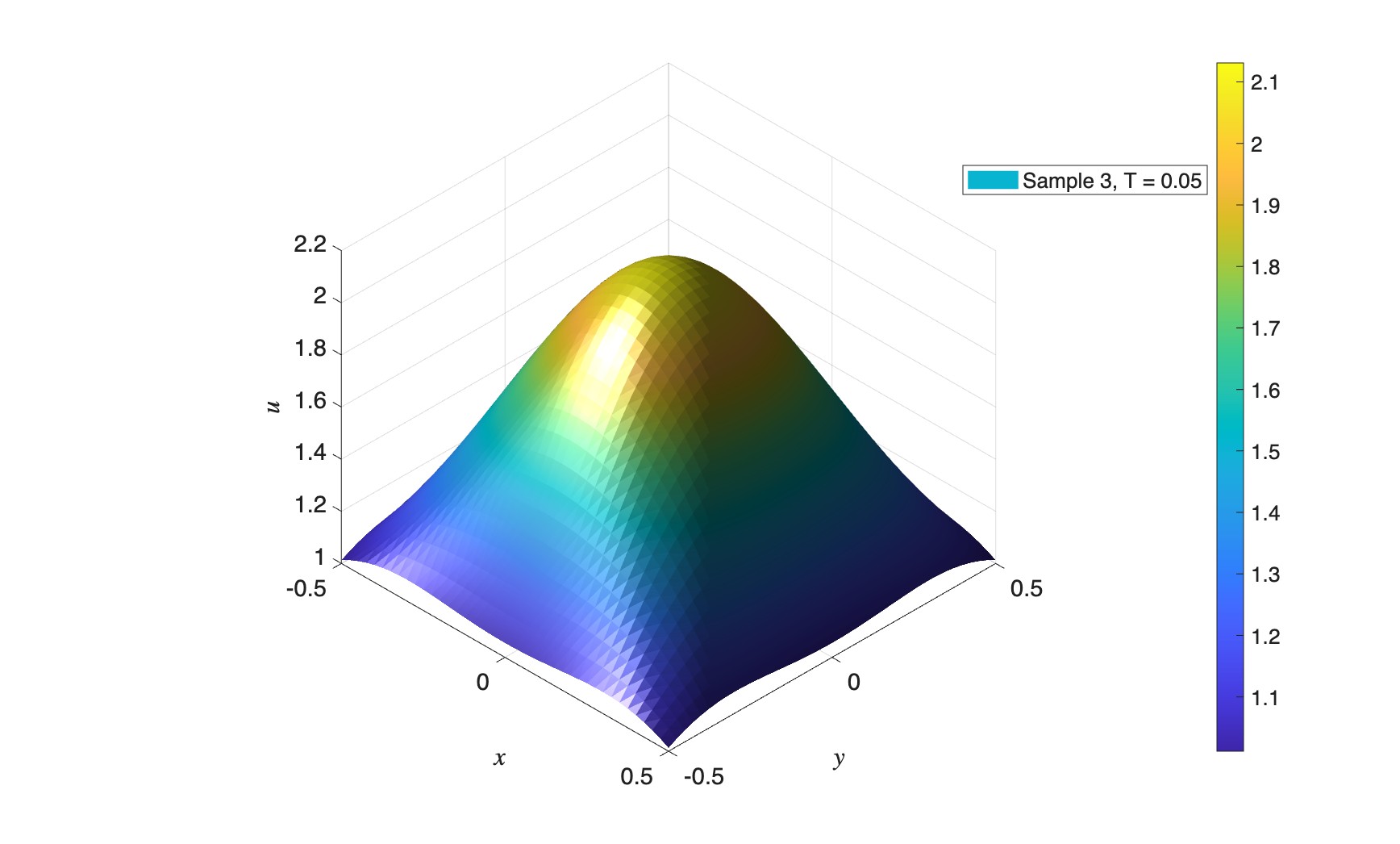}
	\caption{Evolution of the pathwise numerical solution corresponding to the third realization (Sample~3) at the final times $T=0.005$, $0.01$, $0.02$, and $0.05$.}
	\label{fig5.9}
\end{figure}

\section{Conclusion}\label{sec6}

In this paper, we developed and analyzed a splitting mixed finite element method for a stochastic Keller--Segel system with logistic growth driven by multiplicative noise. By introducing the auxiliary variable $\vsigma=\nabla v$ together with a time-lagged strategy, the proposed method decouples the original coupled system into a sequence of subproblems, thereby improving computational efficiency while allowing the use of continuous piecewise linear finite element spaces for all unknown variables.

From the theoretical viewpoint, we established stability estimates and derived rigorous error estimates for the proposed method. By combining a localization argument with suitable stochastic stability techniques, we proved convergence in probability together with explicit convergence rates. To the best of our knowledge, this is the first rigorous finite element error analysis for stochastic Keller--Segel equations driven by multiplicative noise.

The numerical experiments confirm the theoretical convergence rates and demonstrate the effectiveness of the proposed method. Moreover, they show that the numerical solutions successfully capture the global boundedness of the stochastic Keller--Segel system with logistic growth as well as the influence of multiplicative noise on chemotactic aggregation. Future work will focus on establishing strong convergence in $L^p(\Omega)$ with error estimates in full expectation and extending the proposed approach to more general stochastic chemotaxis models.
	
	%			\textbf{Acknowledgments.} 
	
	%\printbibliography[heading=none]
	\bibliographystyle{abbrv}
	\bibliography{references}

\end{document}